\documentclass[11pt,reqno]{amsart}
\usepackage[utf8]{inputenc}

\usepackage{amssymb, mathtools}
\usepackage{tikz-cd}
\usepackage{dsfont}        %
\usepackage{newtxtext}
\usepackage{newtxmath}      %
\usepackage{mathrsfs}      %
\usepackage{stmaryrd}      %

\usepackage{xcolor}
\definecolor{themegreen}{RGB}{0,92,54}   %
\definecolor{themeorange}{RGB}{191,87,0} %
\usepackage[final, colorlinks=true, linkcolor=themegreen, citecolor=themegreen, urlcolor=themegreen,pagebackref]{hyperref}
\usepackage{mathtools}
\usepackage{aliascnt}
\usepackage{cleveref}

\usepackage{tikz}

\usepackage{enumitem}
\usepackage[margin=28mm,a4paper]{geometry}
\usepackage[indent]{parskip}

\newcommand{\newsharedtheorem}[2]{%
  \newaliascnt{#1}{theorem}%
  \newtheorem{#1}[#1]{#2}%
  \aliascntresetthe{#1}%
  \crefname{#1}{#2}{#2s}%
  \Crefname{#1}{#2}{#2s}}

\theoremstyle{plain}
\newtheorem{theorem}{Theorem}[section]
\crefname{theorem}{Theorem}{Theorems}
\Crefname{theorem}{Theorem}{Theorems}
\newtheorem{introtheorem}{Theorem}

\crefname{introtheorem}{Theorem}{Theorems}
\Crefname{introtheorem}{Theorem}{Theorems}
\newsharedtheorem{proposition}{Proposition}
\newsharedtheorem{lemma}{Lemma}
\newsharedtheorem{corollary}{Corollary}
\newsharedtheorem{conjecture}{Conjecture}

\theoremstyle{definition}
\newsharedtheorem{definition}{Definition}
\newsharedtheorem{example}{Example}
\newsharedtheorem{remark}{Remark}
\newsharedtheorem{question}{Question}

\newcommand{\bk}{k}                           %
\newcommand{\PP}{\mathds{P}}
\newcommand{\CC}{\mathds{C}}
\newcommand{\ZZ}{\mathds{Z}}
\newcommand{\Fp}{\mathds{F}_p}
\newcommand{\Oh}{\mathcal{O}}
\newcommand{\cX}{\mathcal{X}}

\newcommand{\cY}{\mathcal{Y}}
\newcommand{\cK}{\mathcal{K}}
\newcommand{\cZ}{\mathcal{Z}}
\newcommand{\Wit}{W}                           %
\newcommand{\exactforms}[2]{B_{#1}\Omega^1_{#2}} %
\newcommand{\Frob}{F}                           %
\newcommand{\Cart}{\mathbf{M}}                  %
\DeclareMathOperator{\Spec}{Spec}
\DeclareMathOperator{\Hom}{Hom}
\DeclareMathOperator{\Ext}{Ext}

\DeclareMathOperator{\HH}{H}
\DeclareMathOperator{\Hilb}{Hilb}
\DeclareMathOperator{\coker}{coker}
\DeclareMathOperator{\hgt}{ht}                 %
\DeclareMathOperator{\crys}{crys}
\DeclareMathOperator{\Sym}{Sym}
\DeclareMathOperator{\Fil}{Fil}
\DeclareMathOperator{\gr}{gr}
\DeclareMathOperator{\Pic}{Pic}
\DeclareMathOperator{\Alb}{Alb}
\DeclareMathOperator{\Lie}{Lie}
\DeclareMathOperator{\length}{length}
\DeclareMathOperator{\rank}{rank}
\DeclareMathOperator{\qF}{qF}
\DeclareMathOperator{\HG}{HG}            %
\DeclareMathOperator{\Fs}{Fs}            %
\DeclareMathOperator{\fin}{fin}          %
\newcommand{\Kzero}{K_0}
\newcommand{\GGm}{\mathds{G}_m}
\newcommand{\Qp}{\mathds{Q}_p}
\newcommand{\QQ}{\mathds{Q}}
\newcommand{\slopedefect}{\delta}
\newcommand{\dR}{\mathrm{dR}}
\newcommand{\et}{\mathrm{\acute et}}

\numberwithin{equation}{section}

\begin{document}

\title{Quasi-$F$-split primitive symplectic varieties in positive characteristic}

\author{Haitao Zou}
\address{Fakult\"at f\"ur Mathematik, Universit\"at Bielefeld, 33615, Bielefeld,Germany}
\email{hzou@math.uni-bielefeld.de}

\date{September 28, 2026}

\subjclass[2020]{14J42, 14G17, 13A35}
\keywords{Primitive symplectic varieties, hyperk\"ahler varieties,
  quasi-$F$-splitting, Frobenius splitting, positive characteristic, Witt vectors} \thanks{The author is supported by the Deutsche Forschungsgemeinschaft (DFG, German Research Foundation), Project-ID 491392403, TRR 358.}

\begin{abstract}
Let $X$ be the good reduction of a projective hyperk\"ahler variety of dimension $2n\ge4$.
We prove that $X$ is quasi-$F$-split if and only if it is Frobenius split, equivalently if $\HH^2_{\crys}(X/\Wit)[1/p]$ has a slope-zero part.
Thus its quasi-$F$-split height is $1$ or $\infty$.
The proof combines a Verbitsky slope comparison with a Witt--Euler identity and requires no crystalline torsion-freeness.
The same dichotomy holds for primitive symplectic varieties in characteristic $p$, and Hodge-goodness is open in smooth proper families.
Hodge-deformations of Hilbert schemes $S^{[n]}$ of $K3$ surfaces ($p>n$) and generalised Kummer varieties $K_n(A)$ ($p>n+1$) remain primitive symplectic, with torsion-free crystalline cohomology and unobstructed mixed-characteristic formal deformations.
\end{abstract}

\maketitle

\setcounter{tocdepth}{1} %
\tableofcontents

\newpage
\section{Introduction}
\label{sec:introduction}

Frobenius splitting is a binary condition, while quasi-$F$-splitting assigns a height through the sheaves of Witt vectors.
For $K3$ surfaces this height recovers the formal Brauer height and takes intermediate finite values.
We ask what happens for higher-dimensional hyperk\"ahler varieties after reduction to characteristic $p$, and for their intrinsic characteristic-$p$ counterparts.

We reserve \emph{hyperk\"ahler} for characteristic zero.
Over a perfect field $\bk$ of characteristic $p$, a smooth proper variety $X$ of dimension $2n$ is \emph{primitive symplectic} if $\HH^0(X,\Omega_X^2)$ is spanned by a nowhere-degenerate $2$-form and if cup product makes $\HH^\bullet(X,\Oh_X)$ the algebra $\bk[\eta]/(\eta^{n+1})$, with $\deg\eta=2$ (\Cref{def:hodge-good,def:primitive-symplectic}).
We call the latter condition \emph{Hodge-goodness}.
The homotopical motivation for Hodge-goodness comes from the unipotent homotopy theory of Mondal--Reinecke: their unipotent homotopy type is encoded by the derived algebra $R\Gamma(X,\Oh_X)$, and their formal-sphere model in the Calabi--Yau case illustrates the simplicity we seek \cite[Remark~1.0.5 and Proposition~7.2.14]{MondalReinecke2026}.
Guided by the characteristic-zero hyperk\"ahler case, we seek comparably simple unipotent homotopy in characteristic $p$.
Hodge-goodness records the expected cup-product algebra; the derived algebra carries further information.
This intrinsic condition lets us study quasi-$F$-splitting without a characteristic-zero lift.
Our other setting is the good reduction of a hyperk\"ahler variety, where Hodge-goodness need not hold at every prime (\Cref{rem:HG-not-automatic}).

\subsection{Quasi-\texorpdfstring{$F$}{F}-splitting}
\label{subsec:quasi-f-theory}

Recall that a variety $X$ in characteristic $p>0$ is \emph{$F$-split} if the Frobenius map $\Oh_X\to\Frob_*\Oh_X$ admits an $\Oh_X$-linear retraction \cite{MehtaRamanathan1985}.
Yobuko's \emph{quasi-$F$-splitting} replaces this map by $\Phi_{X,n}\colon\Oh_X\to Q_{X,n}$, constructed from Frobenius and restriction on the length-$n$ Witt vectors $\Wit_n\Oh_X$ \cite{YobukoQuasiFSplit}.
If $\Phi_{X,n}$ retracts, then $X$ is \emph{$n$-quasi-$F$-split}; the least such $n$ is its \emph{quasi-$F$-split height} $\hgt(X)$, taken to be $\infty$ if no such $n$ exists.
Since $\Phi_{X,1}$ is the Frobenius map, $\hgt(X)=1$ is equivalent to $F$-splitting.

For Calabi--Yau varieties, $F$-splitting is equivalent to ordinarity, and Yobuko identifies $\hgt(X)$ with the height of the top Artin--Mazur formal group \cite{YobukoQuasiFSplit,ArtinMazur1977}.
Quasi-$F$-splitting also constrains the canonical class: if a normal projective $X$ is $n$-quasi-$F$-split, then Grothendieck duality gives a nonzero section of $\Oh_X\bigl(-(p^n-1)K_X\bigr)$, hence $\kappa(X,-K_X)\ge0$ and $\kappa(X)\le0$ \cite[Proposition~3.14]{KTTWYY-BirGeom}.
For $n=1$ this is the usual section of $\omega_X^{1-p}$ associated with an $F$-splitting.
The varieties studied here have $\omega_X\simeq\Oh_X$, so their height is governed instead by cohomology.

For smooth proper varieties, quasi-$F$-splitting implies degeneration of the Hodge--de Rham spectral sequence at $E_1$ (\Cref{lem:petrov-degeneration}); with geometric connectedness and trivial canonical bundle, it also yields a $\Wit_2(\bk)$-lift (\Cref{prop:quasi F split is W2 litable}).
A Fedder-type criterion computes the height from defining equations \cite{KTY-Fedder}, and quasi-$F$-splitting has applications in birational geometry \cite{KTTWYY-BirGeom}.

\subsection{From \texorpdfstring{$K3$}{K3} surfaces to higher dimensions}
\label{subsec:motivation}

The $K3$ and abelian cases give different patterns for quasi-$F$-split height.
For a $K3$ surface, $\hgt(X)$ is the height $h$ of its formal Brauer group \cite[Theorem~4.5]{YobukoQuasiFSplit}.
Every value in $\{1,2,\dots,10\}\cup\{\infty\}$ occurs: $X$ is $F$-split exactly when $h=1$ (the ordinary case), and quasi-$F$-split exactly when $h<\infty$.

For an abelian variety $A$ of dimension $g$ over $\bk$, with $p$-rank $f(A)=\dim_{\Fp}A[p](\overline{\bk})$, one has
\[
  \hgt(A)=
  \begin{cases}
    1,      & f(A)=g,\\
    2,      & f(A)=g-1,\\
    \infty, & f(A)\le g-2,
  \end{cases}
\]
by \cite[Theorem~3.2]{Yobuko2023}.
Thus only $1$, $2$ and $\infty$ occur: $A$ is quasi-$F$-split exactly when it is Hodge--Witt, and $F$-split exactly when it is ordinary \cite[Theorem~3.1]{Yobuko2023}.

For a good reduction of a hyperk\"ahler variety of dimension $2n\ge4$, which heights occur, and does degree-two cohomology govern them?
The next theorem gives a dichotomy sharper than either example.

\subsection{Main results}
\label{subsec:main-results}

Our first result concerns good reductions of hyperk\"ahler varieties.

\begin{introtheorem}\label{thm:main-all-good}
Let $\Oh_K$ be the valuation ring of a finite extension $K/\Qp$, and let $\mathscr X\to\Spec(\Oh_K)$ be a smooth proper morphism of algebraic spaces with projective generic fibre $Y$ whose base change $Y_{\overline K}$ is a hyperk\"ahler variety of dimension $2n\ge4$.
For the special fibre $X$ the following are equivalent:
\begin{enumerate}[label=\textup{(\roman*)},leftmargin=2.2em]
\item $X$ is quasi-$F$-split;
\item $X$ is Frobenius split;
\item $\HH^2_{\crys}(X/\Wit)[1/p]$ has a nonzero slope-zero part.
\end{enumerate}
Consequently $\hgt(X)\in\{1,\infty\}$.
\end{introtheorem}

Two calculations prove \Cref{thm:main-all-good}.
First, Verbitsky's theorem embeds $\Sym^i\HH^2$ into $\HH^{2i}$ for $i\le n$.
On a good reduction, weak admissibility shows that the crystalline cokernel has no Newton slopes below one (\Cref{prop:quotient,lem:WA}).
Thus degree-two cohomology controls the slopes in $[0,1)$ in every even degree; in particular, $\HH^{2n}_{\crys}$ has a slope-zero line exactly when $\HH^2_{\crys}$ does (\Cref{cor:automatic-top-slopes}).

Second, quasi-$F$-splitting makes each $\HH^j(X,\Wit\Oh_X)$ finitely generated over $\Wit$ (\Cref{lem:basic-properties}).
The Witt--Euler identity equates $\chi(X,\Oh_X)$ with the alternating sum of their Verschiebung indices; finite Witt torsion cancels in this sum (\Cref{prop:Witt-Euler}).
On a good reduction the sum is $n+1$, $2$, or $1$, according as the degree-two slopes below one form a slope-zero line, a nonordinary finite-height block, or no block at all (\Cref{prop:HK-Euler-slopes}).
Since $\chi(X,\Oh_X)=n+1\ge3$, only the first case can be quasi-$F$-split.
This argument assumes no integral torsion-freeness of crystalline cohomology.

The first theorem makes no Hodge-goodness assumption.
The coherent cohomology algebra of a complex hyperk\"ahler variety has the form in our definition, and this persists away from finitely many primes of a fixed spread, but it need not hold at every good reduction (\Cref{rem:HG-not-automatic}).
Indeed, the special fibre is simply connected (\Cref{prop:automatic-picard}), yet its Picard scheme may be nonreduced, giving $\HH^1(X,\Oh_X)\ne0$ (\Cref{rem:h1-not-automatic}).
The vanishing holds under small ramification or quasi-$F$-splitting (\Cref{prop:h1-vanishing-small-e,cor:qF-Picard}), while the standard Hilbert and Kummer families are Hodge-good in the stated characteristic ranges (\Cref{prop:hilbert-HG,prop:kummer-HG}).
Imposing Hodge-goodness intrinsically in characteristic $p$ yields a second dichotomy.

\begin{introtheorem}\label{thm:main-family}
A primitive symplectic variety $Z$ of dimension $2n\ge4$ over a perfect field is quasi-$F$-split if and only if it is Frobenius split; in particular $\hgt(Z)\in\{1,\infty\}$.
Moreover the Hodge-good locus is open in any smooth proper family with geometrically connected fibres of dimension $2n\ge4$ and trivial relative canonical sheaf, so the dichotomy holds near any Hodge-good fibre.
\end{introtheorem}

Only Hodge-goodness and $\omega_Z\simeq\Oh_Z$ enter the first assertion, which therefore holds for smooth proper geometrically connected algebraic spaces with those properties.
The proof uses the Cartier box product to identify $\Phi_Z^{2i}$ with $(\Phi_Z^2)^{\boxtimes i}$ and obtain the even Artin--Mazur height patterns $(1,\dots,1)$ and $(h,\infty,\dots,\infty)$ (\Cref{cor:boxheights}).
Openness follows from semicontinuity of coherent cohomology and constancy of $\chi(\cX_s,\Oh_{\cX_s})$ (\Cref{lem:HG-open}).
Hodge-goodness also forces $\HH^1(X,\Oh_X)=0$, so on a good reduction the top Witt-vector cohomology detects the dichotomy (\Cref{prop:top-witt-criterion}).

For a smooth proper family $f\colon\cX\to B$ as in \Cref{sec:quasiFspliting in family}, suppose one fibre is Hodge-good and $h^1(\cX_b,\Oh)=0$ on every fibre.
If $B$ is irreducible and its geometric generic fibre is not Frobenius split, no fibre is quasi-$F$-split (\Cref{thm:generic-dichotomy}).
If $B$ is also smooth of finite type, $n<p$, and $h^2(\cX_b,T_{\cX_b})$ is locally constant, every fibre has height $1$ or $\infty$ (\Cref{cor:dichotomy-all-of-S}).
In the latter setting, either every closed fibre lifts to $\Wit_2(\kappa(b))$ and all fibres are Hodge-good, or a closed fibre fails to lift and no fibre is quasi-$F$-split.

\newpage
The Hilbert schemes $S^{[n]}$ of $K3$ surfaces and the generalised Kummer varieties $K_n(A)$ give standard examples and their Hodge-deformations give further ones.
A \emph{Hodge-deformation family} over $\bk$ is a smooth proper morphism $g\colon\cY\to B$, with $B$ smooth, connected and of finite type over $\bk$, such that every Hodge number $h^{a,b}(\cY_t)$ is constant on the closed points $t\in B$.
Two smooth proper varieties are \emph{Hodge-deformation equivalent} if a finite chain of such families joins them.

\begin{introtheorem}\label{thm:deformation type of hilbert and kummer}
Let $n\ge2$, and let $X/\bk$ be a smooth proper variety Hodge-deformation equivalent to either
\begin{enumerate}[label=\textup{(\alph*)},leftmargin=2.2em]
\item $S^{[n]}$ for a $K3$ surface $S/\bk$, with $p>n$; or
\item $K_n(A)$ for an abelian surface $A/\bk$, with $p>n+1$.
\end{enumerate}
Then $X$ is primitive symplectic, its Hodge--de Rham spectral sequence degenerates at $E_1$, and $\HH^j_{\crys}(X/\Wit)$ is torsion-free for every $j$.
Its mixed-characteristic formal deformations are unobstructed; in particular, $X$ admits a smooth proper formal lifting over $\Wit$ and a smooth proper lifting over $\Wit_2$.
For every $1\le i\le n$, the Artin--Mazur functor $\Phi_X^{2i}$ is prorepresentable by a smooth one-dimensional formal group over $\bk$, and
\[
  X\text{ is quasi-}F\text{-split}
  \quad\Longleftrightarrow\quad X\text{ is }F\text{-split}
  \quad\Longleftrightarrow\quad \hgt(\Phi_X^2)=1.
\]
If $p>2n$, the Hodge-deformation hypothesis may be replaced by a finite chain of smooth proper families joining $X$ to the same standard model, over smooth connected finite-type bases, for which $b\mapsto h^2(\cY_b,T_{\cY_b})$ is locally constant on closed points (condition \eqref{eq:ct2}).
\end{introtheorem}

The proof propagates an integral Beauville--Bogomolov--Fujiki pairing through each family, keeping the coherent cup powers and the symplectic form nonzero (\Cref{lem:fujiki-propagation}).
The symplectic form identifies $T_X$ with $\Omega_X^1$, so condition \eqref{eq:ct2} becomes constancy of $h^{1,2}$ on these families.

The paper is organised as follows.
\Cref{sec:preliminaries,sec:hyperkahler} introduce quasi-$F$-splitting and good reduction.
\Cref{sec:slopes} develops the crystalline slope comparisons and computes Artin--Mazur heights under Hodge-goodness.
\Cref{sec:main-results} applies the Witt--Euler identity to prove \Cref{thm:main-all-good}.
\Cref{sec:quasiFspliting in family} treats Hodge-goodness and heights in families.
\Cref{sec:examples} proves \Cref{thm:deformation type of hilbert and kummer}; Appendix \ref{sec:appendix} supplies the torsion-freeness results and integral Fujiki propagation lemma used there.

\subsection{Conventions}
\label{subsec:conventions}

A \emph{variety} is an integral, separated scheme of finite type over a field.
Good-reduction models and their special fibres are allowed to be algebraic spaces; the generic hyperk\"ahler fibre remains a projective variety.
Unless stated otherwise $\bk$ is a perfect field of characteristic $p>0$; all the properties we consider are insensitive to extension of the perfect base field, so we pass freely to $\overline{\bk}$ when convenient.
We write $\Frob\colon X\to X$ for the absolute Frobenius, $\Wit=\Wit(\bk)$ for the ring of Witt vectors, $\Wit_n=\Wit_n(\bk)$ for its length-$n$ truncations, and $\Kzero=\Wit[1/p]$.
For a smooth proper $X/\bk$ we abbreviate
\[
  D_q \;=\; \HH^q_{\crys}(X/\Wit)\otimes_{\Wit}\Kzero
      \;=\; \HH^q_{\crys}(X/\Wit)[1/p],
\]
an isocrystal for the crystalline Frobenius induced by the absolute $p$-power map.
Slopes are normalised by $v_p(p)=1$, so that a divisor class in $\HH^2_{\crys}$ has slope $1$, and $(D_q)_{[0,1)}\subseteq D_q$ denotes the sum of the slope subspaces with slope in $[0,1)$.
Coherent cohomology and the sheaves of Witt vectors are taken on the small \'etale site.
For schemes, coherent and finite Witt-vector cohomology agree with their Zariski counterparts; infinite Witt-vector cohomology is the inverse limit of the finite-level groups.

\subsection*{Acknowledgements and declarations}

This project began at a workshop organized by Zhiyuan Li at SYSU (Zhuhai) in 2023, where the author was asked to give a series of lectures on the theory of quasi-$F$-splitting of \cite{YobukoQuasiFSplit} and to investigate the hyperk\"ahler case.
Our expectation at the time was that for examples such as $S^{[n]}$ and generalised Kummer varieties, quasi-$F$-splitting would characterise finiteness of the height in degree two, as it does for $K3$ surfaces.
One day before the author's talk, however, the preprint \cite{Yobuko2023} appeared and showed this naive expectation to be false, already for higher-dimensional Hilbert schemes.
The aim then became to understand the phenomenon in general.

The first ideas behind \Cref{sec:slopes} and \Cref{sec:quasiFspliting in family} date from shortly after that workshop, but at the time nothing substantial could be proved without assuming torsion-freeness of crystalline cohomology and degeneration of the Hodge--de Rham spectral sequence.
In December 2024, Fuetaro Yobuko invited the author to visit him in Japan; the discussions there led to the notion of Hodge-goodness used in this paper.
During this visiting, we also completed the computation of the quasi-$F$-split heights of generalised Kummer varieties together, by the methods of \cite{Yobuko2023}, but could not extend it to their deformation types.
The paper presented here doesn't include this computation, but the ingredients from Yobuko help a lot.

In September 2026, the author put the question to current AI tools.
After several prompts, they suggested that the Witt--Euler identity of \Cref{prop:Witt-Euler} could settle it for good reductions of hyperk\"ahler varieties, with no hypothesis on the torsion-freeness of crystalline cohomology.
The same tools were used in proving the technical results in \Cref{sec:quasiFspliting in family}, for example \Cref{lem:relative-W2-lifting} and \Cref{prop:HG-locally-constant}. Other works were finished before these prompts.

The torsion-free results for crystalline cohomology of generalized Kummer in Appendix \ref{sec:appendix} were produced with the help of ChatGPT Sol 5.6 in August 2026, during the summer school on algebraic geometry in SCMS, Shanghai. The author claims no credits on these results and records them here for the math community useage.

Codex and Claude Code were also used for copy-editing, proofreading and reference searching.

\newpage
\section{Preliminaries on quasi-\texorpdfstring{$F$}{F}-splitting}
\label{sec:preliminaries}

We collect the definitions and results on Frobenius splitting, quasi-$F$-splitting, and Artin--Mazur formal groups used below.
Unless specified otherwise, $X$ is a smooth proper algebraic space over $\bk$.

\subsection{Frobenius splitting}
\label{subsec:frobenius-splitting}

\begin{definition}
\label{def:f-split}
$X$ is \emph{$F$-split} if the natural map $\Oh_X \to \Frob_* \Oh_X$ splits as a morphism of $\Oh_X$-modules.
\end{definition}

For a trivial canonical bundle, Frobenius splitting is detected on top coherent cohomology \cite[Lemma~2.11(1)]{KTTWYY-BirGeom}.

\begin{lemma}
\label{lem:Fsplit}
Let $Z$ be a smooth proper geometrically connected algebraic space of dimension $d$ with $\omega_Z\simeq\Oh_Z$ over a field $\bk$ of characteristic $p$ that is $F$-finite, that is $[\bk:\bk^p]<\infty$.
Then $Z$ is Frobenius split if and only if the absolute Frobenius acts nontrivially on $\HH^d(Z,\Oh_Z)$.
\end{lemma}

\begin{proof}
The $F$-finiteness of $\bk$ makes the absolute Frobenius of $Z$ finite.
Grothendieck and Serre duality for proper algebraic spaces \cite[Tags~0E58 and~0E61]{stacks-project} identify the restriction map
\[
 \Hom_{\Oh_Z}(\Frob_*\Oh_Z,\Oh_Z)\longrightarrow
 \Hom_{\Oh_Z}(\Oh_Z,\Oh_Z)=\bk
\]
with the dual of Frobenius on $\HH^d(Z,\Oh_Z)$; here the dualizing complex is $\omega_Z[d]$.
The map $\Oh_Z\to\Frob_*\Oh_Z$ splits precisely when $1\in\bk$ has a preimage under this restriction map.
Since its target is one-dimensional, this is equivalent to the map, and hence Frobenius on $\HH^d(Z,\Oh_Z)$, being nonzero.
\end{proof}

\subsection{Witt vectors and quasi-\texorpdfstring{$F$}{F}-splitting}
\label{subsec:witt}

Let $\Wit_n\Oh_X$ be the sheaf of length-$n$ Witt vectors, with Frobenius $F$, Verschiebung $V$, and restriction $R$.
Yobuko defines $Q_{X,n}$ as the pushout of $F$ and $R_{n-1}$ in the category of $\Wit_n\Oh_X$-modules:
\[
\begin{tikzcd}
\Wit_n \Oh_X \ar[r,"F"] \ar[d,two heads,"R_{n-1}"] & F_* \Wit_n \Oh_X \\
\Oh_X
\end{tikzcd}
\]
The pushout gives a canonical map $\Phi_{X,n}\colon\Oh_X\to Q_{X,n}$.
The ideal $V\Wit_{n-1}\Oh_X$ acts trivially on $Q_{X,n}$, so its $\Wit_n\Oh_X$-module structure factors through $\Wit_n\Oh_X/V\Wit_{n-1}\Oh_X\simeq\Oh_X$ \cite[Proposition~2.9(2)]{KTTWYY-BirGeom}.
For $n=1$, $Q_{X,1}=\Frob_*\Oh_X$ and $\Phi_{X,1}$ is the map of \Cref{def:f-split}.

Witt restriction induces maps $R\colon Q_{X,n+1}\to Q_{X,n}$ satisfying $R\circ\Phi_{X,n+1}=\Phi_{X,n}$; see \cite[\S3]{YobukoQuasiFSplit} and \cite[\S3]{KTTWYY-BirGeom}.

\begin{definition}
\label{def:quasi-f-split}
$X$ is \emph{$n$-quasi-$F$-split} if $\Phi_{X,n}$ admits an $\Oh_X$-linear retraction.
The \emph{quasi-$F$-split height} (or \emph{Yobuko height}) is
\[
  \hgt(X) \;=\; \min\{\, n \ge 1 \mid X \text{ is } n\text{-quasi-}F\text{-split} \,\}
  \;\in\; \ZZ_{\ge 1} \cup \{\infty\},
\]
with $\hgt(X)=\infty$ if no such $n$ exists; $X$ is \emph{quasi-$F$-split} if $\hgt(X)<\infty$.
\footnote{The same invariant is denoted $h_F(X)$ elsewhere in the literature.}
In particular, $X$ is $F$-split if and only if $\hgt(X)=1$.
\end{definition}

\begin{lemma}
\label{lem:basic-properties}
Let $X$ be a smooth proper algebraic space over $\bk$.
\begin{enumerate}
\item If $X$ is $n$-quasi-$F$-split then it is $(n+1)$-quasi-$F$-split; in particular $F$-split $\Rightarrow$ quasi-$F$-split.
\item If $X$ is quasi-$F$-split, then $\HH^j(X,\Wit\Oh_X)\coloneqq\varprojlim_r \HH^j(X,\Wit_r\Oh_X)$ is a finitely generated $\Wit$-module for every $j$.
\end{enumerate}
\end{lemma}

\begin{proof}
(1) A retraction $\sigma$ of $\Phi_{X,n}$ gives a retraction $\sigma\circ R$ of $\Phi_{X,n+1}$, since $R\circ\Phi_{X,n+1}=\Phi_{X,n}$; the case $n=1$ gives the last assertion.

(2) For schemes this is \cite[Theorem~1.2]{Nakkajima2022}; we follow its proof on the \'etale site to extend it to algebraic spaces.
Let $\exactforms{r}{X}$ be the cokernel of $F\colon\Wit_r\Oh_X\to F_*\Wit_r\Oh_X$.
The Cartier operator and the pushout defining $Q_{X,r}$ give exact sequences
\[
\begin{aligned}
  0&\longrightarrow\Oh_X\longrightarrow Q_{X,r}
    \longrightarrow\exactforms{r}{X}\longrightarrow0,\\
  0&\longrightarrow F_*\exactforms{r-1}{X}\longrightarrow Q_{X,r}
    \longrightarrow F_*\Oh_X\longrightarrow0\qquad(r\ge2).
\end{aligned}
\]
They hold on $X$ because they can be checked on \'etale scheme charts.
For $h=\hgt(X)<\infty$, the first sequence splits when $r\ge h$ by (1); the second then gives
\[
  \dim_\bk\HH^j(X,\exactforms{r}{X})
  \le\dim_\bk\HH^j(X,\exactforms{r-1}{X})
  \qquad(r\ge\max\{h,2\}).
\]
Thus these dimensions are bounded.
Finite-level Witt cohomology has finite length and satisfies the Mittag--Leffler condition \cite[Proposition~4.5.2]{Olsson2007Crystalline}.
Serre's finite-generation criterion \cite[arXiv version~2, Remark~4.5]{Nakkajima2022}, applied to the inverse limit of $\HH^j(X,\exactforms{r}{X})$, now yields the claim.
\end{proof}

\begin{remark}
\label{rem:F-finite}
The criterion of \Cref{lem:Fsplit}, the definition of \Cref{def:quasi-f-split}, and the implication in \Cref{lem:basic-properties}(1) apply over any $F$-finite field, without assuming perfection.
Every residue field of a scheme of finite type over the perfect field $\bk$ is $F$-finite.
The Witt-cohomology and Cartier-module results below are stated over perfect fields, for which $\Wit(\bk)$ is a complete discrete valuation ring.
\end{remark}

Petrov's theorem connects quasi-$F$-splitting to Hodge--de Rham degeneration.
For a smooth proper variety $X/\bk$, consider the spectral sequence
\[
E_1^{p,q} = \HH^q(X,\Omega_{X}^p) \Longrightarrow \HH_{\dR}^{p+q}(X).
\]

\begin{lemma}[Petrov]
\label{lem:petrov-degeneration}
If $X$ is quasi-$F$-split, then the Hodge--de Rham spectral sequence degenerates at the first page.
\end{lemma}

\begin{proof}
This is \cite[Theorem~1.1]{Petrov2025}, which gives a decomposition of the de Rham complex of a quasi-$F$-split smooth proper variety.
\end{proof}

Together with a lifting criterion, Petrov's theorem gives the following $\Wit_2$-liftability statement for varieties with trivial canonical bundle.

\begin{proposition}
  \label{prop:quasi F split is W2 litable}
Let $X$ be a smooth proper geometrically connected variety over $\bk$ with $\omega_{X}\simeq\Oh_X$.
If $X$ is quasi-$F$-split, then $X$ is $\Wit_2(\bk)$-liftable.
\end{proposition}

\begin{proof}
By \Cref{lem:petrov-degeneration}, quasi-$F$-splitting makes the Hodge--de Rham spectral sequence degenerate at $E_1$.
For a smooth proper geometrically connected variety with trivial canonical bundle, this implies $\Wit_2(\bk)$-liftability by \cite[Theorem~7.18]{BrantnerTaelman2025}.
The input is \cite[Theorem~1.3]{AchingerSuh2023}: for $d=\dim X$, the conjugate differential
\[
  d_2^{d-2,1}\colon\HH^{d-2}(X^{(p)},\Omega^1_{X^{(p)}/\bk})
  \longrightarrow\HH^d(X^{(p)},\Oh_{X^{(p)}})
\]
is cup product with the obstruction to lifting $X^{(p)}$, and Serre duality with a volume form equates their vanishing.
\end{proof}

\subsection{Artin--Mazur formal group}
\label{subsec:Artin-Mazur formal group}

For a Calabi--Yau variety, quasi-$F$-split height agrees with the height of its top Artin--Mazur formal group.
We recall the functor, its representability criterion, and the comparisons used later.

Write $\mathrm{Art}_\bk$ for the category of local Artinian $\bk$-algebras with residue field $\bk$, and put $X_A=X\times_{\Spec\bk}\Spec A$ for $A\in\mathrm{Art}_\bk$.

\begin{definition}
\label{def:artin-mazur}
For $q\ge1$, the \emph{Artin--Mazur functor} of $X$ in degree $q$ is
\[
  \Phi^q_X\colon\mathrm{Art}_\bk\longrightarrow(\mathrm{Ab}),
  \qquad
  \Phi^q_X(A)=\ker\bigl(\HH^q_{\et}(X_A,\GGm)\longrightarrow\HH^q_{\et}(X,\GGm)\bigr),
\]
the kernel of restriction to the closed fibre \cite[\S II]{ArtinMazur1977}.
For $q=1$ this is the formal Picard functor $\widehat{\Pic}_X$, and for $q=2$ the formal Brauer functor $\widehat{\mathrm{Br}}_X$.
\end{definition}

\begin{lemma}[Artin--Mazur]
\label{lem:AM-prorepresentable}
Let $X$ be a smooth proper algebraic space over $\bk$ of dimension $d$, and let $q\ge1$.
\begin{enumerate}[label=\textup{(\arabic*)},leftmargin=2.2em]
\item The tangent space of $\Phi^q_X$ is $\HH^q(X,\Oh_X)$, and obstructions to lifting along a square-zero extension in $\mathrm{Art}_\bk$ lie in $\HH^{q+1}(X,\Oh_X)$.
\item If $\HH^{q-1}(X,\Oh_X)=0$, then $\Phi^q_X$ is prorepresentable.
\item If in addition $\HH^{q+1}(X,\Oh_X)=0$, which is automatic for $q=d$, then $\Phi^q_X$ is formally smooth; if moreover $\dim_\bk\HH^q(X,\Oh_X)=1$, then $\Phi^q_X$ is a smooth one-dimensional formal group.
\end{enumerate}
\end{lemma}

\begin{proof}
These are the criteria of \cite[\S II, Corollaries~2.4--2.5 and 4.2--4.4]{ArtinMazur1977}.
The proofs apply on the \'etale site of a proper algebraic space: nilpotent thickenings preserve the \'etale topos, coherent cohomology is finite-dimensional, and smoothness gives vanishing above degree $d$.
For (1), a square-zero extension $A\twoheadrightarrow A/I$ in $\mathrm{Art}_\bk$ gives the exact sequence
\[
  1\longrightarrow 1+I\Oh_X\longrightarrow\Oh_{X_A}^{\times}
   \longrightarrow\Oh_{X_{A/I}}^{\times}\longrightarrow1,
  \qquad 1+I\Oh_X\simeq I\otimes_\bk\Oh_X ,
\]
Its cohomology sequence identifies the kernel of $\Phi^q_X(A)\to\Phi^q_X(A/I)$ as a quotient of $I\otimes_\bk\HH^q(X,\Oh_X)$ and places the obstruction to surjectivity in $I\otimes_\bk\HH^{q+1}(X,\Oh_X)$.
Taking $A=\bk[\varepsilon]$ gives the tangent space; for $q=d$ the obstruction group vanishes by cohomological dimension.
\end{proof}

Heights can be read off from Witt-vector cohomology.
For a commutative formal Lie group $E$ over $\bk$ write
\[
  \Cart(E)=\Hom(\widehat{\Wit},E)
\]
for its covariant ($p$-typical) Cartier module, $\widehat{\Wit}$ being the formal completion of the Witt group at the origin.

\begin{lemma}
\label{prop:AM-cartier}
If $\Phi^q_X$ is prorepresentable, there is a canonical isomorphism of Cartier modules
\[
  \Cart(\Phi^q_X)\;\simeq\;\HH^q(X,\Wit\Oh_X)
  \coloneqq\varprojlim_r\HH^q(X,\Wit_r\Oh_X).
\]
\end{lemma}

\begin{proof}
This is \cite[\S II, Corollary~4.3]{ArtinMazur1977}; see \cite[Proposition~6.3.3]{MondalReinecke2026} for a modern account.
The computation uses the multiplicative formal group and the Witt sheaves on the \'etale site, so it applies equally to smooth proper algebraic spaces.
\end{proof}

When $\Phi^q_X$ is a smooth one-dimensional formal group, write $\hgt(\Phi^q_X)\in\ZZ_{\ge1}\cup\{\infty\}$ for its height.
At finite height $h$, its Cartier module is free of rank $h$ over $\Wit$.
At infinite height the module is annihilated by $p$, and the formal group becomes $\widehat{\mathds{G}}_a$ over $\overline{\bk}$.
Thus \Cref{prop:AM-cartier} reads the height from Witt-vector cohomology, and rationally from crystalline slopes.

\begin{remark}
\label{rem:AM-slope-convention}
If $\Phi^q_X$ is a smooth one-dimensional formal group of finite height $h$, then $(D_q)_{[0,1)}$ is isoclinic of slope $1-\tfrac1h$ and dimension $h$; if its height is infinite, then $(D_q)_{[0,1)}=0$ (\Cref{prop:AM-cartier,prop:WittSlope}).
Here Frobenius is the crystalline Frobenius acting on the covariant Cartier module $\HH^q(X,\Wit\Oh_X)$.
\end{remark}

The link with \Cref{def:quasi-f-split} is Yobuko's theorem.

\begin{theorem}[Yobuko]
\label{thm:yobuko-height}
Let $X$ be a Calabi--Yau variety of dimension $d\ge2$ over $\bk$, that is, smooth proper with $\omega_X\simeq\Oh_X$ and $\HH^i(X,\Oh_X)=0$ for $0<i<d$.
Then $\Phi^d_X$ is a smooth one-dimensional formal group and
\[
  \hgt(X)=\hgt(\Phi^d_X).
\]
\end{theorem}

\begin{proof}
The hypotheses give $\HH^{d-1}(X,\Oh_X)=0$, and Serre duality together with $\omega_X\simeq\Oh_X$ gives $\HH^d(X,\Oh_X)\simeq\HH^0(X,\Oh_X)^\vee=\bk$, so $\Phi^d_X$ is a smooth one-dimensional formal group by \Cref{lem:AM-prorepresentable}.
The equality of heights is \cite[Theorem~4.5]{YobukoQuasiFSplit}.
\end{proof}

In dimension $2n\ge4$, good reductions of hyperk\"ahler varieties are not Calabi--Yau in this sense: semicontinuity gives $\HH^{2i}(X,\Oh_X)\ne0$ for $0<i<n$.
The following inequality gives a partial comparison under explicit cohomological hypotheses.
\begin{theorem}[Nakkajima]
\label{thm:nakkajima-inequality}
Let $X$ be a smooth proper variety over $\bk$, and let $q\ge1$.
Assume that $\Phi^q_X$ is prorepresentable, that
\[
  \HH^q(X,\Oh_X)=\bk,\qquad \HH^{q+1}(X,\Oh_X)=0,
\]
and that the Bockstein maps $\beta_m\colon\HH^{q-1}(X,\Oh_X)\to\HH^q(X,\Wit_{m-1}\Oh_X)$ vanish for every $m\ge2$.
Then
\[
  \hgt(\Phi^q_X)\;\le\;\hgt(X).
\]
In particular $\hgt(\Phi^q_X)=\infty$ forces $\hgt(X)=\infty$.
\end{theorem}

\begin{proof}
This is \cite[Theorem~1.5]{Nakkajima2022}, applied to $X$ with the trivial log structure, which is log smooth of Cartier type over $\bk$.
\end{proof}

\begin{remark}
\label{rem:AM-infinite-height}
If $\Phi^q_X$ is prorepresentable by a smooth one-dimensional formal group of infinite height, its Cartier module $\HH^q(X,\Wit\Oh_X)$ is not finitely generated over $\Wit$ (\Cref{prop:AM-cartier}).
Thus $X$ is not quasi-$F$-split by \Cref{lem:basic-properties}(2), without any Bockstein vanishing hypothesis.
\end{remark}

\begin{remark}
\label{rem:AM-recent}
Without the vanishing in \Cref{lem:AM-prorepresentable}(2), the criterion gives no automatic prorepresentability of the classical Artin--Mazur functor.
For this reason, \Cref{def:crystalline-height} defines the degree-two invariant using crystalline cohomology.
\end{remark}

\section{Good reductions of hyperk\"ahler varieties}
\label{sec:hyperkahler}

We fix the arithmetic setting for good reductions of hyperk\"ahler varieties and record properties of their special fibres.
The intrinsic characteristic-$p$ notion of a primitive symplectic variety is introduced in \Cref{def:primitive-symplectic}.

\subsection{Good reduction over a number field}
\label{subsec:good-reduction}

Let $K$ be a number field with ring of integers $\Oh_K$, and fix a prime $\mathfrak{p}\subset\Oh_K$ above $p$.
By a \emph{hyperk\"ahler variety} over $K$ we mean a smooth projective $K$-variety $X_K$ whose base change to $\overline K$ is simply connected and has $\HH^0(X_{\overline K},\Omega^2)$ spanned by a nowhere-degenerate closed $2$-form.

\begin{definition}
\label{def:good-reduction}
We say $X_K$ has \emph{good reduction} at $\mathfrak{p}$ if there is a smooth proper morphism of algebraic spaces $\cX \to \Spec \Oh_{K,\mathfrak{p}}$ with generic fibre $\cX_K \cong X_K$.
The \emph{reduction} associated with this model is the special fibre
\[
  X \;=\; \cX \times_{\Oh_{K,\mathfrak{p}}} \bk ,
  \qquad \bk=\Oh_K/\mathfrak p,
\]
a smooth proper algebraic space over $\bk$ of dimension $2n$.
The total space $\cX$ is regular; neither $\cX$ nor $X$ is required to be a scheme or to be projective over its base.
\end{definition}

\begin{example}
\label{ex:k3}
A $K3$ surface is a hyperk\"ahler variety of dimension $2$, and its good reduction is again a $K3$ surface.
Higher-dimensional examples include Hilbert schemes of points on $K3$ surfaces and generalised Kummer varieties; we return to their reductions in \Cref{ex:hilb,ex:kummer}.
\end{example}

\subsection{Local setting and automatic invariants}
\label{subsec:automatic}

For the cohomological arguments we complete at $\mathfrak p$ and retain the notation $K$ for the resulting finite extension of $\Qp$.
Its valuation ring is $\Oh_K$ and its residue field $\bk$ is perfect of characteristic $p$.
Put $\Wit=\Wit(\bk)$ and $\Kzero=\Wit[1/p]$, and let
\[
  \mathscr{X} \longrightarrow\Spec(\Oh_K)
\]
be a smooth proper morphism of algebraic spaces, with projective generic fibre $Y/K$ and special fibre $X/\bk$.
We assume throughout that $Y_{\overline K}$ is an irreducible hyperk\"ahler variety of dimension $2n$.
Recall that the structure-sheaf Hodge numbers of $Y$ are thus $1$ in each even degree and $0$ in odd degrees, so that $\chi(Y,\Oh_Y)=n+1$.

We begin with properties of $X$ that require neither Hodge-goodness nor a ramification or characteristic bound.

\begin{proposition}
\label{prop:canonical-good}
In the setting above, $X$ is geometrically connected and geometrically integral, and
\[
  \omega_X\simeq\Oh_X,\qquad\chi(X,\Oh_X)=n+1 .
\]
The same assertions hold after any finite extension of the residue field.
\end{proposition}

\begin{proof}
The number of geometric connected components is locally constant in a smooth proper family of algebraic spaces \cite[Tag~0E1E]{stacks-project}; hence $X$ is geometrically connected, and smoothness makes it geometrically integral.
Hilbert's Theorem 90 descends the trivialisation of $\omega_{Y_{\overline K}}$ to $Y$.
Viewed as a rational section of $\omega_{\mathscr X/\Oh_K}$, it has divisor $aX$ for some $a\in\ZZ$, since $X$ is the only vertical prime divisor.
The divisor $X$ is principal, cut out by a uniformiser, so rescaling gives a nowhere-vanishing relative canonical section and hence $\omega_X\simeq\Oh_X$; the divisor argument is checked on \'etale charts of $\mathscr X$.
Finally, proper flat cohomology and base change give a perfect complex for the two fibres \cite[Tag~0CTM]{stacks-project}, so $\chi(X,\Oh_X)=\chi(Y,\Oh_Y)=n+1$.
\end{proof}

\begin{proposition}
\label{prop:automatic-picard}
After extension to an algebraic closure of the residue field, every good reduction $X$ satisfies
\[
  \pi_1^{\et}(X)=1,\qquad \Alb_{X/\bk}=0,\qquad
  (\Pic^0_{X/\bk})_{\mathrm{red}}=0 .
\]
Moreover, the absolute Frobenius acts nilpotently on $\HH^1(X,\Oh_X)$.
\end{proposition}

\begin{proof}
Work over an algebraic closure of $\bk$.
After passing to a complete strictly henselian trait, finite \'etale covers of $X$ extend to the proper model; a connected cover has geometrically connected generic fibre by smooth proper constancy.
Thus $\pi_1^{\et}(Y_{\overline K})\twoheadrightarrow\pi_1^{\et}(X)$, and the source is trivial.
The Albanese variety is therefore zero, since a nonzero abelian variety has a nonzero prime-to-$p$ Tate module detected by $\pi_1^{\et}(X)$; its dual $(\Pic^0_{X/\bk})_{\mathrm{red}}$ is zero as well.

Put $V=\HH^1(X,\Oh_X)$.
The Artin--Schreier sequence gives $\ker(\Frob-1\colon V\to V)=\HH^1_{\et}(X,\Fp)=0$.
The bijective part of a $p$-semilinear operator over an algebraically closed field has fixed vectors by Lang's theorem, so it must vanish here.
Hence $\Frob$ is nilpotent on $V$.
\end{proof}

\begin{remark}
\label{rem:h1-not-automatic}
The vanishing $\HH^1(X,\Oh_X)=0$ is \emph{not} a formal consequence of \Cref{prop:automatic-picard}: a finite connected Picard scheme in characteristic $p$ may have nonzero tangent space while its reduced subscheme is a point.
This is why the dichotomy below is proved through crystalline slopes rather than through the coherent cohomology algebra.
\end{remark}

The additional vanishing of $\HH^1(X,\Oh_X)$ follows when the ramification index is at most $p-1$.
\begin{proposition}
\label{prop:h1-vanishing-small-e}
Let $X$ be a good reduction of a hyperk\"ahler variety as above.
Suppose the absolute ramification index of $K$ is at most $p-1$.
Then $\HH^1(X,\Oh_X)=0$.
\end{proposition}
\begin{proof}
After a faithfully flat base change we may assume that $k$ is algebraically closed.
This preserves the absolute ramification index of $K$ and the desired vanishing of coherent cohomology descends along the base change.
Let $P^{\tau} = \Pic^{\tau}_{\mathscr X/\Oh_K}$.
Since $\mathscr{X}/\Oh_K$ is smooth proper with geometrically connected fibres, it is cohomologically flat in degree $0$.
Hence $P^{\tau}$ is an algebraic space of finite type whose formation commutes with base change.
Since the fibers of $\mathscr{X}/\Oh_K$ are geometrically normal, $P^{\tau}$ is separated and equidimensional over $\Oh_K$ \cite[Theorem~3.6(i),(iii)]{FringuelliViviani2023}.

Let $S = \Spec(\Oh_K)$ and $z\colon S\to P^{\tau}$ be the zero section.
Raynaud's Picard-flatness theorem \cite[Theorem~4.1.2]{Raynaud1979Picard}, in the algebraic-stack formulation of \cite[Theorem~8.1(ii)]{FringuelliViviani2023}, implies that $P^{\tau}$ is flat along $z$ whenever the absolute ramification index is at most $p-1$.
Thus there is a Zariski open neighbourhood $U$ of $z(S)$ in $P^{\tau}$ such that $U\to S$ is flat.
The generic fibre of $P^{\tau}$ is trivial, since it is the $\tau$-part of the Picard scheme of a hyperk\"ahler variety in characteristic zero; hence $U_K=z(K)$.
Since $P^{\tau}$ is separated over $S$, $z(S)$ is closed in $U$.
Its ideal sheaf vanishes on $U_K$, and therefore vanishes on $U$ by flatness over the discrete valuation ring.
Consequently $U=z(S)$, so $z(k)$ is an open reduced point of $P^{\tau}_k$ and
\[
  \HH^1(X,\Oh_X) = \mathrm{T}_{z(k)} P^{\tau}_k = 0. \qedhere
\]
\end{proof}

\section{Crystalline slopes and formal-group heights}
\label{sec:slopes}

We keep the arithmetic setting of \Cref{subsec:automatic}.
The crystalline Frobenius makes each $D_q=\HH^q_{\crys}(X/\Wit)[1/p]$ an isocrystal.
For a good reduction the Galois representation $\HH^q_{\et}(Y_{\overline K},\Qp)$ is crystalline with $D_{\crys}(\HH^q_{\et}(Y_{\overline K},\Qp))=D_q$; in particular $D_q$ underlies a weakly admissible filtered $\varphi$-module, whose filtration comes from the Hodge filtration on $\HH^q_{\dR}(Y/K)$ \cite[\S\S7.3,\,9.1]{BrinonConrad2009}, \cite[Theorems~1.1(i),\,1.10]{BMS2018}.

For the algebraic-space model used here, this is the smooth case of Olsson's comparison for proper tame Deligne--Mumford stacks with schematic generic fibre \cite[\S6.4 and Theorem~9.6.9]{Olsson2007Crystalline}: the special-fibre object is ordinary crystalline cohomology and the monodromy operator is zero.
The comparison respects cup products and identifies the de Rham filtration with the Hodge filtration, so the filtered quotients used below have the same interpretation as for a scheme model.
All Hodge filtrations are decreasing and effective, so $\Fil^0D_q=D_q$.

We first determine the slopes of $D_2$ below one, then use the Verbitsky component to control the slopes in higher even degrees.
Under Hodge-goodness, Cartier theory gives a separate description of the Artin--Mazur heights.

\subsection{Degree-two slopes and crystalline height}
\label{subsec: weak admissibility and slopes}
For a filtered $\varphi$-module $M$ set
\[
  t_H(M)=\sum_a a\dim\gr^a_{\Fil}M_K,
  \qquad
  t_N(M)=\sum_\lambda\lambda\dim M_\lambda
\]
for its Hodge degree and Newton degree respectively.
Weak admissibility means that $t_H(M)=t_N(M)$ and that $t_H(M')\le t_N(M')$ for every $\varphi$-stable subobject $M'$, equipped with the induced filtration \cite[\S8.2]{BrinonConrad2009}.

For an isocrystal $M$, define its \emph{slope defect below one} by
\[
  \slopedefect(M)=\sum_{\lambda<1}(1-\lambda)\dim_{\Kzero}M_\lambda
    =\dim_{\Kzero}M_{<1}-t_N(M_{<1}),
\]
where $M_{<1}$ is its slope subisocrystal with slopes below one.
In particular, $\slopedefect(M)=\slopedefect(M_{<1})$.

\begin{lemma}
\label{lem:WA}
Let $Q$ be a weakly admissible filtered $\varphi$-module with $\Fil^0Q_K=\Fil^1Q_K=Q_K$.
Then every Newton slope of $Q$ is at least $1$.
\end{lemma}

\begin{proof}
Suppose the slope subisocrystal $Q_{<1}$ spanned by the slopes below one were nonzero.
It is $\varphi$-stable, and we give it the induced filtration.
Since $\Fil^1Q_K=Q_K$, the induced filtration on $(Q_{<1})_K$ also satisfies $\Fil^1=(Q_{<1})_K$, whence $t_H(Q_{<1})\ge\dim Q_{<1}$.
On the other hand every Newton slope of $Q_{<1}$ is strictly less than $1$, so $t_N(Q_{<1})<\dim Q_{<1}$.
This contradicts the weak-admissibility inequality $t_H(Q_{<1})\le t_N(Q_{<1})$.
\end{proof}

The next proposition classifies the slopes of $D_2$ below one.

\begin{proposition}
\label{prop:automatic-D2-block}
For every good reduction, exactly one of the following holds:
\begin{enumerate}[label=\textup{(\roman*)},leftmargin=2.2em,
  beginpenalty=10000,midpenalty=10000]
  \item $(D_2)_{[0,1)}=0$; or
  \item there is a unique integer $h\ge1$ for which $(D_2)_{[0,1)}$ is isoclinic of slope $1-\tfrac1h$ and has dimension $h$.
\end{enumerate}
\end{proposition}

\begin{proof}
The filtered $\varphi$-module $D_2$ is weakly admissible with effective Hodge weights $0,1,2$ and $\dim\gr^0(D_2)_K=h^{0,2}(Y)=1$.
Put $S=(D_2)_{[0,1)}$ for simplicity.
Let $r=\dim S$.
The induced filtration on $S_K$ embeds $\gr^a S_K$ into $\gr^a(D_2)_K$ for every $a$.
In particular $\gr^0S_K$ is at most one-dimensional.
Since the filtration is effective, we have
\[
  t_H(S)\ \ge\ 0\cdot\dim\gr^0S_K+1\cdot\bigl(r-\dim\gr^0S_K\bigr)\ \ge\ r-1 .
\]
Weak admissibility gives $t_H(S)\le t_N(S)$, whence
\[
  \slopedefect(S)=r-t_N(S)\le1 .
\]
Slopes of $D_2$ are all non-negative: the sum $N$ of the negative slope subspaces is $\varphi$-stable with $t_H(N)\ge0$ by effectivity and $t_N(N)<0$ if $N\ne0$, which would contradict $t_H(N)\le t_N(N)$.

Let $\lambda = a/b$ be a slope of $S$ in lowest terms with $0 \leq a <b$.
By Dieudonn\'e--Manin theory its multiplicity is divisible by $b$, and an isoclinic part of multiplicity $mb$ contributes the positive integer $m(b-a)$ to $\slopedefect(S)$.
Hence, if $S\ne0$, the bound $\slopedefect(S)\le1$ forces exactly one part to occur.
This part necessarily has multiplicity $b$ and $b-a=1$; taking $h=b$ gives \textup{(ii)}, together with its uniqueness.
\end{proof}

The hard-Lefschetz pairing on $\HH^2(Y_{\overline K},\Qp)$, induced by an ample class on $Y$, pairs the slopes $\lambda$ and $2-\lambda$ of $D_2$.
Thus in case \textup{(ii)} the full list of slopes is
\[
  1-\tfrac1h,\quad 1,\quad 1+\tfrac1h
\]
with multiplicities $h,b_2-2h,h$; in case \textup{(i)} every slope is one.
\Cref{fig:polygons} illustrates the Hodge and Newton polygons in case \textup{(ii)}.

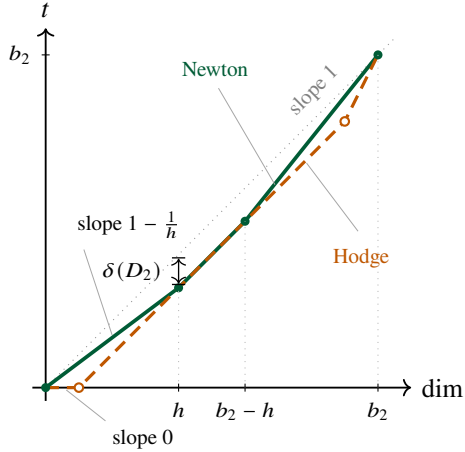
\begin{figure}[ht]
\centering
\begin{tikzpicture}[scale=0.44,
    lead/.style={gray!70,line width=0.3pt},
    nwt/.style={themegreen,line width=1.4pt},
    hdg/.style={themeorange,line width=1.2pt,dash pattern=on 5pt off 3pt},
    brk/.style={themegreen,fill=themegreen,inner sep=0pt,minimum size=3.4pt,
                circle},
    hbrk/.style={draw=themeorange,fill=white,line width=0.8pt,inner sep=0pt,
                 minimum size=3.2pt,circle}]
  \draw[->,black,line width=0.7pt] (-0.4,0) -- (11.0,0) node[right]{\small $\dim$};
  \draw[->,black,line width=0.7pt] (0,-0.4) -- (0,10.8) node[above]{\small $t$};
  \draw[dotted,gray!85] (0,0) -- (10.5,10.5);
  \node[rotate=45,anchor=south,font=\scriptsize,gray] at (8.5,8.5)
       {slope $1$};
  \draw[nwt] (0,0) -- (4,3) -- (6,5) -- (10,10);
  \draw[hdg] (0,0) -- (1,0) -- (9,8) -- (10,10);
  \node[brk] at (0,0) {}; \node[brk] at (4,3) {};
  \node[brk] at (6,5) {}; \node[brk] at (10,10) {};
  \node[hbrk] at (1,0) {}; \node[hbrk] at (9,8) {};
  \draw[|<->|,black,line width=0.5pt] (4,3.08) -- (4,3.92);
  \node[anchor=east,font=\scriptsize] at (3.82,3.5) {$\slopedefect(D_2)$};
  \draw[lead] (5.2,9.0) -- (6.9,6.1);
  \node[themegreen,anchor=south,font=\scriptsize] at (5.1,9.05) {Newton};
  \draw[lead] (9.4,4.7) -- (7.85,6.75);
  \node[themeorange,anchor=north,font=\scriptsize] at (9.5,4.5) {Hodge};
  \draw[lead] (1.15,4.15) -- (2.0,1.58);
  \node[anchor=south west,font=\scriptsize] at (0.65,4.2)
       {slope $1-\tfrac1h$};
  \draw[lead] (1.8,-1.0) -- (0.62,-0.07);
  \node[anchor=north west,font=\scriptsize] at (1.7,-0.96)
       {slope $0$};
  \foreach \x/\y/\l in {4/3/h,6/5/{b_2-h},10/10/{b_2}}{
    \draw[gray!55,dotted] (\x,0) -- (\x,\y);
    \draw[black] (\x,-0.12) -- (\x,0.12);
    \node[below,font=\scriptsize] at (\x,-0.14) {$\l$};}
  \draw[black] (-0.12,10) -- (0.12,10);
  \node[left,font=\scriptsize] at (-0.14,10) {$b_2$};
\end{tikzpicture}
\caption{The Hodge polygon of $D_2$ (dashed) and its Newton polygon (solid) in
  case \textup{(ii)} of \Cref{prop:automatic-D2-block}.
Both have slope one on $[h,b_2-h]$, where they coincide; they differ only at the
  two ends.
The drop $\slopedefect(D_2)\le1$ of the Newton polygon below the line of slope
one is attained at $\dim=h$.
In case \textup{(i)} the Newton polygon is that line.}
\label{fig:polygons}
\end{figure}

\begin{definition}
\label{def:crystalline-height}
For a good reduction $X$, define its \emph{degree-two crystalline height} by
\[
  h_2^{\crys}(X):=
  \begin{cases}
    \dim_{\Kzero}(D_2)_{[0,1)}, & \text{if }(D_2)_{[0,1)}\ne0,\\
    \infty, & \text{if }(D_2)_{[0,1)}=0.
  \end{cases}
\]
By \Cref{prop:automatic-D2-block}, every finite value is a positive integer $h$, and $(D_2)_{[0,1)}$ is then isoclinic of slope $1-\tfrac1h$.
\end{definition}

This definition does not assume that the Artin--Mazur functor $\Phi^2_X$ of \Cref{def:artin-mazur} is prorepresentable.
If $\Phi^2_X$ is a smooth one-dimensional formal group, its height equals $h_2^{\crys}(X)$ by \Cref{rem:AM-slope-convention}.

Consequently, the slope-zero condition \textup{(iii)} of \Cref{thm:main-all-good} can be written as
\begin{equation}\label{eq:iii-equiv}
  \HH^2_{\crys}(X/\Wit)[1/p]\text{ has a nonzero slope-zero part}
  \iff h_2^{\crys}(X)=1 .
\end{equation}

We first record the degree-two slope criterion for the surfaces underlying our examples.
For a smooth proper variety $Z/\bk$, write $D_q(Z)=\HH^q_{\crys}(Z/\Wit)[1/p]$, and let $\Kzero(-1)$ denote the one-dimensional isocrystal of slope one.

\begin{lemma}
\label{lem:ordinary-surface-slopes}
Let $Z$ be a $K3$ surface or an abelian surface over a perfect field $\bk$ of characteristic $p>0$. Then $Z$ is ordinary if and only if $D_2(Z)$ has a nonzero slope-zero part. In this case $(D_2(Z))_{[0,1)}$ is a line of slope zero, and the slopes of $D_2(Z)$ are $0,1,2$, with multiplicities $1,20,1$ in the $K3$ case and $1,4,1$ in the abelian case.
\end{lemma}

\begin{proof}
Both types of surfaces have torsion-free crystalline cohomology and trivial canonical bundle, so $h^{0,2}(Z)=h^{2,0}(Z)=1$. If $Z$ is ordinary, its degree-two Newton and Hodge polygons agree \cite[Proposition~7.3]{BlochKato1986}, giving the stated slope multiplicities.

Conversely, for a $K3$ surface the formal Brauer group is smooth and one-dimensional, so \Cref{rem:AM-slope-convention} shows that a nonzero slope-zero part is equivalent to height one, hence to ordinarity. For an abelian surface $A$, the exterior-algebra description $D_2(A)\simeq\bigwedge^2D_1(A)$ gives slope-zero multiplicity $\binom{f(A)}{2}$, where $f(A)$ is the $p$-rank. This is nonzero exactly when $f(A)=2$, that is, when $A$ is ordinary.
\end{proof}

The degree-two criterion can now be read directly from the underlying surface in two standard families.

\begin{example}
\label{ex:hilb}
Let $\mathscr S\to\Spec(\Oh_K)$ be smooth and proper with $K3$ generic fibre and special fibre $S/\bk$, and let $X=S^{[n]}$ be the good reduction of the Hilbert scheme $\Hilb^n$ of the generic fibre, a hyperk\"ahler variety of dimension $2n\ge4$.
For the Mukai vector $v=(1,0,1-n)$, the degree-two isomorphism \cite[Proposition~4.4(iii)]{FuLi2021} gives a decomposition of $F$-isocrystals
\[
  D_2(X)\simeq D_2(S)\oplus\Kzero(-1).
\]
The last summand comes from the exceptional divisor, so $(D_2(X))_{[0,1)}\simeq(D_2(S))_{[0,1)}$.
Thus $h_2^{\crys}(X)=1$ if and only if $S$ is ordinary, by \Cref{lem:ordinary-surface-slopes} and \eqref{eq:iii-equiv}.
\end{example}

\begin{example}
\label{ex:kummer}
Let $\mathscr A/\Oh_K$ be an abelian scheme with special fibre $A/\bk$, and suppose that the relative generalised Kummer variety $\mathscr X=K_n(\mathscr A)$ is smooth and proper over $\Oh_K$; this holds, for example, if $p\nmid n+1$ \cite[Proposition~6.5]{FuLi2021}.
Set $X=K_n(A)$, of dimension $2n\ge4$.
The degree-two calculation in \cite[proof of Proposition~6.7]{FuLi2021} and the exterior-algebra description of abelian cohomology give
\[
  D_2(X)\simeq D_2(A)\oplus\Kzero(-1)
  \simeq \bigwedge^2D_1(A)\oplus\Kzero(-1)
\]
as $F$-isocrystals. Hence $(D_2(X))_{[0,1)}\simeq(D_2(A))_{[0,1)}$, and $h_2^{\crys}(X)=1$ if and only if $A$ is ordinary, by \Cref{lem:ordinary-surface-slopes} and \eqref{eq:iii-equiv}.
\end{example}

\subsection{The Verbitsky component and higher-degree slopes}

The Verbitsky component of the second cohomology controls the slopes in $[0,1)$ in every even degree.

\begin{proposition}
\label{prop:quotient}
For $1\le i\le n$, cup product gives an injective morphism of isocrystals
\[
  \Sym^i D_2\lhook\joinrel\longrightarrow D_{2i},
\]
whose cokernel $Q_i$ has all Newton slopes at least $1$.
Consequently
\begin{equation}\label{eq:less1iso}
  (D_{2i})_{[0,1)}\simeq(\Sym^i D_2)_{[0,1)} .
\end{equation}
\end{proposition}

\begin{proof}
Verbitsky's theorem gives an injection $\Sym^i\HH^2(Y_{\CC},\QQ)\hookrightarrow\HH^{2i}(Y_{\CC},\QQ)$ for $i\le n$ \cite{Verbitsky1996,Bogomolov1996}; its image is the Verbitsky component, and the injectivity of the corresponding map in the Betti, $\ell$-adic and potentially semistable $p$-adic realisations is recorded in \cite[Theorem~3.14]{IIKTZ2025}.
For $Y$ over an arbitrary field of characteristic zero the statement follows by spreading out and choosing an embedding of a finitely generated field of definition into $\CC$.

The de Rham cup map is therefore injective, and the $B_{\dR}$-comparison yields an injective $G_K$-equivariant map $\Sym^i V_2\hookrightarrow V_{2i}$, where $V_m=\HH^m_{\et}(Y_{\overline K},\Qp)$; let $U_i$ denote its cokernel.
The category of crystalline representations is closed under subquotients, and $D_{\crys}$ is exact and compatible with tensor operations, so $U_i$ is crystalline and $D_{\crys}(U_i)=Q_i=D_{2i}/\Sym^iD_2$.
In particular $Q_i$ is weakly admissible.

Morphisms of weakly admissible filtered $\varphi$-modules are strict, so on degree-zero graded pieces the cup map becomes $\Sym^i\HH^2(Y,\Oh_Y)\to\HH^{2i}(Y,\Oh_Y)$.
For an irreducible hyperk\"ahler variety both sides are one-dimensional and this map is an isomorphism; hence $\gr^0(Q_i)_K=0$, and effectivity gives $\Fil^1(Q_i)_K=(Q_i)_K$.
By \Cref{lem:WA}, every Newton slope of $Q_i$ is at least one.
Finally, slope truncation is exact for isocrystals (as one checks after extending the perfect residue field, using Dieudonn\'e--Manin), so applying the interval $[0,1)$ to the exact sequence
\[
  0\to\Sym^iD_2\to D_{2i}\to Q_i\to 0
\]
yields \eqref{eq:less1iso}.
\end{proof}

\begin{corollary}
\label{cor:automatic-top-slopes}
Every good reduction $X$ of dimension $2n\ge4$ satisfies
\[
  \dim_{\Kzero}(D_{2n})_{[0,1)}=
  \begin{cases}
    1, & h_2^{\crys}(X)=1,\\
    0, & h_2^{\crys}(X)>1,
  \end{cases}
\]
and in the first case $(D_{2n})_{[0,1)}$ is a line of slope zero.
Equivalently, $\HH^{2n}_{\crys}(X/\Wit)[1/p]$ has a nonzero slope-zero part if and only if $\HH^2_{\crys}(X/\Wit)[1/p]$ does.
\end{corollary}

\begin{proof}
By \eqref{eq:less1iso} the left-hand side is $(\Sym^nD_2)_{[0,1)}$, and the slopes of $\Sym^nD_2$ are the sums of $n$ slopes of $D_2$.
If $h_2^{\crys}=1$, the unique slope of $D_2$ below one is a slope-zero line and all other slopes are at least one; the $n$-th power of that line is then the unique summand of $\Sym^nD_2$ with slope below one, and it is again a slope-zero line.
If $1<h_2^{\crys}<\infty$, the smallest slope of $D_2$ is
\[
  1-\tfrac{1}{h_2^{\crys}}\ge\tfrac{1}{2}
\]
so every slope of $\Sym^nD_2$ is at least $\tfrac n2\ge1$.
If $h_2^{\crys}=\infty$ the assertion is immediate.
\end{proof}

To relate these slope computations to Witt-vector cohomology, we use the following comparison for an arbitrary smooth proper algebraic space.

\begin{proposition}
\label{prop:WittSlope}
For every smooth proper algebraic space $X/\bk$ and every $q$ there is a Frobenius-compatible isomorphism
\[
  \HH^q(X,\Wit\Oh_X)\otimes_\Wit\Kzero\simeq (D_q)_{[0,1)} .
\]
\end{proposition}

\begin{proof}
This is the $r=0$ case of the rational slope spectral sequence \cite[\S II, Corollary~3.5]{Illusie1979}.
For algebraic spaces, use the de Rham--Witt comparison and slope spectral sequence on the \'etale site \cite[Theorems~4.4.17 and~4.5.15]{Olsson2007Crystalline}; the degree-zero term is $\Wit\Oh_X$, with the same Frobenius convention.
\end{proof}

\subsection{Hodge-goodness and Artin--Mazur formal groups}

The slope results above apply to every good reduction, without a hypothesis on its coherent cohomology.
We now impose the algebra structure expected from characteristic zero to study the Artin--Mazur formal groups.
This condition makes sense for any smooth proper algebraic space, independently of a good-reduction model.

\begin{definition}
\label{def:hodge-good}
Let $X$ be a smooth proper algebraic space over $\bk$ of even dimension $2n$.
We say that $X$ is \emph{Hodge-good} if there is a class $\eta\in\HH^2(X,\Oh_X)$ for which cup product induces an isomorphism of graded $\bk$-algebras
\begin{equation}\label{eq:HG}
  \HH^\bullet(X,\Oh_X)\;\simeq\;\bk[\eta]/(\eta^{n+1}),
  \qquad\deg\eta=2 .
\end{equation}
Equivalently, $\HH^{2i}(X,\Oh_X)=\bk\,\eta^i$ for $0\le i\le n$ and $\HH^j(X,\Oh_X)=0$ for every odd $j$.
\end{definition}

\begin{definition}
\label{def:hodge-good-reduction}
A \emph{Hodge-good reduction} is a good reduction in the sense of \Cref{def:good-reduction} which is Hodge-good in the sense of \Cref{def:hodge-good}.
\end{definition}

\begin{remark}
\label{rem:HG-not-automatic}
After shrinking the base of a fixed arithmetic spread of a hyperk\"ahler variety, coherent cohomology commutes with base change and the cup-power maps are isomorphisms, so the fibres are Hodge-good.
This does not establish Hodge-goodness at every good reduction: both the dimensions of coherent cohomology and the nonvanishing of cup powers require justification at an arbitrary good prime.
The dichotomy of \Cref{thm:all-good-dichotomy} is therefore proved without this hypothesis.
\end{remark}

Hodge-goodness supplies exactly the vanishing required by the criterion of \Cref{lem:AM-prorepresentable}, in every even degree at once.

\begin{proposition}
\label{prop:AM-hodge-good}
Let $X$ be a Hodge-good smooth proper algebraic space of dimension $2n$ over $\bk$.
\begin{enumerate}[label=\textup{(\arabic*)},leftmargin=2.2em]
\item For every even $q=2i$ with $1\le i\le n$, the functor $\Phi^{2i}_X$ is prorepresentable and formally smooth with tangent space $\HH^{2i}(X,\Oh_X)=\bk\,\eta^i$; hence it is a smooth one-dimensional formal group.
\item For every odd $q$ with $1\le q\le 2n-1$, the functor $\Phi^q_X$ is zero.
  In particular no odd formal group enters the height argument.
\item If $X$ is a Hodge-good reduction of a hyperk\"ahler variety, then $\hgt(\Phi^2_X)=h_2^{\crys}(X)$.
\end{enumerate}
\end{proposition}

\begin{proof}
(1) By \eqref{eq:HG} we have $\HH^{2i-1}(X,\Oh_X)=\HH^{2i+1}(X,\Oh_X)=0$, the second also for $i=n$ by cohomological dimension, and $\dim_\bk\HH^{2i}(X,\Oh_X)=1$; now apply \Cref{lem:AM-prorepresentable}.

(2) For a square-zero extension $A\twoheadrightarrow A/I$ in $\mathrm{Art}_\bk$ the sequence displayed in the proof of \Cref{lem:AM-prorepresentable} exhibits $\ker\bigl(\Phi^q_X(A)\to\Phi^q_X(A/I)\bigr)$ as a quotient of $I\otimes_\bk\HH^q(X,\Oh_X)$, which vanishes for odd $q$.
Induction along the powers of the maximal ideal of $A$ gives $\Phi^q_X(A)=0$.

(3) By (1) and \Cref{rem:AM-slope-convention}, $(D_2)_{[0,1)}$ is isoclinic of slope $1-1/\hgt(\Phi^2_X)$ and of dimension $\hgt(\Phi^2_X)$ when the height is finite, and is zero when it is infinite.
Comparing with \Cref{prop:automatic-D2-block} and \Cref{def:crystalline-height} gives the equality.
\end{proof}

\subsection{Cartier products and heights}

We now compute the even Artin--Mazur heights entirely in characteristic $p$.
Recall the covariant Cartier module $\Cart(E)=\Hom(\widehat{\Wit},E)$ of \Cref{subsec:Artin-Mazur formal group}.
The category of $V$-complete Cartier modules carries a completed symmetric monoidal product $M\mathbin{\widehat\boxtimes}_{\Wit}N$, and the corresponding product of formal Lie groups by the Cartier theory is written as $E\boxtimes E'$.
In this convention, we have
\begin{equation}\label{eq:boxCartier}
  \Cart(E\boxtimes E')\;\simeq\;
  \Cart(E)\mathbin{\widehat\boxtimes}_{\Wit}\Cart(E') ;
\end{equation}
see \cite[\S4.2,~Example~4.18]{AntieauNikolaus2021} and \cite[Remark~7.2.18]{MondalReinecke2026}.
We keep the base ring in the notation: this is the relative product in $\Wit$-module objects, not the absolute Cartier tensor product, whose unit is $\Wit(\ZZ)$.

Throughout the rest of this subsection $X$ is a Hodge-good smooth proper algebraic space of dimension $2n$ over $\bk$, not necessarily a reduction, and we put
\[
  E_i=\Phi^{2i}_X,\qquad M_i=\Cart(E_i)\simeq\HH^{2i}(X,\Wit\Oh_X)
  \qquad(1\le i\le n),
\]
the identification being \Cref{prop:AM-cartier}, which applies by \Cref{prop:AM-hodge-good}(1).

\begin{proposition}
\label{prop:boxcup}
For $a,b\ge1$ with $a+b\le n$, the cup product of Witt-vector cohomology induces a canonical morphism of formal Lie groups
\begin{equation}\label{eq:mubox}
  \mu_{a,b}\colon E_a\boxtimes E_b\longrightarrow E_{a+b},
\end{equation}
and $\mu_{a,b}$ is an isomorphism.
Consequently
\begin{equation}\label{eq:boxpowers}
  \Phi^{2i}_X\;\simeq\;(\Phi^2_X)^{\boxtimes i}\qquad(1\le i\le n).
\end{equation}
\end{proposition}

\begin{proof}
The products of the Witt sheaves give a continuous $\Wit$-balanced pairing $M_a\times M_b\to M_{a+b}$, $(x,y)\mapsto x\smile y$, subject to the standard Witt identities
\[
  F(x\smile y)=Fx\smile Fy,\qquad
  V(x\smile Fy)=Vx\smile y,\qquad
  V(Fx\smile y)=x\smile Vy .
\]
These are exactly the $(V,F)$-bilinearity relations of \cite[Definition~4.7 and Example~4.8]{AntieauNikolaus2021}; such pairings are corepresented by the Cartier box product \cite[Lemma~4.9]{AntieauNikolaus2021}, and the construction passes to derived $V$-completion \cite[Proposition~4.14]{AntieauNikolaus2021}.
As $M_{a+b}$ is $V$-complete we obtain $M_a\mathbin{\widehat\boxtimes}_{\Wit}M_b\to M_{a+b}$, which by the covariance of $\Cart$ and \eqref{eq:boxCartier} is a morphism in the direction \eqref{eq:mubox}.

It remains to compute the differential of $\mu_{a,b}$.
Cartier theory gives $\Lie(E_i)\simeq M_i/VM_i$ \cite[Theorem~4.23]{Zink1984}.
The exact sequences
\[
  0\longrightarrow\Wit_{r-1}\Oh_X\xrightarrow{\ V\ }\Wit_r\Oh_X
   \xrightarrow{\ R_{r-1}\ }\Oh_X\longrightarrow0
\]
together with $\HH^{2i-1}(X,\Oh_X)=0$ show that $V$ is injective on $M_{i}$.
Then, passing to the limit along the restriction maps, the uniqueness of $V$-preimages gives $\ker\bigl(M_i\to\HH^{2i}(X,\Oh_X)\bigr)=VM_i$ and hence an injection
\begin{equation}\label{eq:rhoi}
  \rho_i\colon M_i/VM_i\lhook\joinrel\longrightarrow\HH^{2i}(X,\Oh_X).
\end{equation}
Both sides are one-dimensional, the left by \Cref{prop:AM-hodge-good}(1) and the right by \eqref{eq:HG}, so $\rho_i$ is an isomorphism.
Reduction modulo $V$ is symmetric monoidal \cite[Lemma~4.12 and Proposition~4.14]{AntieauNikolaus2021}, whence
\[
  \Lie(E_a\boxtimes E_b)\simeq(M_a/VM_a)\otimes_\bk(M_b/VM_b),
\]
and since $R_{r-1}\colon\Wit\Oh_X\to\Oh_X$ is a map of sheaves of rings, the differential of $\mu_{a,b}$ is identified under \eqref{eq:rhoi} with the ordinary cup product
\[
  \HH^{2a}(X,\Oh_X)\otimes_\bk\HH^{2b}(X,\Oh_X)\longrightarrow
  \HH^{2a+2b}(X,\Oh_X),
\]
which sends $\eta^a\otimes\eta^b$ to $\eta^{a+b}\ne0$ and is therefore an isomorphism.
The source of \eqref{eq:mubox} is again a smooth one-dimensional formal Lie group \cite[Proposition~7.2.17]{MondalReinecke2026}, and a morphism of smooth one-dimensional formal schemes with invertible differential is an isomorphism by the formal inverse function theorem\cite[discussion after Definition~1.22]{Zink1984}.
Iterating \eqref{eq:mubox} gives \eqref{eq:boxpowers}.
\end{proof}

\begin{corollary}
\label{cor:boxheights}
Let $X$ be a Hodge-good smooth proper algebraic space of dimension $2n\ge4$ over $\bk$ and write $h=\hgt(\Phi^2_X)$, which equals $h_2^{\crys}(X)$ if $X$ is a Hodge-good reduction.
Then
\begin{equation}\label{eq:heightpattern}
  \bigl(\hgt(\Phi^2_X),\hgt(\Phi^4_X),\dots,\hgt(\Phi^{2n}_X)\bigr)=
  \begin{cases}
    (1,1,\dots,1), & h=1,\\
    (h,\infty,\dots,\infty), & 1<h<\infty,\\
    (\infty,\infty,\dots,\infty), & h=\infty .
  \end{cases}
\end{equation}
\end{corollary}

\begin{proof}
Heights may be computed after extending $\bk$ to an algebraic closure.
If $h=1$, then $E_1\simeq\widehat{\mathds{G}}_m$, which is the unit for $\boxtimes$ \cite[Remark~7.2.20]{MondalReinecke2026}; by \eqref{eq:boxpowers} every $E_i$ is then of height one.
If $h>1$, possibly infinite, then the box product of two one-dimensional formal Lie groups of height greater than one is $\widehat{\mathds{G}}_a$ \cite[Proposition~7.2.21]{MondalReinecke2026}, so $E_2\simeq E_1\boxtimes E_1\simeq\widehat{\mathds{G}}_a$, and inductively $E_i\simeq E_1\boxtimes E_{i-1}$ is additive for every $i\ge2$.
Additive groups have infinite height, while the first entry is unchanged.
\end{proof}

\begin{corollary}
\label{cor:HG-qF-dichotomy}
Let $Z$ be a Hodge-good smooth proper geometrically connected algebraic space of dimension $2n\ge4$ over a perfect field of characteristic $p$, with $\omega_Z\simeq\Oh_Z$.
Then
\[
  \hgt(Z)=\hgt(\Phi^{2n}_Z)=
  \begin{cases}
    1, & \hgt(\Phi^2_Z)=1,\\
    \infty, & \hgt(\Phi^2_Z)>1.
  \end{cases}
\]
Thus $Z$ is quasi-$F$-split if and only if it is $F$-split, equivalently if and only if $\Phi^2_Z$ has height one.
\end{corollary}

\begin{proof}
By \Cref{prop:AM-hodge-good}(1), the formal groups $\Phi^2_Z$ and $\Phi^{2n}_Z$ are smooth and one-dimensional.
Put $M=\Cart(\Phi^{2n}_Z)=\HH^{2n}(Z,\Wit\Oh_Z)$ (\Cref{prop:AM-cartier}).
If $Z$ is quasi-$F$-split, then $M$ is finitely generated over $\Wit$ by \Cref{lem:basic-properties}(2), so $\Phi^{2n}_Z$ has finite height by \Cref{rem:AM-infinite-height}.
The height pattern of \Cref{cor:boxheights} then forces $\hgt(\Phi^2_Z)=\hgt(\Phi^{2n}_Z)=1$.

Conversely, suppose $\hgt(\Phi^2_Z)=1$.
By \Cref{cor:boxheights}, the top formal group also has height one, so $F$ is bijective on $M$ and $V=pF^{-1}$.
Thus $M/VM=M/pM$, and $F$ acts nontrivially on this one-dimensional quotient.
The isomorphism $M/VM\simeq\HH^{2n}(Z,\Oh_Z)$ from \eqref{eq:rhoi} is induced by Witt restriction, which commutes with Frobenius; hence Frobenius is nonzero on top coherent cohomology.
Hence $Z$ is $F$-split by \Cref{lem:Fsplit}, and $\hgt(Z)=1$ by \Cref{def:quasi-f-split}.
If $\hgt(\Phi^2_Z)>1$, then \Cref{cor:boxheights} gives $\hgt(\Phi^{2n}_Z)=\infty$, and the first implication shows that $Z$ is not quasi-$F$-split, so $\hgt(Z)=\infty$.
\end{proof}

\begin{remark}
\label{rem:cartier-vs-verbitsky}
The Verbitsky comparison (\Cref{prop:quotient}) uses a characteristic-zero lift and $p$-adic Hodge theory to determine $(D_{2i})_{[0,1)}$, without assuming Hodge-goodness.
The Cartier argument assumes \eqref{eq:HG}, uses no lift, and identifies the formal groups integrally, giving their heights in every even degree.
For Hodge-good reductions, the height pattern in \Cref{cor:boxheights}, together with \Cref{prop:AM-cartier,prop:WittSlope}, recovers the top-degree slope description of \Cref{cor:automatic-top-slopes}.
\end{remark}

\section{The Witt--Euler characteristic and the dichotomy}
\label{sec:main-results}

Unless another base is specified, throughout this section we keep the arithmetic setting of \Cref{subsec:automatic}: $X$ is the special fibre of a smooth proper algebraic space $\mathscr X\to\Spec(\Oh_K)$ whose geometric generic fibre is an irreducible hyperk\"ahler variety of dimension $2n$, and $h_2^{\crys}(X)$ is its degree-two crystalline height (\Cref{def:crystalline-height}).
We first prove the Witt--Euler identity for smooth proper algebraic spaces and evaluate its slope sum for good reductions. We then establish the ordinary implication and the main dichotomy, before recording applications and a top Witt-cohomology refinement under $\HH^1(X,\Oh_X)=0$.

\subsection{The Witt--Euler identity}

For a smooth proper algebraic space $Z$ over a perfect field, write $D_j(Z)=\HH^j_{\crys}(Z/\Wit)[1/p]$.
Using the slope defect defined in \Cref{subsec: weak admissibility and slopes}, set
\[
  \slopedefect_j(Z)=\slopedefect\bigl(D_j(Z)\bigr),
  \qquad
  E_0(Z)=\sum_{j\ge0}(-1)^j\slopedefect_j(Z) .
\]
Crystalline slopes are non-negative, so only the slope-$[0,1)$ part contributes to $\slopedefect_j(Z)$.
When $\HH^j(Z,\Wit\Oh_Z)$ is finitely generated over $\Wit$, \Cref{prop:WittSlope} identifies $\slopedefect_j(Z)$ with its Verschiebung index, the $p$-adic valuation of the determinant of $V$ on its free quotient.
Since $\HH^j_{\crys}$ vanishes for $j>2\dim Z$, the sum defining $E_0(Z)$ is finite.
For the special fibre $X$ of \Cref{subsec:automatic} we abbreviate $\slopedefect_j=\slopedefect_j(X)$; in particular, the quantity $\slopedefect(S)$ in the proof of \Cref{prop:automatic-D2-block} is $\slopedefect_2$.
In Crew's notation $\slopedefect_j(Z)$ is the Hodge--Newton number $m^{0,j}(Z)$ and $E_0(Z)=m^0(Z)$.

The next proposition is the case $i=0$ of Crew's Euler characteristic formula
\[
  m^i(Z)+T^i(Z)+2T^{i-1}(Z)+T^{i-2}(Z)=\chi(Z,\Omega^i_Z)
\]
\cite[Theorem~4]{Crew1985}, the hypothesis of finite generation being what makes the domino term $T^0(Z)$ vanish.
We record a self-contained proof for the reader's convenience.

\begin{proposition}
\label{prop:Witt-Euler}
Let $Z$ be a smooth proper algebraic space of dimension $d$ over a perfect field, and suppose $\HH^j(Z,\Wit\Oh_Z)$ is finitely generated over $\Wit$ for every $j$.
Then
\[
  \chi(Z,\Oh_Z)=E_0(Z) .
\]
\end{proposition}

\begin{proof}
Write $M_j=\HH^j(Z,\Wit\Oh_Z)$ and let $V_j$ denote the Verschiebung on $M_j$, a $\sigma^{-1}$-semilinear endomorphism.
The finite length of each $\HH^j(Z,\Wit_r\Oh_Z)$ makes these inverse systems satisfy the Mittag--Leffler condition, so $M_j$ is computed by the inverse limit with no derived-limit term \cite[Proposition~4.5.2]{Olsson2007Crystalline}.
The exact sequence $0\to\Wit\Oh_Z\xrightarrow{V}\Wit\Oh_Z\to\Oh_Z\to0$ then gives, in each degree,
\[
  0\longrightarrow\coker V_j\longrightarrow\HH^j(Z,\Oh_Z)\longrightarrow\ker V_{j+1}\longrightarrow0 .
\]
Both $\ker V_j$ and $\coker V_j$ are killed by $p$, since $FV=VF=p$, and hence have finite length.
Moreover $M_j=0$ for $j>d$, and $V_0$ is injective.
Taking alternating sums of the displayed sequences, the terms $\ker V_{j+1}$ telescope and we obtain
\[
  \chi(Z,\Oh_Z)=\sum_j(-1)^j\bigl(\length_\Wit\coker V_j-\length_\Wit\ker V_j\bigr) .
\]

It remains to evaluate each summand.
Let $M$ be a finitely generated $\Wit$-module with such operators $F,V$, and write $T=M_{\mathrm{tors}}$ and $L=M/T$.
Then $V$ is injective on the free module $L$, and the snake lemma gives $\ker V_M=\ker V_T$ together with an exact sequence $0\to\coker V_T\to\coker V_M\to\coker V_L\to0$.
As $T$ has finite length and $V_T$ is semilinear for an automorphism of $\Wit$, its kernel and cokernel have the same length, so the torsion cancels:
\[
  \length\coker V_M-\length\ker V_M=\length\coker V_L .
\]
Choose a basis of $L$ and let $A,B$ be the matrices of $F$ and $V$ in it, so that $A\,\sigma(B)=pI_r$ with $r=\rank_\Wit L$.
Then
\[
  \length\coker V_L=v_p(\det B)=r-v_p(\det A)=r-t_N(M[1/p]) .
\]
By \Cref{prop:WittSlope} we have $M_j[1/p]\simeq\bigl(D_j(Z)\bigr)_{[0,1)}$, so the degree-$j$ index is exactly $\slopedefect(M_j[1/p])=\slopedefect_j(Z)$.
Substituting gives $\chi(Z,\Oh_Z)=E_0(Z)$.
\end{proof}

For good reductions, the odd-degree terms in $E_0(X)$ vanish.
\begin{lemma}
\label{lem:HK-odd-low-slopes}
Let $X$ be a good reduction of a hyperk\"ahler variety. For every integer $0\le i\le\dim X-1$, we have $\slopedefect_{2i+1}(X)=0$.
\end{lemma}

\begin{proof}
$D_{2i+1}$ is weakly admissible and $\gr^0(D_{2i+1})_K=\HH^{2i+1}(Y,\Oh_Y)=0$, so $\Fil^1(D_{2i+1})_K=(D_{2i+1})_K$ and \Cref{lem:WA} implies that $D_{2i+1}$ has no Newton slope below one when $2i+1$ is odd.
\end{proof}

\begin{proposition}
\label{prop:HK-Euler-slopes}
Let $X$ be a good reduction of dimension $2n$ with $n\ge2$. Then
\[
  E_0(X)=
  \begin{cases}
    n+1, & h_2^{\crys}(X)=1,\\
    2,   & 1<h_2^{\crys}(X)<\infty,\\
    1,   & h_2^{\crys}(X)=\infty .
  \end{cases}
\]
\end{proposition}

\begin{proof}
Degree zero contributes $\slopedefect_0=1$, and all odd degrees contribute zero by \Cref{lem:HK-odd-low-slopes}.
In degree two, if $h_2^{\crys}=h<\infty$ then the block below one has slope $1-\tfrac1h$ and multiplicity $h$, so $\slopedefect_2=h\bigl(1-(1-\tfrac1h)\bigr)=1$; if $h_2^{\crys}=\infty$ then $\slopedefect_2=0$.
For $2\le i\le n$, \Cref{prop:quotient} identifies the slopes below one in $D_{2i}$ with those of $\Sym^iD_2$.

If $h_2^{\crys}=1$, the only such slope is a slope-zero line, so $\slopedefect_{2i}=1$ for every $0\le i\le n$ and $E_0=n+1$.
If $1<h_2^{\crys}<\infty$, every slope of $\Sym^iD_2$ is at least $i/2\ge1$ for $i\ge2$, so $\slopedefect_{2i}=0$ in those degrees and $E_0=\slopedefect_0+\slopedefect_2=2$.
If $h_2^{\crys}=\infty$, every degree-two slope is already at least one, so $\slopedefect_{2i}=0$ for all $i\ge1$ and $E_0=1$.
\end{proof}

\subsection{The ordinary branch and the dichotomy}

As it is well known, for a K3 surface in a perfect field, being ordinary is equivalent to being $F$-split, and also equivalent to having Artin--Mazur height one.
The following proposition shows that one of the directions holds for good reductions of hyperk\"ahler varieties; the idea comes from the proof of \cite[Theorem 3.2]{Yobuko2023}.
\begin{proposition}
\label{prop:unconditional-ordinary}
For a good reduction $X$, we have the following implications
\[
  h_2^{\crys}(X)=1\ \Longrightarrow\  \text{ $X$ is $F$-split}\ \bigl(\Longleftrightarrow\ \hgt(X)=1 \bigr).
\]
\end{proposition}

\begin{proof}
Extending the perfect residue field, we may assume that it is algebraically closed.
Set $H=\HH^{2n}(X,\Wit\Oh_X)$, and let $T\subseteq H$ denote its $\Wit$-torsion submodule.
The exact sequence
\[
  0\longrightarrow\Wit\Oh_X\xrightarrow{\ V\ }\Wit\Oh_X\longrightarrow\Oh_X\longrightarrow0
\]
gives $H/VH\simeq\HH^{2n}(X,\Oh_X)=\bk$ in top degree, and \Cref{cor:automatic-top-slopes} together with \Cref{prop:WittSlope} shows that $H[1/p]$ is a rank-one slope-zero isocrystal.
By \cite[Theorem~4.5.12]{Olsson2007Crystalline}, $L=H/T$ is a free $\Wit$-module of rank one; $V$ is injective on $L$ since $FV=p$.
The snake lemma applied to $0\to T\to H\to L\to0$ then gives
\[
  0\longrightarrow T/VT\longrightarrow H/VH\longrightarrow L/VL\longrightarrow0 .
\]
The last term is nonzero while the middle one has dimension $1$, so $T/VT=0$; as $H$, and hence $T$, is $V$-adically separated \cite[Corollary~4.5.6]{Olsson2007Crystalline}, $T=VT$ forces $T=0$.
Thus $H$ is free of rank one.

Because $H[1/p]$ has slope zero, $F$ is bijective on it, and $FV=p$ forces $VH=pH$; hence $F$ is a unit on $H$ and acts nontrivially on $H/VH=\HH^{2n}(X,\Oh_X)$.
Since $\omega_X\simeq\Oh_X$ by \Cref{prop:canonical-good}, \Cref{lem:Fsplit} shows that $X$ is Frobenius split.
\end{proof}

The Witt--Euler calculation and the ordinary implication now give the main dichotomy without a Hodge-goodness assumption.
\begin{theorem}
\label{thm:all-good-dichotomy}
Let $X$ be a good reduction of a projective irreducible hyperk\"ahler variety of dimension $2n\ge4$, in the setting of \Cref{subsec:automatic}.
The smooth proper model and its special fibre may be algebraic spaces.
Then the following are equivalent:
\begin{enumerate}[label=\textup{(\roman*)},leftmargin=2.2em]
\item $X$ is quasi-$F$-split;
\item $X$ is Frobenius split;
\item $h_2^{\crys}(X)=1$;
\item $\HH^j(X,\Wit\Oh_X)$ is finitely generated over $\Wit$ for every $j$.
\end{enumerate}
Consequently
\[
  \hgt(X)=
  \begin{cases}
    1, & h_2^{\crys}(X)=1,\\
    \infty, & h_2^{\crys}(X)>1.
  \end{cases}
\]
\end{theorem}

\begin{proof}
Condition (i) implies (iv) by \Cref{lem:basic-properties}(2).
If (iv) holds, then \Cref{prop:Witt-Euler} and $\chi(X,\Oh_X)=n+1$ (\Cref{prop:canonical-good}) give $E_0(X)=n+1$; since $n+1\ge3$, the two non-ordinary values $2$ and $1$ allowed by \Cref{prop:HK-Euler-slopes} are impossible, so $h_2^{\crys}(X)=1$ and (iv) implies (iii).
\Cref{prop:unconditional-ordinary} gives (iii)$\Rightarrow$(ii), and (ii)$\Rightarrow$(i) is \Cref{lem:basic-properties}(1).

For the last assertion, if $h_2^{\crys}(X)=1$ then $\hgt(X)=1$ by \Cref{prop:unconditional-ordinary}.
If $h_2^{\crys}(X)>1$ then $X$ fails (iii), hence also (i), and $\hgt(X)=\infty$.
\end{proof}

\begin{proof}[Proof of \Cref{thm:main-all-good}]
By \eqref{eq:iii-equiv}, condition (iii) of \Cref{thm:main-all-good} is equivalent to $h_2^{\crys}(X)=1$, so the statement is exactly \Cref{thm:all-good-dichotomy}.
\end{proof}

\begin{remark}
\label{rem:no-supersingularity}
The dichotomy does not imply full supersingularity: it does not force all slopes of $\HH^{2i}_{\crys}$ to equal $i$ when $h_2^{\crys}(X)>1$.
The argument also does not identify which Witt-vector cohomology group fails to be finitely generated; \Cref{prop:top-witt-criterion} identifies the top degree when $\HH^1(X,\Oh_X)=0$.
\end{remark}

\subsection{Consequences and examples}

The main theorem also rules out Hodge--Wittness in the nonordinary case.
\begin{corollary}
\label{cor:hodge-witt-dichotomy}
If a good reduction $X$ of dimension $2n\ge4$ has $h_2^{\crys}(X)>1$, then there is some $j$ for which $\HH^j(X,\Wit\Oh_X)$ is not finitely generated over $\Wit$.
In particular, $X$ is not Hodge--Witt.
\end{corollary}

\begin{example}\label{ex:hilb2}
By \Cref{thm:all-good-dichotomy} and \Cref{ex:hilb,ex:kummer}, the quasi-$F$-splitting height of either series can be read off from the surface which it is built on:
\[
  \begin{aligned}
    \hgt\bigl(S^{[n]}\bigr)
      &=\begin{cases}
          1,      & S\text{ ordinary},\\
          \infty, & S\text{ non-ordinary},
        \end{cases}\\
    \hgt\bigl(K_n(A)\bigr)
      &=\begin{cases}
          1,      & A\text{ ordinary},\\
          \infty, & A\text{ non-ordinary}.
        \end{cases}
  \end{aligned}
\]
This is consistent with Yobuko's direct computation that $S^{[n]}$ with $n\ge2$ is quasi-$F$-split only if $S$ is Frobenius split \cite[Theorem~5.6]{Yobuko2023}.
In particular, if $S$ is a non-ordinary $K3$ surface, or $A$ a non-ordinary abelian surface, of finite formal Brauer height $h>1$, then $S^{[n]}$ and $K_n(A)$ are neither quasi-$F$-split nor Hodge--Witt; \Cref{cor:generalized Kummer and Hilbert scheme no fg} below makes this precise.
\end{example}

The ordinary case also applies to moduli spaces of sheaves. For $p>2$, case \textup{(a)} extends the $K3$-surface instance of Kumar--Thomsen's Hilbert-scheme theorem \cite[Theorem~2]{KumarThomsen2001} to other primitive Mukai vectors. The following application is suggested by Charles Vial.
\begin{corollary}\label{cor:ordinary-sheaf-moduli-F-split}
Let $\bk$ be an algebraically closed field of characteristic $p>0$. Let $Z$ be a $K3$ or abelian surface. In each case let $v=(r,c_1,s)$ be an algebraic Mukai vector on the indicated surface $Z$, with $\langle v,v\rangle=c_1^2-2rs$, and let $H$ be an ample polarization general for $v$. Write $M_H(Z,v)$ for the moduli space of Gieseker-stable sheaves on $Z$.
\begin{enumerate}[label=\textup{(\alph*)},leftmargin=2.2em]
\item If $S/\bk$ is an ordinary $K3$ surface and $v$ is primitive with $r>0$ and $\langle v,v\rangle>0$, then $M_H(S,v)$ is $F$-split.
\item If $A/\bk$ is an ordinary abelian surface and $v$ is primitive with $r>0$ and $\langle v,v\rangle=2n+2$ for $n\ge2$ and $p\nmid n+1$, then both the Albanese fibre $K_H(A,v)$ and $M_H(A,v)$ are $F$-split.
\end{enumerate}
\end{corollary}

\begin{proof}
Let $Z=S$ in case \textup{(a)} and $Z=A$ in case \textup{(b)}. By \Cref{lem:ordinary-surface-slopes}, $(D_2(Z))_{[0,1)}$ is a line of slope zero. Consider the crystalline Mukai isocrystal
\[
  \widetilde D(Z)=D_0(Z)(-1)\oplus D_2(Z)\oplus D_4(Z)(1)
\]
whose outer summands have slope one, as does $\Kzero v$. The Mukai pairing takes values in $\Kzero(-2)$, of slope two. A slope-$\lambda$ subobject with $\lambda<1$ therefore pairs trivially with $\Kzero v$, since their tensor product has slope $\lambda+1<2$. Thus $(v^\perp)_{[0,1)}=(D_2(Z))_{[0,1)}$ is a line of slope zero.

In case \textup{(a)}, \cite[Proposition~4.4 and its proof]{FuLi2021} gives a smooth projective mixed-characteristic lift of $M_H(S,v)$ with hyperk\"ahler geometric generic fibre and an isomorphism $D_2(M_H(S,v))\simeq v^\perp$ of isocrystals. Hence $h_2^{\crys}(M_H(S,v))=1$, so \Cref{prop:unconditional-ordinary} gives the assertion. The same argument applies to $K_H(A,v)$ by \cite[Proposition~6.9 and its proof]{FuLi2021}, proving the first assertion of \textup{(b)}. Here we use the local form of \Cref{prop:unconditional-ordinary}: its proof also applies to these complete mixed-characteristic lifts with perfect residue field.

Finally, $A\times\widehat A$ is ordinary and hence $F$-split. The isotrivialization of the Albanese map in \cite[diagram~(50)]{FuLi2021} gives a finite \'etale cover
\[
  K_H(A,v)\times A\times\widehat A\longrightarrow M_H(A,v)
\]
of degree $(n+1)^8$, prime to $p$. Its source is $F$-split, and the normalized trace descends a Frobenius splitting to $M_H(A,v)$.
\end{proof}

\subsection{Primitivity and top Witt cohomology}
\label{subsec:primitive reduction}
A $K3$ surface $S$ has $h^{0,1}(S)=\dim_\bk\HH^1(S,\Oh_S)=0$ by definition.
Since hyperk\"ahler varieties are the higher-dimensional analogues of $K3$ surfaces, it is natural to ask whether a good reduction $X$ inherits the vanishing $\HH^1(X,\Oh_X)=0$.
It does not inherit it formally: by \Cref{rem:h1-not-automatic}, simple connectedness still leaves room for a nonreduced Picard scheme, whose tangent space is exactly $\HH^1(X,\Oh_X)$.

We give this vanishing a name; in characteristic $p$ it takes over the role that simple connectedness plays in characteristic zero.

\begin{definition}
\label{def:primitive}
A smooth proper algebraic space $X$ over $\bk$ is \emph{primitive} if $h^{0,1}(X)=\dim_\bk\HH^1(X,\Oh_X)=0$.
\end{definition}

A good reduction is simply connected by \Cref{prop:automatic-picard}, so \Cref{rem:h1-not-automatic} says exactly that it need not be primitive.
Two sufficient conditions are available, and they constrain different things.
\Cref{prop:h1-vanishing-small-e} deduces the vanishing from a bound on the absolute ramification index, a condition on the model $\mathscr X/\Oh_K$.
The dichotomy of \Cref{thm:all-good-dichotomy} deduces it instead from quasi-$F$-splitting, a condition on $X$ alone, and so sharpens \Cref{prop:automatic-picard} on the quasi-$F$-split locus.

\begin{corollary}
\label{cor:qF-Picard}
If a good reduction $X$ of dimension $2n\ge4$ is quasi-$F$-split, then $X$ is primitive.
\end{corollary}

\begin{proof}
By \Cref{thm:all-good-dichotomy}, $X$ is Frobenius split, so Frobenius is injective on $\HH^1(X,\Oh_X)$.
But it is nilpotent there by \Cref{prop:automatic-picard} and hence $\HH^1(X,\Oh_X)=0$.
\end{proof}

\Cref{cor:hodge-witt-dichotomy} does not identify which Witt-vector cohomology group fails to be finitely generated when $h_2^{\crys}(X)>1$.
Under $\HH^1(X,\Oh_X)=0$, the top group detects the dichotomy, and $\Phi_X^{2n}$ is prorepresentable by Serre duality and \Cref{lem:AM-prorepresentable}.

\begin{proposition}
\label{prop:top-witt-criterion}
Let $X$ be a good reduction of dimension $2n\ge4$, in the setting of \Cref{subsec:automatic}, and suppose $\HH^1(X,\Oh_X)=0$.
Then $X$ is quasi-$F$-split if and only if $\HH^{2n}(X,\Wit\Oh_X)$ is finitely generated over $\Wit$. Moreover, we have
\[ \hgt(X) = \hgt(\Phi^{2n}_X) \in \{ 1, \infty \}.\]
\end{proposition}

\begin{proof}
If $X$ is quasi-$F$-split, then $\HH^{2n}(X,\Wit\Oh_X)$ is finitely generated by \Cref{lem:basic-properties}(2); this direction uses no hypothesis on $\HH^1(X,\Oh_X)$.

Since $\omega_X\simeq\Oh_X$ and $X$ is smooth proper and geometrically connected (\Cref{prop:canonical-good}), Serre duality gives
\[
  \HH^{2n-1}(X,\Oh_X)\simeq\HH^1(X,\omega_X)^{\vee}=\HH^1(X,\Oh_X)^{\vee}=0,
  \qquad
  \HH^{2n}(X,\Oh_X)\simeq\HH^0(X,\omega_X)^{\vee}=\bk .
\]
Conversely, suppose $X$ is not quasi-$F$-split, so that $h_2^{\crys}(X)>1$ by \Cref{thm:all-good-dichotomy}.
For $m\geq 2$ the Witt sheaves sit in exact sequences
\[
  0\longrightarrow\Oh_X\xrightarrow{\ V^{m-1}\ }\Wit_m\Oh_X
   \xrightarrow{\ R\ }\Wit_{m-1}\Oh_X\longrightarrow0 .
\]
Induction on $m$ gives $\HH^{2n-1}(X,\Wit_m\Oh_X)=0$ for all $m\geq 1$, the long exact sequence squeezing this group between $\HH^{2n-1}(X,\Oh_X)=0$ and $\HH^{2n-1}(X,\Wit_{m-1}\Oh_X)=0$.
The same sequences therefore give short exact sequences
\[
  0\longrightarrow\HH^{2n}(X,\Oh_X)\longrightarrow\HH^{2n}(X,\Wit_m\Oh_X)
   \longrightarrow\HH^{2n}(X,\Wit_{m-1}\Oh_X)\longrightarrow0,
\]
surjective on the right because $\HH^{2n+1}$ vanishes.
Hence
\[
  \length_\Wit\HH^{2n}(X,\Wit_m\Oh_X)=m
\]
for every $m$.
The transition maps $R$ are surjective, so the inverse limit $\HH^{2n}(X,\Wit\Oh_X)$ surjects onto each $\HH^{2n}(X,\Wit_m\Oh_X)$ and has infinite length.

On the other hand $h_2^{\crys}(X)>1$ forces $(D_{2n})_{[0,1)}=0$ by \Cref{cor:automatic-top-slopes}, so \( \HH^{2n}(X,\Wit\Oh_X)[1/p]=0\) by \Cref{prop:WittSlope}.
A finitely generated torsion $\Wit$-module has finite length, so $\HH^{2n}(X,\Wit\Oh_X)$ is not finitely generated.

Finally, the vanishing of $\HH^{2n-1}(X,\Oh_X)$ and the isomorphism
$\HH^{2n}(X,\Oh_X)\simeq\bk$ show that $\Phi_X^{2n}$ is a smooth
one-dimensional formal group by \Cref{lem:AM-prorepresentable}.
By \Cref{cor:automatic-top-slopes}, the slope-$[0,1)$ part of $D_{2n}$
is a line of slope zero when $h_2^{\crys}(X)=1$, and vanishes when
$h_2^{\crys}(X)>1$.
Thus \Cref{rem:AM-slope-convention} gives $\hgt(\Phi_X^{2n})=1$ in the
first case and $\hgt(\Phi_X^{2n})=\infty$ in the second.
These are exactly the two values of $\hgt(X)$ in
\Cref{thm:all-good-dichotomy}.
\end{proof}

For these standard families \Cref{ex:hilb,ex:kummer}, non-ordinarity of the underlying surface forces infinite length in the top Witt-vector cohomology of the $2n$-dimensional variety.

\begin{corollary}\label{cor:generalized Kummer and Hilbert scheme no fg}
Let $\bk$ be a perfect field of characteristic $p$, let $n\ge2$, and let $X$ be either
\begin{enumerate}[label=\textup{(\alph*)},leftmargin=2.2em]
\item the Hilbert scheme $X=S^{[n]}$ of $n$ points on a $K3$ surface $S$ over $\bk$, with $p >2$; or
\item the generalised Kummer variety $X=K_n(A)$ attached to an abelian surface $A$ over $\bk$, with $p>2$ and $p\nmid n+1$.
\end{enumerate}
Then $X$ is smooth of dimension $2n\ge4$ and is a good reduction of a hyperk\"ahler variety.
If the underlying surface $S$, respectively $A$, is not ordinary, or equivalently if $h_2^{\crys}(X)>1$, then $X$ is not quasi-$F$-split and
\[
  \length_\Wit\HH^{2n}(X,\Wit\Oh_X)=\infty .
\]
In particular $\HH^{2n}(X,\Wit\Oh_X)$ is a $\Wit$-torsion module that is not finitely generated over $\Wit$.
\end{corollary}

\begin{proof}
In both cases $X$ is a smooth good reduction of a hyperk\"ahler variety of dimension $2n\ge4$ by the existence of algebraic lifting over $W(k)$.
\Cref{prop:h1-vanishing-small-e} gives $\HH^1(X,\Oh_X)=0$.
The slope-$[0,1)$ part of $D_2$ is that of the underlying surface (\Cref{ex:hilb,ex:kummer}), so $h_2^{\crys}(X)>1$ precisely when that surface is non-ordinary; then $X$ is not quasi-$F$-split by \Cref{thm:all-good-dichotomy}, and \Cref{prop:top-witt-criterion} applies.
\end{proof}

\section{Quasi-\texorpdfstring{$F$}{F}-split height in a family}
\label{sec:quasiFspliting in family}

Throughout this section $f\colon\cX\to S$ is a smooth proper morphism of noetherian $\bk$-schemes whose fibres are geometrically connected of dimension $d$, and
\begin{equation}\label{eq:family-setup}
  \omega_{\cX/S}\simeq\Oh_{\cX} .
\end{equation}
We impose even dimension $d=2n$, connectedness of $S$, and the vanishing
\begin{equation}\label{eq:family-h01}
  \HH^1(\cX_s,\Oh_{\cX_s})=0\qquad\text{for every }s\in S
\end{equation}
only where stated.
Let
\[
  S^{\HG},\qquad S^{\Fs},\qquad S^{\qF},\qquad S^{\fin}
\]
be the sets of $s\in S$ whose geometric fibre $\cX_{\bar s}$ is Hodge-good, Frobenius split, quasi-$F$-split, and of finite top Artin--Mazur height $\hgt(\Phi^{d}_{\cX_{\bar s}})<\infty$, respectively.
We use $S^{\fin}$ under \eqref{eq:family-h01}, or after restriction to $S^{\HG}$, so that the top Artin--Mazur functor is a smooth one-dimensional formal group (\Cref{lem:relative-AM-family}).
Geometric fibres allow us to apply the results over perfect fields even when $\kappa(s)$ is imperfect.
If $\kappa(s)$ is perfect, each condition can be tested on $\cX_s$ itself; this applies to closed points when $S$ is of finite type over $\bk$.
The top Artin--Mazur height is unchanged by field extension.

The first two subsections establish openness of the relevant loci and the height dichotomy on $S^{\HG}$, together with its consequences at the generic point.
We then pass from fibrewise to relative $\Wit_2$-liftings, use Deligne--Illusie to control Hodge-goodness under specialization, and deduce a dichotomy on the whole base.
The final subsection relates these results to primitive symplectic varieties.

\subsection{Openness and the top Artin--Mazur group}
\label{subsec:openness}

The Hodge-good and Frobenius-split loci are open.
Under \eqref{eq:family-h01}, the relative Artin--Mazur group also gives openness of $S^{\fin}$ and the inclusion $S^{\qF}\subseteq S^{\fin}$.

\begin{lemma}
\label{lem:HG-open}
Assume $d=2n$.
Then $S^{\HG}$ is open in $S$, and its formation commutes with base change on $S$.
\end{lemma}

\begin{proof}
The assertion is immediate for $n=0$, so assume $n\ge1$.
Hodge-goodness of $\cX_{\bar s}$ is equivalent to that of $\cX_s$, as \eqref{eq:HG} may be tested after the flat base change $\kappa(s)\to\kappa(\bar{s})$, and it is a condition on fibres; since $\cX\times_SS_{\mathrm{red}}\to S_{\mathrm{red}}$ has the same fibres as $f$, we may assume $S$ reduced.
Let $s_0\in S^{\HG}$; we produce an open neighbourhood of $s_0$ inside $S^{\HG}$.

Since $f$ is flat and proper, $s\mapsto\chi(\cX_s,\Oh_{\cX_s})$ is locally constant, and \eqref{eq:HG} gives $\chi(\cX_{s_0},\Oh_{\cX_{s_0}})=n+1$; shrink $S$ so that $\chi(\cX_s,\Oh_{\cX_s})=n+1$ for every $s$.
By the semicontinuity theorem each function $s\mapsto h^q(s)\coloneqq\dim_{\kappa(s)}\HH^q(\cX_s,\Oh_{\cX_s})$ is upper semicontinuous, so
\[
  V=\bigl\{\,s\in S \mid h^q(s)\le h^q(s_0)\ \text{for all } q\,\bigr\}
\]
is open and contains $s_0$.
On $V$ we have $h^q(s)=0$ for odd $q$ and $h^{2i}(s)\le1$ for $0\le i\le n$, whence
\[
  n+1=\chi(\cX_s,\Oh_{\cX_s})=\sum_{i=0}^{n}h^{2i}(s)\le n+1 .
\]
Equality forces $h^{2i}(s)=1$ for every $i$ and every $s\in V$.
Thus all the functions $h^q$ are constant on $V$; as $V$ is reduced, Grauert's theorem makes each $R^qf_*\Oh_{\cX}|_V$ locally free of that rank with formation commuting with arbitrary base change.
In particular $R^{2i}f_*\Oh_{\cX}|_V$ is invertible for $0\le i\le n$.

The top cup-power map
\[
  (R^2f_*\Oh_{\cX}|_V)^{\otimes n}
     \longrightarrow R^{2n}f_*\Oh_{\cX}|_V
\]
is a map of line bundles, so its non-vanishing locus is open and contains $s_0$.
On each fibre in this locus, a generator $\eta\in\HH^2(\cX_s,\Oh_{\cX_s})$ satisfies $\eta^n\ne0$, hence $\eta^i\ne0$ for every $0\le i\le n$.
These powers generate the one-dimensional even cohomology groups, proving \eqref{eq:HG} there.
The same fibrewise description proves compatibility with base change.
\end{proof}

\begin{lemma}
\label{lem:Fsplit-open}
The locus $S^{\Fs}$ is open in $S$, and equals $\{\,s\in S \mid \hgt(\cX_{\bar s})=1\,\}$.
\end{lemma}

\begin{proof}
Geometric connectedness gives $f_*\Oh_{\cX}=\Oh_S$.
Relative Serre duality and \eqref{eq:family-setup} therefore make $L=R^df_*\Oh_{\cX}$ invertible, with formation commuting with arbitrary base change and $L\otimes\kappa(s)=\HH^d(\cX_s,\Oh_{\cX_s})$.
The relative Frobenius $\cX\to\cX\times_{S,\Frob_S}S$ induces an $\Oh_S$-linear map $\Frob_S^*L\to L$ whose fibre at $s$ is the absolute Frobenius on $\HH^d(\cX_s,\Oh_{\cX_s})$.
A map of invertible sheaves is a section of an invertible sheaf, so its non-vanishing locus is open.
Non-vanishing can be tested after passage to $\kappa(\bar s)$, where \Cref{lem:Fsplit,def:quasi-f-split} identify it with $F$-splitting and with height one.
\end{proof}

The remaining two loci are governed by the Artin--Mazur formal group of the family, which exists as soon as \eqref{eq:family-h01} holds.

\begin{lemma}
\label{lem:relative-AM-family}
Assume \eqref{eq:family-h01}.
Then:
\begin{enumerate}[label=\textup{(\arabic*)},leftmargin=2.2em]
\item $R^{d-1}f_*\Oh_{\cX}=0$, and $R^df_*\Oh_{\cX}$ is an invertible $\Oh_S$-module whose formation commutes with arbitrary base change.
\item The relative Artin--Mazur functor $\Phi^d_{\cX/S}$ is prorepresentable by a smooth one-dimensional formal group $G$ over $S$, with $\Lie(G)\simeq R^df_*\Oh_{\cX}$ and $G\times_S\Spec\kappa(s)\simeq\Phi^d_{\cX_s}$ for every $s\in S$.
\end{enumerate}
\end{lemma}

\begin{proof}
(1) Each fibre is smooth proper and geometrically connected with $\omega_{\cX_s}\simeq\Oh_{\cX_s}$, so Serre duality gives
\[
  \HH^{d-1}(\cX_s,\Oh_{\cX_s})\simeq\HH^{1}(\cX_s,\omega_{\cX_s})^{\vee}
  =\HH^{1}(\cX_s,\Oh_{\cX_s})^{\vee}=0
\]
by \eqref{eq:family-h01}.
Cohomology and base change gives $R^{d-1}f_*\Oh_{\cX}\otimes\kappa(s)=0$ for every $s$ and hence $R^{d-1}f_*\Oh_{\cX}=0$ by Nakayama.
The assertion about $R^df_*\Oh_{\cX}$ was proved in \Cref{lem:Fsplit-open}.

(2) The fibrewise assertion is \Cref{lem:AM-prorepresentable} applied with $q=d$: the hypothesis $\HH^{d-1}(\cX_s,\Oh_{\cX_s})=0$ of part (2) there holds by (1), the obstruction group $\HH^{d+1}(\cX_s,\Oh_{\cX_s})$ vanishes for dimensional reasons, and the tangent space $\HH^d(\cX_s,\Oh_{\cX_s})$ is one-dimensional.
The relative statement is the same criterion of \cite[\S II]{ArtinMazur1977} applied to $f$: by (1) the sheaf $R^{d-1}f_*\Oh_{\cX}$ vanishes and $R^{d+1}f_*\Oh_{\cX}=0$, so $\Phi^d_{\cX/S}$ is prorepresentable and formally smooth over $S$ with tangent sheaf $R^df_*\Oh_{\cX}$, which is invertible; and its formation commutes with base change $\Spec\kappa(s)\to S$ because that of $R^df_*\Oh_{\cX}$ does.
\end{proof}

\begin{proposition}
\label{prop:height-semicontinuity}
Assume \eqref{eq:family-h01}.
The function $s\mapsto\hgt(\Phi^d_{\cX_s})$ is upper semicontinuous: for every $h\ge1$, the locus $\{\,s\in S:\hgt(\Phi^d_{\cX_s})\ge h\,\}$ is closed.
In particular, $S^{\fin}$ is open.
\end{proposition}

\begin{proof}
Let $G$ be the formal group of \Cref{lem:relative-AM-family}(2).
The question is local on $S$, so we may assume that $G$ admits a coordinate, that is $G\simeq\Oh_S[[t]]$ as a formal scheme, with a one-dimensional formal group law over $\Gamma(S,\Oh_S)$.
Write its multiplication-by-$p$ endomorphism as
\[
  [p]_G(t)=\sum_{i\ge1}a_it^{i},\qquad a_i\in\Gamma(S,\Oh_S).
\]
Its formation commutes with base change, so $[p]_{G_s}(t)=\sum_i a_i(s)t^i$ for every $s$.
Over a field of characteristic $p$ a one-dimensional formal group has height at least $h$ precisely when its $p$-series lies in $\kappa(s)[[t^{p^{h}}]]$, that is, precisely when $a_i(s)=0$ for every $i<p^{h}$.
Hence
\[
  \{~ s \mid \hgt(\Phi^d_{\cX_s})\ge h~\}=\bigcap_{i<p^{h}}\{\,s \mid a_i(s)=0\,\}
\]
is closed, and its complement $\{\hgt\le h-1\}$ is open.
Finally $S^{\fin}=\bigcup_{h\ge1}\{\hgt\le h\}$ is a union of open subsets.
\end{proof}

\begin{proposition}
\label{prop:qF-in-fin}
Assume \eqref{eq:family-h01}.
Then $S^{\qF}\subseteq S^{\fin}$.
\end{proposition}

\begin{proof}
Let $s\in S^{\qF}$ and work over the perfect field $\kappa(\bar{s})$.
By \Cref{lem:relative-AM-family}(2) the functor $\Phi^d_{\cX_{\bar s}}$ is a smooth one-dimensional formal group, and by \eqref{eq:family-h01} and Serre duality the hypotheses of \Cref{thm:nakkajima-inequality} hold in degree $q=d$: the tangent space $\HH^d(\cX_{\bar s},\Oh)$ is $\kappa(\bar{s})$, the obstruction group $\HH^{d+1}(\cX_{\bar s},\Oh)$ vanishes, and the Bockstein maps vanish because their source $\HH^{d-1}(\cX_{\bar s},\Oh)$ is zero.
That theorem gives $\hgt(\Phi^d_{\cX_{\bar s}})\le\hgt(\cX_{\bar s})<\infty$, so $s\in S^{\fin}$.
\end{proof}

\begin{remark}
\label{rem:openness-gap}
The reverse inclusion $S^{\fin}\subseteq S^{\qF}$ is not established here in general.
The next result gives equality on the Hodge-good locus in dimension at least four.
\end{remark}

\subsection{The dichotomy on the Hodge-good locus}
\label{subsec:dichotomy-family}

The fibrewise dichotomy of \Cref{cor:HG-qF-dichotomy} applies on $S^{\HG}$ without a lifting hypothesis or a bound on $p$.

\begin{theorem}
\label{thm:dichotomy in a family}
Assume $d=2n\ge4$.
Then
\[
  S^{\Fs}\cap S^{\HG}
    =S^{\qF}\cap S^{\HG}
    =S^{\fin}\cap S^{\HG}
\]
is open in $S$.
The height $\hgt(\cX_{\bar s})$ is $1$ on this common locus and $\infty$ on its complement in $S^{\HG}$.
\end{theorem}

\begin{proof}
For $s\in S^{\HG}$, the geometric fibre $\cX_{\bar s}$ is Hodge-good, smooth proper and geometrically connected, with trivial canonical bundle and dimension $2n\ge4$.
Moreover, \eqref{eq:family-h01} holds on $S^{\HG}$ by \eqref{eq:HG}, so its top Artin--Mazur group is defined by \Cref{lem:relative-AM-family}.
Applying \Cref{cor:HG-qF-dichotomy} over $\kappa(\bar s)$ gives the equalities and the height assertion.
Openness follows from \Cref{lem:HG-open,lem:Fsplit-open}.
\end{proof}

In particular, the dichotomy holds near any Hodge-good fibre, including any Hodge-good reduction of a hyperk\"ahler variety (\Cref{def:hodge-good-reduction}).

\begin{theorem}
\label{thm:generic-dichotomy}
Assume $d=2n\ge4$, $S$ is irreducible with generic point $\eta$, $S^{\HG}\ne\emptyset$, and \eqref{eq:family-h01} holds.
\begin{enumerate}[label=\textup{(\roman*)},leftmargin=2.2em]
\item $S^{\fin}\ne\emptyset$ if and only if $S^{\qF}\ne\emptyset$, if and only if the geometric generic fibre is $F$-split.
\item If the geometric generic fibre is not $F$-split, then $S^{\fin}=S^{\qF}=\emptyset$ and $\hgt(\cX_{\bar s})=\infty$ for every $s\in S$.
\item The dichotomy $\hgt(\cX_{\bar s})\in\{1,\infty\}$ holds on the dense open subset $S^{\HG}\cup S^{\Fs}$.
\end{enumerate}
\end{theorem}

\begin{proof}
(i) If $S^{\fin}\ne\emptyset$, both $S^{\fin}$ and $S^{\HG}$ are nonempty open subsets by \Cref{prop:height-semicontinuity,lem:HG-open}.
They contain $\eta$, so \Cref{thm:dichotomy in a family} gives $\eta\in S^{\Fs}$.
The other implications follow from $S^{\Fs}\subseteq S^{\qF}\subseteq S^{\fin}$, using \Cref{lem:basic-properties}(1) and \Cref{prop:qF-in-fin}.

(ii) By (i), $S^{\fin}=S^{\qF}=\emptyset$, so every geometric fibre has infinite quasi-$F$-split height by definition.

(iii) Apply \Cref{thm:dichotomy in a family} on $S^{\HG}$ and \Cref{lem:Fsplit-open} on $S^{\Fs}$.
Their union is open and contains the nonempty open subset $S^{\HG}$, hence is dense in the irreducible base $S$.
\end{proof}

\paragraph{Outside the Hodge-good locus.}
Under the hypotheses of \Cref{thm:generic-dichotomy}, a fibre $Y=\cX_{\bar s}$ with $s\notin S^{\HG}$ is constrained in two ways.
First, semicontinuity from the Hodge-good generic fibre gives $h^{2i}(Y,\Oh_Y)=1+a_i$ and $h^{2j+1}(Y,\Oh_Y)=b_j$ with $a_i,b_j\ge0$.
Constancy of $\chi(Y,\Oh_Y)=n+1$ then gives
\begin{equation}\label{eq:paired-jumps}
  \sum_i a_i=\sum_j b_j ,
\end{equation}
so an even jump is accompanied by an odd one.
Hodge-goodness can also fail without a dimension jump, through vanishing of the top cup power detected in \Cref{lem:HG-open}.

Second, if $Y$ is quasi-$F$-split, then all $\HH^j(Y,\Wit\Oh_Y)$ are finitely generated by \Cref{lem:basic-properties}(2), and \Cref{prop:Witt-Euler} gives
\[
  \sum_j(-1)^j\slopedefect_j(Y)=n+1,
  \qquad 0\le\slopedefect_j(Y)\le h^j(Y,\Oh_Y).
\]
If the coherent dimensions do not jump, these relations force $\slopedefect_{2i}(Y)=1$ for every $i$, and every $\Phi^{2i}_Y$ has finite height.
They do not control the cup powers.
Thus we cannot exclude a fibre outside $S^{\HG}$ that is quasi-$F$-split but not $F$-split; by \Cref{thm:generic-dichotomy}(ii), such a fibre can occur only when the geometric generic fibre is already $F$-split.

\subsection{From fibrewise to relative $\Wit_2$-liftings}
\label{subsec:relative-W2-liftings}

The family argument requires a relative $\Wit_2$-lifting.
Condition \eqref{eq:ct2} makes the fibrewise lifting obstructions into a section of a vector bundle on the base; its vanishing on a dense open subset then gives local liftings of the whole family.

In the following, $T_{\cY/B}$ denotes the relative tangent sheaf and $T_{\cY_b}$ its restriction to a fibre.

\begin{lemma}
\label{lem:relative-W2-lifting}
Let $B$ be a smooth finite-type $\bk$-scheme and let $g\colon\cY\to B$ be smooth and proper.
Assume that
\begin{equation}\label{eq:ct2}
  b\longmapsto\dim_{\kappa(b)}\HH^2(\cY_b,T_{\cY_b})
  \tag{$\textup{ct}_2$}
\end{equation}
is locally constant on the set of closed points of $B$.
Then $R^2g_*T_{\cY/B}$ is locally free and its formation commutes with arbitrary base change, and for every affine open $U\subseteq B$ and every smooth affine $\Wit_2(\bk)$-lifting $\widetilde U$ of $U$ there is a class
\[
  o\in\Gamma\bigl(U,R^2g_*T_{\cY/B}\bigr)
\]
with the following two properties.
\begin{enumerate}[label=\textup{(\roman*)},leftmargin=2.2em]
\item $o=0$ if and only if $\cY_U$ admits a smooth proper lifting over $\widetilde U$.
\item For every closed point $b\in U$, the value $o(b)$ is the obstruction to lifting $\cY_b$ over $\Wit_2(\kappa(b))$.
  Thus the non-liftable closed fibres are the closed points of an open subset of $U$.
\end{enumerate}
Consequently, if the closed fibres over a dense open subset of $B$ lift to $\Wit_2(\kappa(b))$, then locally on $B$ the morphism $g$ admits a smooth proper lifting over a smooth $\Wit_2(\bk)$-lifting of its base.
\end{lemma}

\begin{proof}
The sheaf $T_{\cY/B}$ is locally free and $B$-flat, so $b\mapsto\dim\HH^2(\cY_b,T_{\cY_b})$ is upper semicontinuous on $B$.
A finite-type $\bk$-scheme is Jacobson, so a closed subset meeting no closed point is empty; hence local constancy on closed points forces local constancy on $B$.

Fix $b\in B$ and put $A=\Oh_{B,b}$, a regular local domain.
By properness $Rg_*T_{\cY/B}$ is a perfect complex whose formation commutes with derived base change; choose a minimal finite free model $P^\bullet$ of it over $A$, so that the differentials of $P^\bullet$ have entries in the maximal ideal and $\dim\HH^2(\cY_b,T_{\cY_b})=\rank P^2$.
By the local constancy just established the same number computes the second cohomology of $P^\bullet$ at the generic point of $\Spec A$, so the two differentials adjacent to degree two vanish there; as $A$ is a domain and $P^\bullet$ is free, they vanish.
Hence $\HH^2(P^\bullet)=P^2$, and $R^2g_*T_{\cY/B}$ is locally free with formation commuting with arbitrary base change.

Let $U\subseteq B$ be an affine open subset and $\widetilde U$ a smooth affine $\Wit_2(\bk)$-lifting of $U$, which exists by lifting an \'etale coordinate presentation.
Flatness of $\widetilde U$ over $\Wit_2(\bk)$ identifies the square-zero ideal $p\Oh_{\widetilde U}$ with $\Oh_U$ through multiplication by $p$, so the obstruction to extending $\cY_U\to U$ to a flat lifting over $\widetilde U$ is a class
\[
  o\in\operatorname{Ext}^2_{\cY_U}\bigl(\Omega^1_{\cY_U/U},\Oh_{\cY_U}\bigr)
   =\HH^2\bigl(\cY_U,T_{\cY_U/U}\bigr)
   =\Gamma\bigl(U,R^2g_*T_{\cY/B}\bigr),
\]
the first equality because smoothness identifies the relative cotangent complex with $\Omega^1_{\cY_U/U}$; the class vanishes if and only if such a lifting exists \cite[Chapter~III, \S2.1]{Illusie1971}.
A flat lifting is automatically smooth and proper, both properties persisting across a nilpotent thickening of the base, so this is \textup{(i)}.

For \textup{(ii)}, smoothness of $\widetilde U$ over $\Wit_2(\bk)$ lifts a closed point $b\in U$ to a $\Wit_2(\kappa(b))$-point of $\widetilde U$, the residue field $\kappa(b)$ being finite over $\bk$ and hence perfect.
By the base-change property established above, together with naturality of the obstruction along that morphism of square-zero extensions, the value $o(b)$ is the obstruction to lifting $\cY_b$ over $\Wit_2(\kappa(b))$.
The non-vanishing locus of a section of a locally free sheaf is open, which gives the last assertion of \textup{(ii)}.

If the closed fibres lift over a dense open subset $V\subseteq B$, then $o$ vanishes at every closed point of $U\cap V$.
Since $B$ is reduced and Jacobson, $o$ vanishes on $U\cap V$, hence on $U$.
Thus $o=0$ and \textup{(i)} applies.
\end{proof}

\begin{corollary}\label{cor:relative-W2-from-one-fibre}
Under the hypotheses of \Cref{lem:relative-W2-lifting}, suppose $B$ is connected and one geometric fibre is connected, has trivial canonical bundle, has $E_1$-degeneration, and has torsion-free crystalline cohomology in every degree.
Then, locally on $B$, the morphism $g$ admits a smooth proper lifting over a smooth $\Wit_2(\bk)$-lifting of its base.
\end{corollary}

\begin{proof}
Let $Y_0$ be the specified geometric fibre and write $d=\dim Y_0$.
Smooth proper base change makes all geometric fibres connected and their Betti numbers $b_j$ constant.
The assumptions on $Y_0$ give $\sum_{a+b=j}h^{a,b}(Y_0)=b_j$ for every $j$.
By upper semicontinuity there is an open neighbourhood $V$ of its image on which $h^{a,b}(\cY_{\bar t})\le h^{a,b}(Y_0)$ for all $a,b$.
For $t\in V$,
\[
 b_j\le\dim\HH^j_{\dR}(\cY_{\bar t})
     \le\sum_{a+b=j}h^{a,b}(\cY_{\bar t})\le b_j.
\]
The first inequality follows from the crystalline universal-coefficient sequence.
Equality throughout gives constant Hodge numbers on $V$, $E_1$-degeneration, and torsion-free crystalline cohomology in every degree.
In particular $g_*\omega_{\cY_V/V}$ is a line bundle and commutes with base change.
Its evaluation map is an isomorphism on $Y_0$.
The locus in $\cY_V$ where it is not an isomorphism is closed and has closed image in $V$ by properness.
Removing that image leaves an open neighbourhood of the image of $Y_0$ on which every fibre has trivial canonical bundle.
By \cite[Theorem~7.18]{BrantnerTaelman2025}, all closed fibres over this open subset lift to $\Wit_2$.
Since $B$ is smooth and connected, it is integral, so this open subset is dense and \Cref{lem:relative-W2-lifting} applies.
\end{proof}

\subsection{Constancy of Hodge-goodness under relative lifting}
\label{subsec:HG-constancy}

Relative $\Wit_2$-liftings control Hodge cohomology in degrees below $p$.
For varieties with trivial canonical bundle, Serre duality extends this control to every structure-sheaf degree and, when $n<p$, to the cup powers needed for Hodge-goodness.

\goodbreak
\begin{samepage}
\begin{proposition}
\label{prop:HG-locally-constant}
Let $B$ be a smooth connected finite-type $\bk$-scheme and let $g\colon\cY\to B$ be smooth proper of relative dimension $2n$, with $n\ge1$.
Suppose that, locally on $B$, the morphism $g$ admits a smooth lifting over a smooth $\Wit_2(\bk)$-lifting of its base.
Then the following hold.
\begin{enumerate}[label=\textup{(\roman*)},leftmargin=2.2em]
\item For $a+b<p$, the sheaves $R^bg_*\Omega^a_{\cY/B}$ are locally free and commute with arbitrary base change.
  For $1\le r\le n$ with $2r<p$, both maps
\[
    (R^2g_*\Oh_{\cY})^{\otimes r}\longrightarrow R^{2r}g_*\Oh_{\cY},
    \qquad
    (g_*\Omega^2_{\cY/B})^{\otimes r}\longrightarrow g_*\Omega^{2r}_{\cY/B}
  \]
  have locally constant rank.
  \item If $2n<p$, the local-freeness and base-change assertions hold for all $a,b$, and the locus of Hodge-good geometric fibres is open and closed.
  \item Suppose every geometric fibre has trivial canonical bundle.
    If $n\le p$, all $R^qg_*\Oh_{\cY}$ are locally free and commute with arbitrary base change.
    If $n<p$, the locus of Hodge-good geometric fibres is either empty or all of $B$.
\end{enumerate}
\end{proposition}
\end{samepage}

\begin{proof}
Work over an affine open $U\subseteq B$ carrying a lifting $\widetilde{\cY}\to\widetilde U$ of $\cY_U\to U$ with $\widetilde U$ smooth over $\Wit_2(\bk)$.
We first obtain the low-degree Deligne--Illusie comparison.
Smoothness of $\widetilde U$ provides a lift of the absolute Frobenius of $U$, semilinear for the Witt Frobenius; base changing $\widetilde{\cY}$ along it lifts the Frobenius twist $\cY'=\cY_U\times_{U,\Frob_U}U$.
It is this lifting of $\cY'$, and not a lifting of individual fibres, that the relative form of the Deligne--Illusie theorem consumes \cite[Corollary~3.7(a), Remark~4.1.6]{DeligneIllusie1987}.
Without any dimension bound, the resulting decomposition is
\[
  \phi\colon\bigoplus_{0\le a<p}\Omega^a_{\cY'/U}[-a]
  \xrightarrow{\ \sim\ }\tau_{<p}\Frob_*\Omega^\bullet_{\cY_U/U}
  \qquad\text{in }D(\cY'),
\]
$\Frob$ denoting the relative Frobenius, and $\phi$ may be chosen compatibly with products in total form degree less than $p$, that is $\phi_{a+b}(u\wedge v)=\phi_a(u)\phi_b(v)$ for $a+b<p$: one takes $\phi_1$, forms its tensor powers, antisymmetrises by $1/a!$, which is legitimate exactly in the range $a<p$, and composes with the multiplication of the de Rham complex \cite[proof of Theorem~2.1(a) and Corollary~3.7(a)]{DeligneIllusie1987}.
By \cite[Corollary~4.1.4]{DeligneIllusie1987} the sheaves $E^{a,b}=R^bg_*\Omega^a_{\cY_U/U}$ and $D^m=R^mg_*\Omega^\bullet_{\cY_U/U}$ are locally free and commute with arbitrary base change for $a+b<p$ and $m<p$, respectively, and the relative Hodge spectral sequence degenerates in these total degrees.
Since the omitted truncation has hypercohomology only in degrees at least $p$, $\phi$ induces isomorphisms \footnote{These need not preserve the Hodge filtration.}
\[
  \theta_m\colon \Frob_U^*\Bigl(\bigoplus_a E^{a,m-a}\Bigr)
  \xrightarrow{\ \sim\ }D^m,\qquad m<p.
\]

For $1\le r\le n$ with $2r<p$, let $u_r\colon(D^2)^{\otimes r}\to D^{2r}$ be the cup product of relative de Rham cohomology.
It is a filtered map of vector bundles for the Hodge filtration, whose associated graded $v_r$ is the direct sum, over the total form degree, of the Hodge cup products on $\bigl(\bigoplus_{a=0}^2E^{a,2-a}\bigr)^{\otimes r}$.
Every product occurring here has total form degree at most $2r<p$, so multiplicativity of $\phi$ gives $u_r\circ\theta_2^{\otimes r}=\theta_{2r}\circ \Frob_U^*v_r$.

To show that each graded block of $v_r$ has locally constant rank, fix $b\in U$ and put $A=\Oh_{U,b}$, with maximal ideal $\mathfrak m$.
Over $A$, write $u\colon E\to G$ for $u_r$ and $v=\gr u$.
The Hodge filtrations split because their graded pieces are free, so we may write
\[
  E=\bigoplus_{i=0}^{2r}E_i,\qquad
  G=\bigoplus_{i=0}^{2r}G_i,
  \qquad
  \Fil^iE=\bigoplus_{j\ge i}E_j,\quad
  \Fil^iG=\bigoplus_{j\ge i}G_j,
\]
where $E_i$ and $G_i$ identify with the graded pieces.
Choose bases ordered by increasing $i$.
Since $u(\Fil^iE)\subseteq\Fil^iG$, its component $E_i\to G_j$ vanishes for $j<i$.
Thus, writing $M_i$ for the matrix of $v_i\colon E_i\to G_i$, we have
\[
  M_u=
  \begin{pmatrix}
    M_0&0&\cdots&0\\
    *&M_1&\ddots&\vdots\\
    \vdots&\ddots&\ddots&0\\
    *&\cdots&*&M_{2r}
  \end{pmatrix},
  \qquad
  M_v=
  \begin{pmatrix}
    M_0&0&\cdots&0\\
    0&M_1&\ddots&\vdots\\
    \vdots&\ddots&\ddots&0\\
    0&\cdots&0&M_{2r}
  \end{pmatrix}.
\]
Each $M_i$ has $\rank G_i$ rows and $\rank E_i$ columns and may be rectangular.

For a map $w$ of finite free modules, let $I_j(w)$ be the ideal generated by its $j\times j$ minors; it is independent of bases by the Cauchy--Binet formula.
A nonzero minor of $M_v$ selects equally many rows and columns in each block, giving square submatrices $N_i$ of $M_i$.
The same rows and columns in $M_u$ give a block lower triangular matrix with diagonal blocks $N_i$.
Both determinants are $\prod_i\det N_i$, so
\[
  I_j(v)\subseteq I_j(u).
\]

Write $M_v^{(p)}$ for the matrix obtained from $M_v$ by raising every entry to its $p$-th power.
The identity $u_r\circ\theta_2^{\otimes r}=\theta_{2r}\circ\Frob_U^*v_r$ gives
\[
  M_u=P M_v^{(p)}Q^{-1}.
\]
Here $P$ and $Q$ represent $\theta_{2r}$ and $\theta_2^{\otimes r}$.
These invertible matrices preserve determinantal ideals, although they need not preserve the filtrations.
Since taking minors commutes with raising entries to their $p$-th powers,
\[
  I_j(v)\subseteq I_j(u)=I_j(v)^{[p]},
  \qquad J^{[p]}\coloneqq (x^p\mid x\in J).
\]

Let $\rho=\rank_{\kappa(b)}v(b)$ and $I=I_{\rho+1}(v)$.
Then $I\subseteq\mathfrak m$, and $x^p=x^{p-1}x\in\mathfrak m I$ for $x\in I$.
Thus
\[
  I\subseteq I^{[p]}\subseteq\mathfrak m I\subseteq I.
\]
Nakayama's lemma gives $I=0$.
Thus $v$ has constant rank $\rho$ over $A$ and thus a Zariski neighborhood of $b$.
For each diagonal block, put $\rho_i=\rank v_i(b)$.
A $\rho_i$-minor nonzero at $b$ remains nonzero nearby, so after shrinking $U$ we have $\rank v_i(s)\ge\rho_i$ for every $i$ and $s \in U$; for $\rho_i=0$ this is automatic.
But
\[
  \sum_i\rank v_i(s)=\rank v(s)=\rho=\sum_i\rho_i,
\]
so each $\rank v_i(s)=\rho_i$.
Hence every graded block has locally constant rank.
The blocks $M_0$ and $M_{2r}$ of form degree zero and $2r$ are precisely the coherent cup-power and wedge-power maps in \textup{(i)}, respectively, proving their rank constancy.

If $2n<p$, the Deligne--Illusie decomposition is one of the whole de Rham complex, so \cite[Corollary~4.1.5]{DeligneIllusie1987} gives the assertion about all Hodge sheaves.
By \eqref{eq:HG} a geometric fibre $\cY_{\bar b}$ is Hodge-good exactly when $\HH^q(\cY_{\bar b},\Oh)=0$ for odd $q$, $\dim\HH^{2i}(\cY_{\bar b},\Oh)=1$ for $0\le i\le n$, and the $i$-th cup-power map has rank one for $1\le i\le n$: granted the first two conditions, a nonzero $\eta\in\HH^2(\cY_{\bar b},\Oh)$ satisfies $\eta^i\ne0$ precisely when that rank is one, and $\eta^i$ then generates $\HH^{2i}(\cY_{\bar b},\Oh)$.
These dimensions and ranks are locally constant by \textup{(i)}, proving \textup{(ii)}.

For \textup{(iii)}, Serre duality on every geometric fibre $Y$ gives $h^q(Y,\Oh_Y)=h^{2n-q}(Y,\Oh_Y)$.
If $n<p$, at least one of $q$ and $2n-q$ is less than $p$, so \textup{(i)} makes every coherent cohomology dimension locally constant.
If $n=p$, the same argument covers all degrees except $q=p$, and constancy of $\chi(Y,\Oh_Y)$ covers the remaining degree.
Since $B$ is reduced, cohomology and base change now give local freeness and arbitrary base change for every $R^qg_*\Oh_{\cY}$.

Assume henceforth $n<p$ and one geometric fibre is Hodge-good.
The dimensions are then zero in odd degrees and one in even degrees on every geometric fibre.
For $n=1$ these dimensions already give Hodge-goodness.
For $n\ge2$, put $a=\lfloor n/2\rfloor$ and $b=\lceil n/2\rceil$.
Both $2a$ and $2b$ are less than $p$: for odd $n$, this uses that $p$ is odd, so $n<p$ implies $n+1<p$.
By \textup{(i)}, both low cup-power maps have rank one on every fibre, as they do on the Hodge-good fibre.
Thus, for any nonzero $\eta\in\HH^2(Y,\Oh_Y)$, the classes $\eta^a$ and $\eta^b$ generate $\HH^{2a}(Y,\Oh_Y)$ and $\HH^{2b}(Y,\Oh_Y)$.
After choosing a trivialisation of $\omega_Y$, Serre duality identifies cup product
\[
  \HH^{2a}(Y,\Oh_Y)\otimes\HH^{2b}(Y,\Oh_Y)
       \longrightarrow\HH^{2n}(Y,\Oh_Y)
\]
with a perfect pairing of one-dimensional spaces.
Hence $\eta^n\ne0$, which forces every $\eta^i$, $0\le i\le n$, to be nonzero.
These powers generate all coherent cohomology, so $Y$ is Hodge-good.
On the connected base $B$, the Hodge-good locus is therefore either empty or all of $B$.
\end{proof}

\subsection{The dichotomy on the whole base}
\label{subsec:HG-to-all}

The relative lifting criterion and the constancy of Hodge-goodness now extend the dichotomy beyond $S^{\HG}$.

\begin{theorem}
\label{cor:dichotomy-all-of-S}
Let $S$ be smooth connected of finite type over $\bk$ and $d=2n$ with $2\le n<p$.
Assume $S^{\HG}\ne\emptyset$ and that \eqref{eq:family-h01} and \eqref{eq:ct2} hold.
Then $\hgt(\cX_{\bar s})\in\{1,\infty\}$ for every $s\in S$.

If every closed fibre lifts to $\Wit_2(\kappa(s))$, then $S^{\HG}=S$ and
\[
  S^{\Fs}=S^{\qF}=S^{\fin}
\]
is open in $S$.
This lifting condition holds whenever $S^{\fin}\ne\emptyset$.
\end{theorem}

\begin{proof}
Since $S$ is smooth and connected, it is irreducible.
Suppose first that every closed fibre lifts.
By \eqref{eq:ct2} and \Cref{lem:relative-W2-lifting}, the family admits smooth relative $\Wit_2$-liftings locally on $S$.
Since $n<p$ and the fibres have trivial canonical bundle, \Cref{prop:HG-locally-constant}\textup{(iii)} propagates Hodge-goodness from one fibre to all of $S$.
The asserted equality of open loci and the height dichotomy follow from \Cref{thm:dichotomy in a family}.

If $S^{\fin}\ne\emptyset$, the subset $V=S^{\fin}\cap S^{\HG}$ is dense and open by \Cref{prop:height-semicontinuity,lem:HG-open}.
Its closed fibres are $F$-split by \Cref{thm:dichotomy in a family}, hence lift to $\Wit_2(\kappa(s))$ by \Cref{prop:quasi F split is W2 litable}.
The dense-open assertion of \Cref{lem:relative-W2-lifting} gives local relative liftings on all of $S$; in particular, every closed fibre lifts, and the preceding paragraph applies.

Finally, if $S^{\fin}=\emptyset$, then $S^{\qF}=\emptyset$ by \Cref{prop:qF-in-fin}, so $\hgt(\cX_{\bar s})=\infty$ for every $s\in S$.
\end{proof}

\begin{remark}
\label{rem:family-hypotheses}
In the proof of \Cref{cor:dichotomy-all-of-S}, condition \eqref{eq:ct2} makes $R^2f_*T_{\cX/S}$ locally free and compatible with base change, so the lifting obstruction is a section of a vector bundle.
The bound $n<p$ is used only for the low-degree decomposition and the middle Serre pairing in \Cref{prop:HG-locally-constant}\textup{(iii)}.
It does not give the all-degree assertions about Hodge sheaves or wedge powers in \textup{(ii)} of that proposition.
At $n=p$, the coherent dimensions still remain constant, but the cup product from degree $p-1$ to degree $p+1$ is outside the range controlled by \textup{(i)}.
We do not know whether Hodge-goodness remains constant there, or whether \eqref{eq:ct2} can be removed from the family theorem.
\end{remark}

\subsection{Primitive symplectic varieties}
\label{subsec:primitive-symplectic}

For applications, we combine Hodge-goodness with a nondegenerate $2$-form.

\begin{definition}
\label{def:primitive-symplectic}
A smooth proper Hodge-good variety $X$ over $\bk$ of dimension $2n$ is \emph{primitive symplectic} if $\HH^0(X,\Omega^2_X)$ is one-dimensional and spanned by a nowhere-degenerate $2$-form. Here being nowhere-degenerate means the induced $\Oh_X$-linear morphism $T_X \to \Omega_{X}^1$ is an isomorphism.
\end{definition}

Hodge-goodness makes the odd structure-sheaf cohomology vanish, so a primitive symplectic variety is primitive in the sense of \Cref{def:primitive}; this is what the name records.

\begin{remark}
\label{rem:pfaffian-volume}
There is also a linear-algebra distinction in small characteristic.
A differential $2$-form defines an alternating pairing even in characteristic $2$, so nondegeneracy always forces even dimension.
Its Pfaffian volume form trivialises $\omega_X$: in a local symplectic coframe $e_1,f_1,\ldots,e_n,f_n$ with $\sigma=\sum_i e_i\wedge f_i$, this volume form is $e_1\wedge f_1\wedge\cdots\wedge e_n\wedge f_n$.
The ordinary exterior power is $n!$ times this volume form; it therefore vanishes when $p\le n$, and gives the trivialization of canonical bundle only when $p>n$.
\end{remark}

By \Cref{rem:pfaffian-volume,cor:HG-qF-dichotomy}, a primitive symplectic variety of dimension $2n\ge4$ is quasi-$F$-split if and only if it is Frobenius split.
Together with \Cref{lem:HG-open}, this proves \Cref{thm:main-family}.
The isomorphism $T_X\simeq\Omega^1_X$ induced by the nowhere-degenerate $2$-form implies that, in a family of primitive symplectic varieties, \eqref{eq:ct2} is constancy of $h^{1,2}$, as used for the standard families in \Cref{lem:ct2 for relative Kummer}.

Over an algebraically closed field, Fu--Li \cite[Definition~3.1]{FuLi2021} call a connected smooth projective variety \emph{irreducible symplectic} if its \emph{\'etale} fundamental group is trivial and $\HH^0(X,\Omega_X^2)=\bk\sigma$ for a closed, nowhere-degenerate $2$-form $\sigma$.
In characteristic $p$, \'etale simple connectedness does not imply $\HH^1(X,\Oh_X)=0$, and Hodge symmetry cannot be used to identify the conditions on $\HH^0(X,\Omega_X^2)$ and $\HH^2(X,\Oh_X)$.
Our definition imposes the full coherent cohomology algebra \eqref{eq:HG}, including its cup products.
It is a working definition adapted to the arguments here: it requires only properness, and neither \'etale simple connectedness nor closedness of $\sigma$ is explicitly imposed.
Closedness follows from $E_1$-degeneration of the Hodge--de Rham spectral sequence, since its differential $\HH^0(X,\Omega_X^2)\to\HH^0(X,\Omega_X^3)$ sends $\sigma$ to $\mathrm{d}\sigma$.
We do not assert an equivalence with Fu--Li's notion without additional hypotheses.

Srivastava's examples in \cite{Srivastava2021} make the distinction concrete.
A supersingular Enriques surface $E$ in characteristic $2$ has trivial \'etale fundamental group and $\HH^0(E,\Omega_E^2)=\bk\sigma$ with $\sigma$ nowhere degenerate, but $\HH^1(E,\Oh_E)=\bk$.
Thus it is irreducible symplectic in Fu--Li's sense (see also \cite[Example~3.3(i)]{FuLi2021}), whereas Hodge-goodness excludes it from \Cref{def:primitive-symplectic}.
In dimension two the latter definition recovers the usual $K3$ condition $\omega_X\simeq\Oh_X$ and $\HH^1(X,\Oh_X)=0$.
For $m\ge2$, Srivastava shows that $E^{[m]}$ is simply connected and symplectic, but $h^{2,0}(E^{[m]})>1$, so these Hilbert schemes already fail the one-dimensionality condition.
The same paper discusses deformations of a supersingular Enriques surface to classical Enriques surfaces, where the canonical bundle becomes nontrivial.
These phenomena show why the characteristic-zero behaviour under deformations and Hilbert schemes cannot be inferred from the naive conditions alone.

Hodge-goodness supplies the stronger cohomological input needed for the Artin--Mazur groups and the height dichotomy, but does not by itself settle all the geometric requirements of a satisfactory positive-characteristic analogue.
In particular, preservation in families is a separate assertion: we prove it under the hypotheses of \Cref{prop:HG-locally-constant}, and control the symplectic form for the standard deformation classes in \Cref{thm:deformation type of hilbert and kummer}.

\section{Hilbert schemes, Kummer varieties, and their deformations}
\label{sec:examples}

We apply the preceding results to Hilbert schemes of $K3$ surfaces and generalised Kummer varieties.
We first compute their coherent cohomology algebras in the tame range, then construct relative symplectic families and verify \eqref{eq:ct2}.
These inputs, together with the integral BBF pairings in \Cref{sec:appendix}, allow us to propagate primitive symplecticity and the height dichotomy along Hodge-deformations.
We conclude by explaining the different behaviour in dimension two.

\subsection{Hodge-goodness in the tame range}
\label{subsec:HG-examples}

For the two standard series, the coherent cohomology algebra can be computed using the Hilbert--Chow morphism and invariants under a symmetric group.
The bounds $p>n$ for $S^{[n]}$ and $p>n+1$ for $K_n(A)$ make the relevant group order invertible and ensure that the degree-two class generates the algebra.

\begin{proposition}[Tame Hilbert schemes]
\label{prop:hilbert-HG}
Let $S$ be a $K3$ surface over a perfect field of characteristic $p$.
If $p>n$, then $S^{[n]}$ is Hodge-good.
\end{proposition}

\begin{proof}
We may extend the ground field.
For the Hilbert--Chow morphism $\rho\colon S^{[n]}\to S^{(n)}=S^n/\mathfrak S_n$ the target admits the finite cover $S^n\to S^{(n)}$ of degree $n!$, which is prime to $p$, so $R\rho_*\Oh_{S^{[n]}}=\Oh_{S^{(n)}}$ by \cite[Theorem~3.2.14]{ChatzistamatiouRulling2011}.
Exactness of $\mathfrak S_n$-invariants then gives an isomorphism of graded algebras
\[
  \HH^\bullet(S^{[n]},\Oh)\simeq\HH^\bullet(S^n,\Oh)^{\mathfrak S_n} .
\]
Write $x_j$ for the degree-two generator coming from the $j$-th factor.
Then $x_j^2=0$, the invariants in degree $2i$ are spanned by the elementary symmetric function $e_i(x_1,\dots,x_n)$, and $e_1^i=i!\,e_i$.
These scalars are units because $p>n$, so \eqref{eq:HG} holds with $\eta=e_1$.
\end{proof}

\begin{proposition}[Tame generalised Kummer varieties]
\label{prop:kummer-HG}
Let $A$ be an abelian surface over a perfect field of characteristic $p$, and let $K_n(A)$ be the fibre over zero of $A^{[n+1]}\to A$.
If $p>n+1$, then $K_n(A)$ is smooth and Hodge-good.
\end{proposition}

\begin{proof}
Hodge-goodness and smoothness are preserved and detected by field extension, so we may assume that the ground field is algebraically closed.
Put $m=n+1$, $G=\mathfrak S_m$, $B=\ker(A^m\xrightarrow{+}A)$, and $X_0=B/G$.
In characteristic zero, rationality of quotient singularities and the restricted Hilbert--Chow morphism $K_n(A)\to X_0$ give
\begin{equation}\label{eq:kummer-coherent-invariants}
  \HH^\bullet(K_n(A),\Oh)
    \simeq \HH^\bullet(B,\Oh_B)^G.
\end{equation}
This is the $(0,\bullet)$-part of the more general calculation in \cite[Theorem~7]{GottscheSoergel1993}.
If $V=\HH^1(A,\Oh_A)$ and $W_0=\ker(\bk^m\xrightarrow{\sum}\bk)$, then the right-hand side of \eqref{eq:kummer-coherent-invariants} is $\bigl(\bigwedge^\bullet(V\otimes W_0)\bigr)^G$.
In characteristic zero the anticommutative Molien formula \cite[\S2.2, Exercise~(5)]{Sturmfels2008} gives
\[
 \sum_q\dim\bigl(\textstyle\bigwedge^q(V\otimes W_0)\bigr)^G t^q
 =\frac1{m!}\sum_{g\in G}\det(1+tg\mid W_0)^2
 =1+t^2+\cdots+t^{2n}.
\]
The degree-two invariant $\eta$ defined by the inverse of the standard symmetric form on $W_0$ satisfies $\eta^n=\pm(n!/m)\operatorname{vol}$, so its powers generate these invariant lines.

In characteristic $p>m$, this proof is unchanged after two tame modifications.
First, $K_n(A)$ is smooth by \cite[Proposition~6.5]{FuLi2021}, and \eqref{eq:kummer-coherent-invariants} follows from \cite[Theorem~3.2.14]{ChatzistamatiouRulling2011} and exactness of $G$-invariants.
Second, the Reynolds projectors and the preceding calculation descend to $\ZZ[1/m!]$; moreover $n!/m$ is a unit in $\bk$.
Thus the same invariant lines and their generators survive modulo $p$, proving \eqref{eq:HG}.
\end{proof}

\subsection{Relative families and ordinary deformations}
\label{subsec:standard-relative-families}

The next lemma constructs relative symplectic forms for both series in families whose underlying surfaces have ordinary geometric generic fibre.
It also verifies \eqref{eq:ct2} in the stated characteristic ranges by identifying $h^2(T)$ with $h^{1,2}$.

\begin{lemma}\label{lem:ct2 for relative Kummer}
Let $n\ge2$ and let $Z$ be either a $K3$ surface or an abelian surface over $\bk$; in the second case suppose $p\nmid n+1$.
After a finite extension of $\bk$, there is a smooth projective family $\cZ\to B$ over a smooth connected finite-type $\bk$-scheme, having $Z$ among its closed fibres and with ordinary geometric generic fibre.
Write $\cX\to B$ for the associated relative Hilbert scheme $\cZ^{[n]}$ in the first case and for the relative generalised Kummer variety $\cK_n$ in the second.
It is smooth and proper of relative dimension $2n$ with $T_{\cX/B}\simeq\Omega^1_{\cX/B}$.
If $p\ge5$, or if $Z$ is a $K3$ surface and $(p,n)=(3,2)$, then $\cX\to B$ satisfies \eqref{eq:ct2}.
\end{lemma}

\begin{proof}
We first work over $\overline\bk$.
For an abelian surface, the equicharacteristic deformation theorem of Norman--Oort gives a polarized deformation of $Z$ with ordinary generic fibre \cite{NormanOort1980}.
For a $K3$ surface, choose a primitive ample line bundle: the ordinary locus is open and dense in every irreducible component of the corresponding polarized moduli space, including when $p$ divides the degree \cite[proof of Corollary~7.5]{Bragg2023}.
In either case, an algebraic chart of polarized moduli therefore contains a point representing $Z$ in the closure of the ordinary locus.
Choose an integral curve through this point meeting that locus and pull back the universal family to its normalization.
The resulting base is a smooth connected curve, and the family is smooth and projective, with ordinary geometric generic fibre.
This construction, together with the identification of the chosen fibre with $Z$, descends to a finite extension of $\bk$.
We replace $\bk$ by that extension and write $f\colon\cZ\to B$ for the resulting family; in the abelian case it is an abelian scheme.

The relative Hilbert scheme $\cZ^{[n]}\to B$ of a smooth proper family of surfaces is smooth and proper of relative dimension $2n$.
In the abelian case, put $m=n+1$; since $p\nmid m$ the relative summation morphism $\Sigma\colon\cZ^{[m]}\to\cZ$ is smooth, by the base change along $[m]\colon\cZ\to\cZ$ recalled in \Cref{lem:app-lift}, so $\cK_n=\Sigma^{-1}(0)\to B$ is smooth and proper of relative dimension $2n$.

The sheaf $f_*\omega_{\cZ/B}$ is invertible and its evaluation map $f^*f_*\omega_{\cZ/B}\to\omega_{\cZ/B}$ is an isomorphism, by cohomology and base change and the triviality of the canonical bundle of each surface fibre.
Shrinking $B$ around the chosen point, trivialize this line bundle and let $\eta$ be the resulting relative $2$-form on $\cZ/B$.
The usual Hilbert-scheme construction, followed in the abelian case by restriction to $\cK_n$, gives a relative $2$-form $\sigma$ on $\cX/B$; these constructions commute with base change (cf.\ \cite[Propositions~4.4 and~6.5]{FuLi2021}).

To check nondegeneracy, let $Y$ be a geometric fibre of $\cX/B$.
The canonical-bundle formula for Hilbert schemes of surfaces gives $\omega_Y\simeq\Oh_Y$ in the $K3$ case; in the Kummer case the same formula on the ambient Hilbert scheme and adjunction for the smooth summation map give this triviality.
On the locus of distinct points, $\sigma$ is the sum of the surface forms.
Its restriction to the zero-sum tangent space in the Kummer case is nondegenerate because its orthogonal complement is the diagonal and $m$ is invertible \cite[Lemma~6.6]{FuLi2021}.
Thus the determinant of contraction $\sigma^\flat\colon T_Y\to\Omega^1_Y$ is a nonzero section of $\omega_Y^{\otimes2}\simeq\Oh_Y$.
Since $Y$ is proper and geometrically connected, this section is nowhere vanishing.
Hence $\sigma^\flat$ is an isomorphism on every fibre, and therefore $T_{\cX/B}\simeq\Omega^1_{\cX/B}$.
In particular
\begin{equation}\label{eq:ct2-is-h12}
  \dim_{\kappa(b)}\HH^2\bigl(\cX_b,T_{\cX_b}\bigr)=h^{1,2}(\cX_b)
\end{equation}
for every closed point $b$, so \eqref{eq:ct2} for $\cX\to B$ is the local constancy of $b\mapsto h^{1,2}(\cX_b)$.

It remains to prove \eqref{eq:ct2}.
If $Z$ is a $K3$ surface and $(p,n)=(3,2)$, \Cref{prop:app-hilbert-tame} identifies the Hodge numbers of every fibre with those in characteristic zero; hence $h^{1,2}(\cX_b)=0$, proving \eqref{eq:ct2}.

For the remaining assertion assume $p\ge5$.
For a closed point $b$, put $k_b=\kappa(\bar{b})$ and $X=\cX_b\otimes_{\kappa(b)}k_b$.
Hodge numbers are unchanged by this field extension.
The surface $\cZ_b\otimes_{\kappa(b)}k_b$ admits a smooth projective lift over $\Wit=\Wit(k_b)$: in the abelian case by \Cref{lem:app-lift}, and in the $K3$ case by the Deligne--Ogus lifting theorem \cite[Theorem~2.9]{Liedtke2016}.
Taking the associated relative Hilbert scheme, resp.\ relative generalised Kummer variety, of that lift produces a smooth projective lift $\mathscr X/\Wit$ of $X$; write $K=\operatorname{Frac}(\Wit)$.
In the $K3$ case, \Cref{prop:app-hilbert} gives $\HH^3_{\dR}(X/k_b)=0$ for every $n\ge2$.
In the abelian case, $\HH^3_{\crys}(X/\Wit)$ and $\HH^4_{\crys}(X/\Wit)$ are torsion-free by \Cref{thm:app-torsion-free}\textup{(2)}, so the universal-coefficient sequence \eqref{eq:app-crys-uct} gives
\[
  \dim_{k_b}\HH^3_{\dR}(X/k_b)
   =b_3(\mathscr X_{\overline K})=8
\]
by \cite{GottscheSoergel1993}.
Since $\mathscr X$ supplies a $\Wit_2$-lift and $3<p$, Deligne--Illusie gives $E_1$-degeneration in total degree three \cite[Corollary~2.5]{DeligneIllusie1987}, so $\sum_{a+b'=3}h^{a,b'}(X)$ is $0$ in the $K3$ case and $8$ in the abelian case.
In the $K3$ case the four summands are non-negative with sum zero, so $h^{1,2}(X)=0$.
In the abelian case the same sum for the geometric generic fibre of $\mathscr X$ is $8$ as well, so upper semicontinuity of the individual $h^{a,b'}$ along $\mathscr X/\Wit$ forces $h^{a,b'}(X)=h^{a,b'}(\mathscr X_{\overline K})$ for $a+b'=3$, whence $h^{1,2}(X)=4$.
Either way $h^{1,2}(\cX_b)$ is independent of $b$, and \eqref{eq:ct2} follows from \eqref{eq:ct2-is-h12}.
\end{proof}

\subsection{Hodge-deformations and the height dichotomy}
\label{subsec:standard-hodge-deformations}

\begin{proof}[Proof of \Cref{thm:deformation type of hilbert and kummer}]
Put $d=2n$ and let $Z$ be the standard model $S^{[n]}$ or $K_n(A)$.
By \Cref{prop:hilbert-HG,prop:kummer-HG,lem:ct2 for relative Kummer}, $Z$ is Hodge-good with trivial canonical bundle and a symplectic form.
By \Cref{prop:app-hilbert-tame,thm:app-torsion-free,cor:app-numbers}, its Hodge numbers agree with those of its characteristic-zero model, its Hodge--de Rham spectral sequence degenerates, and its crystalline cohomology is torsion-free.
Thus $Z$ is primitive symplectic, with $h^{2,0}(Z)=1$, and \eqref{eq:app-crys-uct} gives
\begin{equation}\label{eq:standard-hodge-betti}
  \sum_{a+b=j}h^{a,b}(Z)=b_j(Z)\qquad\text{for every }j,
\end{equation}
where $b_j$ denotes the $\ell$-adic Betti number for $\ell\ne p$.

After extending $\bk$ to an algebraic closure, $Z$ also carries a perfect Frobenius-compatible crystalline Beauville--Bogomolov pairing satisfying \eqref{eq:integral-fujiki}, with $c=1$ in the Hilbert case and $c=n+1$ in the Kummer case.
Indeed, the standard models admit smooth projective Witt liftings whose generic-fibre Beauville--Bogomolov lattices have discriminants $2(n-1)$ and $2(n+1)$, respectively, and Fujiki constant $c(2n)!/(2^n n!)$ \cite[Introduction]{Rapagnetta2008}.
The discriminants are units in the stated characteristic ranges.
Torsion-freeness in degrees two and three, integral comparison \cite[Theorem~14.6(iii)]{BMS2018}, and the tensor construction of \cite[proof of Proposition~2.1.5]{Yang2023} therefore carry these self-dual lattices and pairings to crystalline cohomology, compatibly with cup products and with target Frobenius $p^2\Frob$.

Let $g\colon\cY\to B$ be a Hodge-deformation family containing a reference fibre $Y_0$ with these properties.
Constancy on closed points extends to all points by semicontinuity and the Jacobson property.
The fibres are geometrically integral and $h^{d,0}=1$, so $L=g_*\omega_{\cY/B}$ is a line bundle commuting with base change.
The evaluation map $g^*L\to\omega_{\cY/B}$ is nonzero on every fibre, so its zero divisor is flat over $B$ and has open and closed image.
Since it misses $Y_0$, it is empty; hence $\omega_{\cY/B}\simeq g^*L$ and every fibre has trivial canonical bundle.

For each geometric fibre $Y=\cY_{\bar t}$, smooth proper base change, the crystalline universal-coefficient sequence and the Hodge--de Rham spectral sequence give
\[
  b_j(Y)\le\dim_{\kappa(\bar t)}\HH^j_{\dR}(Y)
       \le\sum_{a+b=j}h^{a,b}(Y)=b_j(Y).
\]
Thus $Y$ has $E_1$-degeneration and torsion-free crystalline cohomology in every degree.
Over an algebraic closure of $\bk$, apply \Cref{lem:fujiki-propagation} on the component of the base containing the reference fibre; Galois conjugacy gives the same conclusions on the remaining components.
Every geometric fibre inherits the pairing and its Fujiki identity, and generators $\sigma\in\HH^0(Y,\Omega_Y^2)$ and $\eta\in\HH^2(Y,\Oh_Y)$ satisfy $\sigma^n\ne0$ and $\eta^n\ne0$.
Since $\omega_Y$ is trivial, $\sigma^n$ is nowhere vanishing; invertibility of $n!$ makes $\sigma$ nondegenerate, and $E_1$-degeneration makes it closed.
Hodge constancy and the nonzero powers of $\eta$ give Hodge-goodness, so $Y$ is primitive symplectic.

Iterating along the chain, carrying the integral pairing forward at each common fibre, proves the symplectic and cohomological assertions.
The cohomological properties and triviality of the canonical bundle descend to the original perfect field, where \cite[Corollary~7.19]{BrantnerTaelman2025} gives unobstructed mixed-characteristic formal deformations and the stated liftings.
The Artin--Mazur assertion and the two equivalences follow from \Cref{prop:AM-hodge-good}\textup{(1)} and \Cref{cor:HG-qF-dichotomy}.

Finally, if $p>2n$, \Cref{cor:relative-W2-from-one-fibre} and \Cref{prop:HG-locally-constant}\textup{(ii)}, starting from the standard fibre, successively make each family in a chain satisfying \eqref{eq:ct2} a Hodge-deformation family.
The preceding argument then applies.
\end{proof}

The preceding argument also gives the following refinement of \cite[Lemma~2.3.5]{Yang2023} for smooth bases.

\begin{corollary}\label{cor:k3n-hodge-deformation}
Assume that $\bk$ is algebraically closed, $n\ge2$, and $p>n+1$.
Let $(g\colon\cY\to B,\xi)$ be a primitively polarized Hodge-deformation family of relative dimension $2n$.
If one geometric fibre is of $\mathrm{K3}^{[n]}$-type in the sense of \cite[Definition~1]{Yang2023}, then every geometric fibre is of $\mathrm{K3}^{[n]}$-type in that sense.
If $p>2n$, the hypothesis that the Hodge numbers are constant may be replaced by \eqref{eq:ct2}.
\end{corollary}

\begin{proof}
To check the fibrewise cohomological conditions below, we may extend $\bk$ so that the reference fibre is over a $\bk$-point.
By definition, it admits a smooth projective mixed-characteristic lifting with unchanged Hodge numbers, so the crystalline universal-coefficient sequence gives $E_1$-degeneration and torsion-free crystalline cohomology on the reference fibre.
Its canonical bundle is trivial.
Under the alternative hypotheses $p>2n$ and \eqref{eq:ct2}, \Cref{lem:relative-W2-lifting} followed by \Cref{prop:HG-locally-constant}\textup{(ii)} therefore shows that $g$ is a Hodge-deformation family.
The canonical-bundle and cohomology arguments in the proof of \Cref{thm:deformation type of hilbert and kummer}, applied to this lifting and then to $g$, show that every geometric fibre $Y$ has trivial canonical bundle, $E_1$-degeneration, and torsion-free crystalline cohomology.
The reference fibre carries a perfect Frobenius-compatible crystalline Beauville--Bogomolov pairing satisfying \eqref{eq:integral-fujiki} with $c=1$ \cite[Proposition~2.1.5 and Remark~2.1.8]{Yang2023}.
By \Cref{lem:fujiki-propagation}, every $Y$ carries such a pairing $q$, a symplectic form $\sigma$, and a class $\eta\in\HH^2(Y,\Oh_Y)$ with $\eta^n\ne0$.

It remains to verify the perfect cup-product pairing required in \cite[Proposition~2.2.3]{Yang2023}.
Put $D=\HH^2_{\dR}(Y)$ and lift $\eta$ to $\widetilde\eta\in D$.
As in the proof of \Cref{lem:fujiki-propagation}, $\Fil^1D=(\Fil^2D)^\perp$, so $q$ induces a perfect pairing on $\Fil^1D/\Fil^2D=\HH^1(Y,\Omega_Y^1)$, and $q(\sigma,\widetilde\eta)\ne0$.
For $a,b\in\Fil^1D$, \eqref{eq:integral-fujiki} gives
\[
 \operatorname{tr}(ab\sigma^{n-1}\widetilde\eta^{\,n-1})
   =(n-1)!\,q(a,b)q(\sigma,\widetilde\eta)^{n-1}.
\]
The right-hand side is a perfect pairing on $\HH^1(Y,\Omega_Y^1)$.
Under contraction with $\sigma$, the left-hand side identifies, up to the unit $\pm1/n$, with the natural pairing
\[
 \HH^1(Y,T_Y)\times\HH^1(Y,\Omega_Y^1)\longrightarrow\HH^2(Y,\Oh_Y),
\]
followed by the isomorphism $u\mapsto\operatorname{tr}(\sigma^n\eta^{n-1}u)$; compare \cite[proof of Corollary~2.1.7]{Yang2023}.
Thus this natural pairing is perfect.
All geometric fibres therefore satisfy the hypotheses of \cite[Proposition~2.2.3]{Yang2023}, and \cite[Lemma~2.3.5]{Yang2023} proves the assertion.
\end{proof}

\subsection{The surface case}
\label{subsec:k3-examples}

The dimension hypothesis $2n\ge4$ in \Cref{thm:main-all-good} is essential.
For a $K3$ surface one has $\chi(X,\Oh_X)=2$, which coincides with the finite non-ordinary value of $E_0$ in \Cref{prop:HK-Euler-slopes}, so the Euler-characteristic contradiction disappears.
$\hgt(X)$ is then the height of the formal Brauer group and takes every value in $\{1,\dots,10\}\cup\{\infty\}$, the finite heights greater than one occurring exactly for the non-ordinary, non-supersingular $K3$ surfaces.
The collapse of the height to $\{1,\infty\}$ is therefore a higher-dimensional phenomenon, starting in dimension four and driven by $\chi(X,\Oh_X)=n+1\ge3$.

\appendix
\section{Crystalline cohomology and Beauville--Bogomolov--Fujiki form}
\label{sec:appendix}

We record torsion-freeness results for generalised Kummer varieties and Hilbert schemes of $K3$ surfaces, independently of the rest of the paper.
In \Cref{subsec:fujiki-propagation} we establish the integral Fujiki propagation lemma used in \Cref{thm:deformation type of hilbert and kummer}.
For the Kummer series, let $A$ be an abelian surface over a perfect field $\bk$ of characteristic $p>0$, put $m=n+1$, let $\Sigma\colon A^{[m]}\to A$ be the Hilbert--Chow morphism followed by addition, and set $X=K_n(A)=\Sigma^{-1}(0)$.
If $p\nmid m$, translation trivialises $\Sigma$ after the finite \'etale base change $[m]\colon A\to A$, so $X$ is smooth of dimension $2n$.

\begin{theorem}
\label{thm:app-torsion-free}
Let $A$ be an abelian surface over a perfect field $\bk$ of characteristic $p$, and let $X=K_n(A)$ with $n\ge1$.
\begin{enumerate}[label=\textup{(\arabic*)},leftmargin=2.2em]
\item If $p>n+1$, then the Hodge--de Rham spectral sequence of $X$ degenerates at $E_1$ in every degree and $\HH^j_{\crys}(X/\Wit(\bk))$ is torsion-free for every $j$.
\item If $p\ge5$ and $p\nmid n+1$, then $\HH^3_{\crys}(X/\Wit(\bk))$ and $\HH^4_{\crys}(X/\Wit(\bk))$ are torsion-free.
\end{enumerate}
\end{theorem}

Neither part contains the other: \textup{(1)} is unrestricted in the degree but needs $p$ large compared with $n$, whereas \textup{(2)} allows any $p\ge5$ prime to $n+1$, however large $n$ is, at the cost of treating only two degrees.
Both are proved after extending $\bk$ to an algebraic closure, which is harmless because crystalline cohomology commutes with base change along a faithfully flat extension of Witt rings.

\subsection{All degrees in the tame range}

Part \textup{(1)} comes from a motivic decomposition of $X$ into abelian varieties which is integral away from $m!$.

\begin{lemma}[Tame Kummer correspondences]
\label{lem:app-tame}
Let $\bk$ be algebraically closed with $p>m\ge2$.
For a partition $\lambda=(\lambda_1,\dots,\lambda_\ell)$ of $m$ put
\[
  B_\lambda=\ker\Bigl(A^\ell\to A,\ (x_i)\mapsto\textstyle\sum_i\lambda_ix_i\Bigr),
  \qquad c_\lambda=m-\ell,\qquad a_\lambda=(-1)^{c_\lambda}\prod_i\lambda_i ,
\]
and let $G_\lambda$ permute the coordinates belonging to equal parts.
Then $B_\lambda$ is a disjoint union of translates of an abelian variety of dimension $2\ell-2$, and the reduced incidence cycles $\Gamma_\lambda\subset B_\lambda\times X$, defined by $\operatorname{HC}(\xi)=\sum_i\lambda_i[x_i]$, satisfy ${}^t\Gamma_\lambda\circ\Gamma_\mu=0$ for $\lambda\ne\mu$, ${}^t\Gamma_\lambda\circ\Gamma_\lambda =a_\lambda\sum_{g\in G_\lambda}[\operatorname{Graph}(g)]$, and
\begin{equation}\label{eq:app-diagonal}
  [\Delta_X]=\sum_{\lambda\vdash m}
    \frac{1}{a_\lambda|G_\lambda|}\,\Gamma_\lambda\circ{}^t\Gamma_\lambda
\end{equation}
as Chow correspondences with $\ZZ_{(p)}$-coefficients.
\end{lemma}

\begin{proof}
If $e=\gcd(\lambda_i)$, an integral change of coordinates on $A^\ell$ turns the weighted sum into $(y_i)\mapsto[e]y_1$, so $B_\lambda\simeq A[e]\times A^{\ell-1}$; as $e\mid m$ and $p>m$, this is smooth of the stated dimension.
To restrict the correspondence calculus to $X$, write $s_\lambda\colon A^\ell\to A$ for the weighted sum and $\widetilde\Gamma_\lambda\subset A^\ell\times_A A^{[m]}$ for the ambient incidence correspondence.
Both $s_\lambda$ and $\Sigma$ are smooth.
Simultaneous translation by $a\in A$ adds $ma$ to their values, so base change along $[m]\colon A\to A$ gives
\[
 A^{[m]}\times_{A,[m]}A\simeq X\times A,
 \qquad
 A^\ell\times_{A,[m]}A\simeq B_\lambda\times A,
\]
and identifies the pulled-back incidence correspondence with $\Gamma_\lambda\times A$.

The Hilbert--Chow calculations can therefore be made relatively over $A$.
Indeed, the two equalities of addition values in a triple product impose independent base conditions, so the absolute composition calculations are the pushforwards of the refined products in the smooth fibre products over $A$.
The dimension estimates and intersection multiplicities of \cite[\S4.3, Lemma~5.1.2 and Propositions~5.1.3--5.1.4]{deCataldoMigliorini2002} give the two transpose-composition identities, with multiplicity $a_\lambda$ on each graph.
The diagonal identity of \cite[Proposition~6.1.5]{deCataldoMigliorini2002} is an equality of cycles supported on $A^{[m]}\times_A A^{[m]}$.
These relative products commute with the flat base change $[m]$ and, under the displayed product identifications, with refined restriction to the fibre over $0$.
Thus they give the asserted identities on $B_\lambda$ and $X$.

If $m_j$ is the multiplicity of $j$ in $\lambda$, then $|a_\lambda||G_\lambda|=\prod_j j^{m_j}m_j!$ divides $m!$.
The dimension calculations and the cycle identity hold with these denominators already before passing to Chow classes; since the group of cycles is free, they hold with $\ZZ_{(p)}$-coefficients.
\end{proof}

\begin{proof}[Proof of \Cref{thm:app-torsion-free}\textup{(1)}]
Integral cycle classes and proper Gysin maps act on Hodge, de Rham and crystalline cohomology, torsion included, and are compatible with composition of correspondences: for Hodge cohomology this is \cite[Theorem~3.1.8]{ChatzistamatiouRulling2011}, for the de Rham and crystalline realisations one uses the cycle classes and trace maps of \cite[II]{Gros1985} together with the comparison of the de Rham--Witt complex with crystalline cohomology \cite[II, Theorem~1.4 and Scholium~2.8]{Illusie1979}, packaged for de Rham--Witt cohomology in \cite[Theorem~3.4.6]{ChatzistamatiouRulling2012}.
Applying these to \Cref{lem:app-tame} gives, with $c_\lambda$ as there,
\begin{equation}\label{eq:app-kummer-hodge}
  \HH^q(X,\Omega^r_X)\simeq\bigoplus_{\lambda\vdash m}
    \HH^{q-c_\lambda}\bigl(B_\lambda,\Omega^{r-c_\lambda}_{B_\lambda}\bigr)^{G_\lambda},
  \qquad
  \HH^j_{\dR}(X/\bk)\simeq\bigoplus_{\lambda\vdash m}
    \HH^{j-2c_\lambda}_{\dR}(B_\lambda/\bk)^{G_\lambda}.
\end{equation}
Each $B_\lambda$ is a union of translates of abelian varieties, whose Hodge--de Rham spectral sequence degenerates in every characteristic, and taking $G_\lambda$-invariants is exact because $p\nmid|G_\lambda|$.
Hence $\dim_\bk\HH^j_{\dR}(X/\bk)=\sum_{r+q=j}h^q(X,\Omega^r_X)$ for every $j$.
In a bounded spectral sequence of finite-dimensional vector spaces a nonzero differential drops the dimension of its source and target diagonals, so equality on every diagonal forces all differentials to vanish.

For the crystalline assertion, put $M_j=\HH^j_{\crys}(X/\Wit)$ and $N_j=\bigoplus_\lambda\HH^{j-2c_\lambda}_{\crys}(B_\lambda/\Wit)$.
The correspondences give $\Wit$-linear maps $v_j\colon M_j\to N_j$, $x\mapsto\bigl(({}^t\Gamma_\lambda)_*x/a_\lambda|G_\lambda|\bigr)_\lambda$, and $u_j\colon N_j\to M_j$, $(y_\lambda)\mapsto\sum_\lambda(\Gamma_\lambda)_*y_\lambda$; all denominators are units of $\Wit$, and \eqref{eq:app-diagonal} gives $u_jv_j=\mathrm{id}_{M_j}$ on all of $M_j$, torsion included.
The crystalline cohomology of an abelian variety is the exterior algebra on its finite free $\HH^1_{\crys}$ \cite[II, (7.1.1)]{Illusie1979}, so $N_j$ is finite free and $M_j$ is a direct summand of it.
\end{proof}

\begin{proposition}[Tame Hilbert schemes in all degrees]
\label{prop:app-hilbert-tame}
Let $S$ be a $K3$ surface over a perfect field $\bk$ of characteristic $p$, and let $n\ge2$.
If $p>n$, then the Hodge--de Rham spectral sequence of $S^{[n]}$ degenerates at $E_1$ and $\HH^j_{\crys}(S^{[n]}/\Wit(\bk))$ is torsion-free for every $j$.
Consequently, the Hodge numbers of $S^{[n]}$ agree with those of a complex variety of $\mathrm{K3}^{[n]}$-type.
\end{proposition}

\begin{proof}
We may extend $\bk$ to an algebraic closure.
For a partition $\lambda=(\lambda_1,\ldots,\lambda_\ell)$ of $n$, put $c_\lambda=n-\ell$ and $a_\lambda=(-1)^{c_\lambda}\prod_i\lambda_i$, and let $G_\lambda$ permute the coordinates with equal parts.
Let $\Gamma_\lambda\subset S^\ell\times S^{[n]}$ be the reduced incidence correspondence defined by $\operatorname{HC}(\xi)=\sum_i\lambda_i[x_i]$.
The Hilbert--Chow correspondence identities give
\[
  [\Delta_{S^{[n]}}]=\sum_{\lambda\vdash n}
  \frac{\Gamma_\lambda\circ{}^t\Gamma_\lambda}
       {a_\lambda|G_\lambda|}.
\]
Here we use the cycle identity of \cite[\S4.3, Propositions~5.1.3--5.1.4 and~6.1.5] {deCataldoMigliorini2002}, before passing to Chow classes.
If $m_j$ is the multiplicity of $j$ in $\lambda$, then $|a_\lambda||G_\lambda|=\prod_j j^{m_j}m_j!$, the order of the centralizer of a permutation of cycle type $\lambda$.
It divides $n!$.
Thus all coefficients in the displayed identity lie in $\ZZ_{(p)}$, and the identity holds with those coefficients: the group of cycles is free on the irreducible subvarieties, so no torsion is lost here.
The transpose compositions for distinct partitions vanish by the dimension calculation in Proposition~5.1.3 of the same source, while Proposition~5.1.4 gives, on the ordered products,
\[
  {}^t\Gamma_\lambda\circ\Gamma_\lambda
    =a_\lambda\sum_{\gamma\in G_\lambda}[\operatorname{Graph}(\gamma)].
\]
These identities likewise hold with $\ZZ_{(p)}$-coefficients.

Apply the Hodge, de Rham and crystalline realization of correspondences used in the proof of \Cref{thm:app-torsion-free}\textup{(1)}.
The resulting decompositions are
\begin{align*}
  \HH^q(S^{[n]},\Omega^r)
   &\simeq\bigoplus_{\lambda\vdash n}
     \HH^{q-c_\lambda}(S^\ell,\Omega^{r-c_\lambda})^{G_\lambda},\\
  \HH^j_{\dR}(S^{[n]})
   &\simeq\bigoplus_{\lambda\vdash n}
     \HH^{j-2c_\lambda}_{\dR}(S^\ell)^{G_\lambda}.
\end{align*}
The Hodge--de Rham spectral sequence of a $K3$ surface degenerates and its crystalline cohomology is free in every degree \cite[Proposition~2.5]{Liedtke2016}.
The same properties hold for its products by the K\"unneth formulas.
Since $|G_\lambda|$ is prime to $p$, taking invariants is exact.
The two displayed decompositions therefore give $E_1$-degeneration for $S^{[n]}$.
The diagonal identity also makes its crystalline cohomology a direct summand of $\bigoplus_\lambda\HH^{j-2c_\lambda}_{\crys}(S^\ell/\Wit)$, proving torsion-freeness.

Since $p>n\ge2$, the characteristic is odd, and $S$ has a projective lift over $\Wit$ by \cite[Theorem~2.9]{Liedtke2016}.
Its relative Hilbert scheme lifts $S^{[n]}$.
Degeneration, torsion-freeness and \eqref{eq:app-crys-uct} show that the sum of the Hodge numbers on each diagonal is the corresponding Betti number on both fibres.
Upper semicontinuity then makes every Hodge number equal on the special and geometric generic fibres, proving the last assertion.
\end{proof}

\subsection{Degrees three and four for \texorpdfstring{$p\ge5$}{p>=5}}

The engine for part \textup{(2)} is the following criterion, which is independent of the geometry.

\begin{proposition}[Fontaine--Messing criterion in adjacent degrees]
\label{prop:app-adjacent}
Let $\mathscr Y/\Wit(\bk)$ be smooth and proper with special fibre $Y$, and let $i\le p-2$.
If $\HH^i_{\et}(\mathscr Y_{\overline K},\ZZ_p)$ and $\HH^{i+1}_{\et}(\mathscr Y_{\overline K},\ZZ_p)$ are torsion-free, then so are $\HH^i_{\crys}(Y/\Wit(\bk))$ and $\HH^{i+1}_{\crys}(Y/\Wit(\bk))$.
\end{proposition}

\begin{proof}
Freeness of the two integral \'etale groups makes the coefficient sequence $0\to\HH^i_{\et}(\ZZ_p)/p\to\HH^i_{\et}(\Fp)\to\HH^{i+1}_{\et}(\ZZ_p)[p]\to0$ give $\dim_{\Fp}\HH^i_{\et}(\mathscr Y_{\overline K},\Fp)=b_i$.
Fontaine--Messing comparison applies in degrees $0\le r\le p-2$ and preserves invariant factors \cite[Theorem~0.3]{Min2021}, whence $\dim_\bk\HH^i_{\dR}(Y/\bk)=b_i$.
Derived crystalline base change gives
\begin{equation}\label{eq:app-crys-uct}
  0\to\HH^i_{\crys}(Y/\Wit)/p\to\HH^i_{\dR}(Y/\bk)
   \to\HH^{i+1}_{\crys}(Y/\Wit)[p]\to0 ,
\end{equation}
so with $t_j=\dim_\bk\HH^j_{\crys}(Y/\Wit)[p]$ one gets $b_i=b_i+t_i+t_{i+1}$, that is $t_i=t_{i+1}=0$.
A finitely generated $\Wit$-module without $p$-torsion is torsion-free.
\end{proof}

\begin{remark}
\label{rem:app-adjacent-LiLiu}
The conclusion of \Cref{prop:app-adjacent} also follows from the torsion crystalline--\'etale comparison of Li--Liu \cite[Theorem~1.2 and Corollary~7.28]{LiLiu2025Comparison}.
For the unramified base $\Wit(\bk)$, their condition $ei<p-1$ is exactly $i\le p-2$; the comparison gives $\dim_\bk\HH^i_{\dR}(Y/\bk)=\dim_{\Fp}\HH^i_{\et}(\mathscr Y_{\overline K},\Fp)$, and \eqref{eq:app-crys-uct} then yields torsion-freeness in degrees $i$ and $i+1$.
\end{remark}

Two inputs remain: a lift of $X$ over $\Wit$, and freeness of the integral \'etale cohomology of a complex generalised Kummer variety in degrees three and four.

\begin{lemma}
\label{lem:app-lift}
An abelian surface $A$ over an algebraically closed field of characteristic $p\ge3$ admits a projective lift to an abelian scheme over $\Wit=\Wit(\bk)$.
Consequently, if $p\nmid m$, then $X$ lifts to a smooth projective $\mathscr K\to\Spec\Wit$.
\end{lemma}

\begin{proof}
By \cite[Proposition~11.1]{Oort1987}, some polarization of $A$ lifts together with $A$ over $\Wit(\bk)$, giving a projective abelian scheme $\mathscr A/\Wit$.
For the consequence, $m$ is a unit in $\Wit$, so $[m]$ is finite \'etale on the lift $\mathscr A$ and translation identifies $\mathscr A\times_\Wit\mathscr K$ with $\mathscr A^{[m]}\times_{\mathscr A,[m]}\mathscr A$, making the summation $\Sigma$ smooth by \'etale descent.
\end{proof}

\begin{proposition}
\label{prop:app-integral}
Let $B$ be a complex abelian surface and $X_{\CC}=K_n(B)$.
Then $m\cdot\HH^3(X_{\CC},\ZZ)_{\mathrm{tors}}=0$ and $m\cdot\HH^4(X_{\CC},\ZZ)_{\mathrm{tors}}=0$.
\end{proposition}

\begin{proof}
Degree three is \cite[Corollary~2.2]{HartliebVerni2025}, via a dominant rational map of degree $m$ from $B^{[m-1]}$ and the torsion-freeness of the cohomology of the Hilbert scheme.
Degree four is not birationally invariant, so we use instead the zero-sum incidence variety $Y=\{(\xi\subset\xi')\in B^{[n,n+1]}:\Sigma(\xi')=0\}$ and the forgetful morphism $f\colon Y\to X_{\CC}$, which is proper and generically finite of degree $m$.

Set $U=B^{[n]}\times B$ with $s(\xi,x)=\Sigma_n(\xi)+x$ and $U_0=s^{-1}(0)$, so that $U_0\simeq B^{[n]}$ and $Y\simeq\PP_{U_0}(\mathcal I_0)$ for the restriction $\mathcal I_0$ of the universal ideal.
Jiang's hypotheses hold for $\mathcal I_0$: simultaneous translation $\tau_a(\xi,x)=(t_a\xi,x+a)$ satisfies $s\circ\tau_a=s+ma$, so base change along $[m]$ makes the whole configuration a direct product with $B$, and the smoothness and codimension conditions verified for the ambient universal ideal in \cite[Lemma~5.3 and Corollary~5.4]{Jiang2023} descend along this finite \'etale cover.
(This check is what licenses the restriction to the fibre $U_0$; the decomposition is not quoted for it in \cite{Jiang2023}.)
Jiang's integral motive formula \cite[Corollary~4.3]{Jiang2023} then gives $h(Y)\simeq h(U_0)\oplus h(Z)(1)$ with $Z=\PP_{U_0}(\Ext^1(\mathcal I_0,\Oh_{U_0}))$, whence
\[
  \HH^4(Y,\ZZ)\simeq\HH^4(B^{[n]},\ZZ)\oplus\HH^2(Z,\ZZ)(-1).
\]
The first summand is torsion-free by \cite{Totaro2020}.
For the second, $Z$ is birational to $V=\{(\eta,x)\in B^{[n-1]}\times B:\Sigma_{n-1}(\eta)+2x=0\}$, and $V\to B^{[n-1]}$ is the pullback of $[2]$, hence finite \'etale.
If $n\ge3$, \cite[Corollary~1.3]{BiswasHogadi2015} identifies $\pi_1^{\et}(B^{[n-1]})$ with the abelianisation of $\pi_1^{\et}(B)$, hence with $\widehat{\ZZ}^{\,4}$; for $n=2$ the same description follows directly from $B^{[1]}=B$.
Each connected component $V_i$ of $V$ therefore has \'etale fundamental group an open subgroup of $\widehat{\ZZ}^{\,4}$, again isomorphic to $\widehat{\ZZ}^{\,4}$.
Comparison with the topological fundamental group shows that the profinite completion of $\HH_1(V_i,\ZZ)$ is $\widehat{\ZZ}^{\,4}$.
Since $\HH_1(V_i,\ZZ)$ is finitely generated, it has no torsion.
Birational smooth projective complex varieties have isomorphic topological fundamental groups, so $\HH_1(Z,\ZZ)$ is torsion-free as well.
For $n=1$, both $V$ and $Z$ are finite and this conclusion is immediate.
The universal-coefficient theorem now gives $\HH^2(Z,\ZZ)_{\mathrm{tors}}\simeq\Ext^1_\ZZ(\HH_1(Z,\ZZ),\ZZ)=0$.
Thus $\HH^4(Y,\ZZ)$ is torsion-free, so $f^*\alpha=0$ for torsion $\alpha$, and the projection formula gives $m\alpha=f_*f^*\alpha=0$.
\end{proof}

\begin{proof}[Proof of \Cref{thm:app-torsion-free}\textup{(2)}]
Let $\mathscr K/\Wit$ be the lift of \Cref{lem:app-lift}.
Embedding a field of definition of its generic fibre into $\CC$ and comparing with singular cohomology, \Cref{prop:app-integral} shows that the $p$-primary torsion of $\HH^3_{\et}(\mathscr K_{\overline K},\ZZ_p)$ and $\HH^4_{\et}(\mathscr K_{\overline K},\ZZ_p)$ is annihilated by $m$, hence vanishes since $p\nmid m$.
As $3\le p-2$ exactly when $p\ge5$, \Cref{prop:app-adjacent} with $i=3$ applies.
\end{proof}

The Hilbert-scheme series needs no separate integral input, because Totaro's theorem already supplies torsion-freeness in every degree.

\begin{proposition}
\label{prop:app-hilbert}
Let $S$ be a $K3$ surface over a perfect field $\bk$ of characteristic $p\ge5$, put $\Wit=\Wit(\bk)$, and let $n\ge2$.
Then $\HH^3_{\crys}(S^{[n]}/\Wit)$ and $\HH^4_{\crys}(S^{[n]}/\Wit)$ are torsion-free, and $\dim_\bk\HH^3_{\dR}(S^{[n]}/\bk)=0$.
\end{proposition}

\begin{proof}
Crystalline and de Rham cohomology commute with extension of the perfect ground field, and the corresponding extension of Witt rings is faithfully flat, so we may assume that $\bk$ is algebraically closed.
Since $p$ is odd, \cite[Theorem~2.9]{Liedtke2016} supplies a smooth projective lift $\mathscr S/\Wit$.
Its relative Hilbert scheme is a smooth projective lift of $S^{[n]}$.
The integral cohomology of the Hilbert scheme of points on a complex surface whose own integral cohomology is torsion-free is again torsion-free \cite{Totaro2020}, and that of a $K3$ surface is torsion-free; so every $\HH^j_{\et}$ of the geometric generic fibre with $\ZZ_p$-coefficients is torsion-free, and the degree-by-degree work of \Cref{prop:app-integral} is not needed here.
\Cref{prop:app-adjacent} with $i=3$, available because $3\le p-2$, gives the two torsion-free crystalline groups.
The odd Betti numbers of a variety of $\mathrm{K3}^{[n]}$-type vanish \cite{GottscheSoergel1993}, so \eqref{eq:app-crys-uct} reads $\dim_\bk\HH^3_{\dR}(S^{[n]}/\bk)=b_3=0$.
\end{proof}

\begin{remark}
\label{rem:app-char3}
For generalised Kummer varieties with $p=3$ and $n\ge2$, neither part of \Cref{thm:app-torsion-free} applies.
Degree three lies outside the Fontaine--Messing range $r\le p-2$ used in \Cref{prop:app-adjacent}, so this argument does not establish $\dim_\bk\HH^3_{\dR}(X/\bk)=b_3$ in characteristic three.
In contrast, \Cref{prop:app-hilbert-tame} applies to Hilbert squares of $K3$ surfaces in characteristic three.
\end{remark}

\begin{corollary}
\label{cor:app-numbers}
Let $n\ge2$ and $p>n+1$.
Then $h^{a,b}(X)=h^{a,b}(X_{\CC})$ for all $a,b$; in particular $\dim_\bk\HH^3_{\dR}(X/\bk)=8$ and $\dim_\bk\HH^2(X,T_X)=4$.
The equality $\dim_\bk\HH^3_{\dR}(X/\bk)=8$ also holds under the hypotheses of \Cref{thm:app-torsion-free}\textup{(2)}.
\end{corollary}

\begin{proof}
We may extend $\bk$ to an algebraic closure and choose a projective lift $\mathscr A/\Wit$ by \Cref{lem:app-lift}.
For each partition $\lambda$, the weighted-sum kernel $\mathscr B_\lambda\subset\mathscr A^\ell$ is a smooth proper lift of $B_\lambda$, with the same coordinate-permutation action of $G_\lambda$.
Its Hodge cohomology is finite free over $\Wit$ and commutes with base change: as in \Cref{lem:app-tame}, it is a disjoint union of copies of $\mathscr A^{\ell-1}$, whose Hodge cohomology is an exterior algebra on finite free modules.
Since $|G_\lambda|$ is a unit, averaging defines an idempotent on each Hodge cohomology group; its image is finite free, commutes with base change, and has the same rank on the special and generic fibres.
Applying \eqref{eq:app-kummer-hodge} and the corresponding decomposition in characteristic zero therefore gives equality of all Hodge numbers.
The value $b_3=8$ for a complex generalised Kummer variety of dimension $2n\ge4$ is the computation of G\"ottsche--Soergel \cite{GottscheSoergel1993}, so $\dim_\bk\HH^3_{\dR}=8$ by \Cref{thm:app-torsion-free}\textup{(1)}, and $\dim_\bk\HH^2(X,\Omega^1_X)=h^{1,2}=4$.
The $2$-form induced by a nonzero translation-invariant $2$-form on $A$ is symplectic on $X$ because $p\nmid m$ \cite[Proposition~6.5 and Lemma~6.6]{FuLi2021}.
Thus $T_X\simeq\Omega^1_X$, giving $\dim_\bk\HH^2(X,T_X)=4$.
Under the hypotheses of \textup{(2)}, torsion-freeness in degrees three and four makes $\dim_\bk\HH^3_{\dR}(X/\bk)=\rank_\Wit\HH^3_{\crys}(X/\Wit)=b_3$ by the universal-coefficient sequence used in \Cref{prop:app-adjacent}.
\end{proof}

\subsection{Constancy of integral Beauville--Bogomolov--Fujiki form}
\label{subsec:fujiki-propagation}

We prove the constancy lemma used in \Cref{thm:deformation type of hilbert and kummer}.
Its purpose is to control integral cup products throughout a smooth proper family.

\begin{lemma}\label{lem:fujiki-propagation}
Let $\bk$ be algebraically closed of characteristic $p>2$, let $n\ge2$, and let $g\colon\cY\to B$ be smooth proper of relative dimension $2n$, with geometrically connected fibres and $B$ smooth connected of finite type.
Suppose the crystalline cohomology of every geometric fibre is torsion-free in every degree.
Let $c\in\ZZ_{(p)}^\times$.
Assume that a fibre $Y_0$ over a point $b_0\in B(\bk)$ carries a perfect symmetric pairing
\[
 q_0\colon\HH^2_{\crys}(Y_0/\Wit(\bk))^{\otimes2}\longrightarrow\Wit(\bk)
\]
such that $q_0(Fx,Fy)=p^2\Frob(q_0(x,y))$ and
\begin{equation}\label{eq:integral-fujiki}
 \operatorname{tr}(x_1\cdots x_{2n})
   =c\sum_{\mathcal P}\prod_{\{i,j\}\in\mathcal P}q_0(x_i,x_j),
\end{equation}
where $\mathcal P$ runs through partitions of $\{1,\ldots,2n\}$ into unordered pairs.
Then every geometric fibre carries such a perfect pairing, with the same $c$.

If moreover $n<p$ and a geometric fibre $Y$ has $E_1$-degeneration and $h^{2,0}(Y)=h^{0,2}(Y)=1$, then
\[
 \sigma^n\ne0\quad\text{and}\quad\eta^n\ne0
\]
for all nonzero $\sigma\in\HH^0(Y,\Omega_Y^2)$ and $\eta\in\HH^2(Y,\Oh_Y)$.
In particular, $\omega_Y\simeq\Oh_Y$ then implies that $\sigma$ is symplectic.
\end{lemma}

\begin{proof}
Put $\mathcal H^j=R^jg_{\crys*}\Oh$.
Their rationalisations are convergent $F$-isocrystals of constant ranks $b_j$, compatible with base change \cite[Proposition~3.2 and Corollary~6.2]{MorrowCrystallineDirectImages}.
Choose a smooth formal Witt lifting $\operatorname{Spf}A$ of an affine open $U\subseteq B$.
The complex $C=R\Gamma_{\crys}(\cY_U/A)$ is bounded coherent, hence perfect since $A$ is regular, and satisfies derived base change \cite[\S2, proof of Lemma~2.2]{MorrowCrystallineDirectImages}.
For a closed point $b\in U$, with corresponding maximal ideal $\mathfrak m\subset A$, choose a minimal finite free complex $P^\bullet$ representing $C_{\mathfrak m}$, so that its differentials vanish modulo $\mathfrak m$.
Derived base change and crystalline torsion-freeness in degrees $j$ and $j+1$ give
\[
 \rank_{A_{\mathfrak m}}P^j
 =\dim_\bk\HH^j_{\dR}(Y_b)=b_j.
\]
Over $L=\operatorname{Frac}(A_{\mathfrak m})$ we therefore have
\[
 b_j=\dim_L\HH^j(P^\bullet\otimes L)
     =b_j-\rank_Ld^{j-1}-\rank_Ld^j.
\]
Thus every differential is zero over $L$, hence over $A_{\mathfrak m}$.
Since every maximal ideal of $A$ contains $p$, all $\HH^j(C)$ are finite locally free.
Derived base change now identifies the value of $\mathcal H^j$ on any divided-power thickening $T$ over $U$ with $\HH^j(C)\otimes_A\Oh_T$, using a local lifting $T\to\operatorname{Spf}A$.
These identifications are compatible with pullback, so the $\mathcal H^j$ are finite locally free crystals and commute with crystalline base change.

If $\mathcal H^2$ has rank zero, the first assertion is immediate and the additional Hodge hypotheses cannot hold.
We may therefore assume its rank is positive.
Put $K=\Wit(\bk)[1/p]$ and $V=\mathcal H^2[1/p](1)$.
The top cup product and trace define a symmetric tensor $w\colon V^{\otimes2n}\to\mathbf1$.
Work first in the $K$-linear category of underlying convergent isocrystals, with fibre functor at $Y_0$ \cite[Lemma~1.8]{Crew1992}.
The monodromy group preserves $w_0$.
The diagonal polynomial of $w_0$ is
\[
 P_0(v)=c\lambda_n q_0(v,v)^n,
 \qquad \lambda_n=\frac{(2n)!}{2^n n!}.
\]
Over characteristic zero its quadratic-root line $Kq_0$ is unique.
Consequently the monodromy group preserves this line and acts on it through $\mu_n$.
Tannakian duality produces a line subisocrystal $\mathcal L\subseteq\Sym^2V^\vee$ with $\mathcal L^{\otimes n}\simeq\mathbf1$.
Uniqueness of the quadratic-root line makes $\mathcal L$ Frobenius-stable; the displayed trivialisation is Frobenius-compatible.
It is therefore a unit-root $F$-isocrystal of finite order.
By the unit-root equivalence \cite[Theorem~1.3]{Crew1992}, a finite \'etale cover trivialises it.
On a connected component of this cover, choose the resulting horizontal Frobenius-compatible rational pairing $q$ to agree with $q_0$ at a point above $Y_0$.
Faithfulness of the fibre functor shows that $q$ is perfect as a rational isocrystal pairing and satisfies \eqref{eq:integral-fujiki}.

It remains to check the integral lattice; rational nondegeneracy alone would not suffice.
Locally choose $A$ as above with $A/pA$ integral and trivialise $\mathcal H^2_A$.
The coefficients of $q$ belong to $A[1/p]$ \cite[\S1.1, equations~(1.1.5)--(1.1.6)]{Crew1992}.
If $q$ were not integral, choose $m\ge1$ such that $Q=p^mq$ is integral and has a coefficient not divisible by $p$.
Since $p\ne2$, there is a vector $v$ for which $Q(v,v)$ is not divisible by $p$; basis vectors and their pairwise sums suffice.
The diagonal cup polynomial is integral, so
\[
 p^{mn}\operatorname{tr}(v^{2n})=c\lambda_n Q(v,v)^n.
\]
But $A/pA$ is a domain and
\[
 v_p(c\lambda_n)\le v_p((2n)!)<\frac{2n}{p-1}\le n,
\]
contradicting $mn\ge n$.
Hence $q$ is integral.
Its determinant is a unit after inverting $p$.
Since $(p)$ is prime, an integral element invertible in $A[1/p]$ is $p^a$ times a unit.
The exponent $a$ for $\det(q)$ is locally constant on the connected base and is zero at $Y_0$, where $q=q_0$ is perfect.
Thus $q$ is perfect everywhere.
Its restrictions give the asserted pairings on all geometric fibres; the finite \'etale cover does not change this fibrewise conclusion.

Now suppose $Y$ satisfies the additional Hodge hypotheses.
Write $M=\HH^2_{\crys}(Y/\Wit)$ and $D=M/pM=\HH^2_{\dR}(Y)$.
Mazur's description of the Hodge filtration \cite[Theorem~8.26]{BerthelotOgus1978} gives
\[
 \Fil^iD=\operatorname{im}\bigl(F^{-1}(p^iM)\longrightarrow M/pM\bigr).
\]
Thus $q(\Fil^2D,\Fil^1D)=0$: lifts $x,y$ of such classes satisfy $Fx\in p^2M$ and $Fy\in pM$, so Frobenius compatibility gives $p^2\Frob(q(x,y))=q(Fx,Fy)\in p^3\Wit$.
Since $q$ is perfect, $\Fil^2D$ is a line, and $D/\Fil^1D$ is a line, $q$ induces a perfect pairing between these two lines.

Let $\widetilde\eta\in D$ lift $\eta\in D/\Fil^1D=\HH^2(Y,\Oh_Y)$.
The class of $\sigma$ generates $\Fil^2D$, so $q(\sigma,\widetilde\eta)\ne0$ and $q(\sigma,\sigma)=0$.
In \eqref{eq:integral-fujiki} applied to $n$ copies of each class, only the $n!$ pairings matching every $\sigma$ with a $\widetilde\eta$ survive.
Therefore
\begin{equation}\label{eq:fujiki-mixed-top}
 \operatorname{tr}(\sigma^n\widetilde\eta^{\,n})
   =c\,n!\,q(\sigma,\widetilde\eta)^n\ne0.
\end{equation}
This proves $\sigma^n\ne0$.
If $\eta^n=0$, then $\widetilde\eta^{\,n}\in\Fil^1\HH^{2n}_{\dR}(Y)$, and its product with $\sigma^n\in\Fil^{2n}\HH^{2n}_{\dR}(Y)$ belongs to $\Fil^{2n+1}\HH^{4n}_{\dR}(Y)=0$, a contradiction.
Hence $\eta^n\ne0$.
Finally, $\sigma$ is closed by $E_1$-degeneration.
If $\omega_Y$ is trivial, the nonzero section $\sigma^n$ is nowhere vanishing; as $n!$ is invertible, this is equivalent to nondegeneracy of $\sigma$.
\end{proof}

\bibliographystyle{amsalpha}
\bibliography{references}

\end{document}